\documentclass[11pt]{amsart}
\usepackage[linkcolor=blue, citecolor = blue]{hyperref}
\hypersetup{colorlinks=true}
\usepackage{amsmath, amsfonts, amsthm, amssymb}
\usepackage{tikz}
\usepackage{ytableau}
\usepackage{caption}
\usepackage{subcaption}
\usepackage{graphicx} 
\usepackage[makeroom]{cancel}

\newtheorem*{theorem*}{Theorem}
\newtheorem{theorem}{Theorem}[section]
\newtheorem{lemma}[theorem]{Lemma}
\newtheorem{proposition}[theorem]{Proposition}
\newtheorem*{corollary*}{Corollary}
\newtheorem{corollary}[theorem]{Corollary}

\theoremstyle{definition}
\newtheorem{remark}[theorem]{Remark}
\newtheorem{definition}[theorem]{Definition}
\newtheorem{example}[theorem]{Example}

\usetikzlibrary{decorations.markings}
\usetikzlibrary{decorations.pathmorphing}
\usetikzlibrary{matrix,decorations.pathreplacing}

\pgfkeys{tikz/mymatrixenv/.style={decoration=brace,every left delimiter/.style={xshift=4pt},every right delimiter/.style={xshift=-4pt}}}
\pgfkeys{tikz/mymatrix/.style={matrix of math nodes,left delimiter=[,right delimiter={]},inner sep=1pt,row sep=0em,column sep=0em,nodes={inner sep=6pt}}}

\DeclareMathOperator*{\argmax}{arg\,max}

\title{Flagged Hamel--Goulden formulas}
\author{Alibek Adilzhan \and Damir Yeliussizov}
\address{Kazakh-British Technical University, Almaty, Kazakhstan}
\email{\href{mailto:alibek.adilzhanov@gmail.com}{alibek.adilzhanov@gmail.com}}
\email{\href{mailto:yeldamir@gmail.com}{yeldamir@gmail.com}}

\begin{document}
\tikzset{
    -->-/.style={decoration={
    markings,
    mark=at position .9 with {\arrow{>}}},postaction={decorate}},
    ->-/.style={decoration={
    markings,
    mark=at position .7 with {\arrow{>}}},postaction={decorate}}
}

\begin{abstract}
We obtain Hamel--Goulden-type ribbon decomposition determinantal formulas for 
flagged supersymmetric Schur functions.
As an application, we derive corresponding new determinantal formulas for dual refined canonical stable Grothendieck polynomials. These results generalize and produce a number of new determinantal formulas for these symmetric functions including Jacobi--Trudi and skew Giambelli-type determinants. 
\end{abstract}

\maketitle

\section{Introduction}
Determinantal formulas for Schur functions admit a remarkable unifying approach via ribbon decompositions due to Hamel and Goulden \cite{hg}, which generalizes Jacobi--Trudi and their dual N\"agelsbach--Kostka determinants, Giambelli and Lascoux--Pragacz formulas \cite{lasp}. 

In this paper, we develop new generalizations of Hamel--Goulden formulas for {\it flagged supersymmetric Schur functions} with two doubly infinite sets of variables.
As the main application, we obtain new determinantal formulas for {\it dual refined canonical stable Grothendieck polynomials}.

Flagged Schur functions are extensions of Schur functions arising in connection with Schubert polynomials, due to Lascoux and Sch\"utzenberger \cite{ls}, see also \cite{wachs, macdonald1, cll}. We consider an appropriate supersymmetric generalization of these functions. 

{Dual stable Grothendieck polynomials} 
can be viewed as a $K$-theoretic analogue of Schur functions  
due to Lam and Pylyavskyy \cite{lp}. 
One of their most general versions (with two sets of extra refined parameters)  
is known as 
the family of dual {\it refined canonical} stable Grothendieck polynomials which were introduced and studied by Hwang, Jang, Kim, Song, and Song in \cite{hwang1, hwang2}. These functions generalize refined versions by Galashin, Grinberg, and Liu \cite{ggl}, and canonical versions by the second author \cite{yel17}. Besides connections to $K$-theory, dual stable Grothendieck polynomials have interesting combinatorics related to plane partitions \cite{lp, ggl, yel21b}, and they also naturally arise in last passage percolation model in probability \cite{yel20, yel21a, ms}. 

Determinantal formulas for dual stable Grothendieck polynomials have received considerable attention, especially Jacobi--Trudi-type formulas, see \cite{ln, yel17, iwao1, kim1, ay22, kim2, hwang1, hwang2, ms} for some related works. Generalizing such formulas for Hamel--Goulden identities becomes more technical, as even Giambelli-type hook formulas obtained  by Lascoux and Naruse \cite{ln} were not obvious generalizations and had different presentations than Schur functions. 
In contrast, other various generalizations of Schur functions have straightforward analogues of ribbon decomposition formulas, see e.g. \cite{macdonald1}. 

It turns out that to obtain Hamel--Goulden-type formulas for dual stable Grothendieck polynomials it is useful to lift this problem for more general families of flagged supersymmetric functions, which are additionally indexed by skew shapes with shifted contents called {\it r-shapes}. 
The crucial technical 
part in deriving such formulas for flagged functions is to define {\it induced ribbon flags}  used in the enumeration determinants, along with appropriate tableaux formulas. 

In particular, our formulas also produce new Jacobi--Trudi and skew Giambelli type determinantal formulas. Interestingly, we obtain new formulas for dual stable Grothendieck polynomials even for Jacobi--Trudi case which differ from \cite{hwang2}. We also obtain skew Giambelli-type formulas which (in straight shape case) generalize the formula from \cite{ln}.  

\subsection{Structure of the paper} This paper is organized as follows. 
\begin{itemize}
    \item In \textsection\,\ref{sec:prelim} we give preliminary definitions, recall elementary supersymmetric functions and dual stable Grothendieck polynomials. 
    \item In \textsection\,\ref{sec:flag} we define column and row flagged supersymmetric Schur polynomials, show their duality and  define $g$-specialization which reduces these functions to dual refined canonical stable Grothendieck polynomials. 
    \item In \textsection\,\ref{sec:tableaux} we state the main tableaux formulas for flagged supersymmetric Schur polynomials via super tableaux. 
    \item In \textsection\,\ref{sec:ribbon} we define induced ribbon flags and state the main Hamel--Goulden-type formulas for flagged supersymmetric Schur polynomials along with specializations of these formulas for dual canonical refined stable Grothendieck polynomials. 
    \item In \textsection\,\ref{sec:proofs} we give proofs of the main results stated in \textsection\,\ref{sec:tableaux},\ref{sec:ribbon}. 
    \item In \textsection\,\ref{sec:special} we show some special cases of our Hamel--Goulden-type formulas; namely, we show new Jacobi--Trudi-type and skew Giambelli-type hook determinantal formulas.
    \item In \textsection\,\ref{sec:g} we discuss special cases of our formulas for dual refined canonical stable Grothendieck polynomials; in particular, we present new Jacobi--Trudi-type and skew Giambelli-type hook determinantal formulas. 
    \item In \textsection\,\ref{sec:concl} we conclude with some remarks and open questions. 
\end{itemize}

\section{Preliminaries}\label{sec:prelim}
\subsection{Variables notation}
We denote $[m] := \{1,\ldots, m \}$, $[0] := \varnothing$, and $[a, b] := \{a, a+1, \dots, b\}$ for integers $a \le b$. We shall use the following notation for variables. 
For a set of variables $\mathbf{x} = (x_i)$ we denote $\mathbf{x}_m = (x_i)_{i \in [m]} = (x_1,\ldots, x_m)$ and $\mathbf{x}_{a,b} = (x_i)_{i \in [a,b]} = (x_a, \ldots, x_b)$ if $a \le b$ (and $\mathbf{x}_{a,b} = \varnothing$ if $a > b$). 
We also denote $\overline{\mathbf{x}} = (-x_i)$ for the negation of variables. 

\subsection{Partitions and shapes} 
A {\it partition} is a sequence $\lambda = (\lambda_1, \ldots, \lambda_{\ell})$ of positive integers $\lambda_1 \ge \cdots \ge \lambda_{\ell}$, where $\ell = \ell(\lambda)$ is its {\it length}. The set $[\lambda] := \{(i,j) : 1 \le i \le \ell(\lambda), 1 \le j \le \lambda_i \}$ is the {\it Young diagram} (or shape) of $\lambda$. The {\it conjugate} partition $\lambda'$ of $\lambda$ is the partition with the transposed Young diagram. For partitions $\lambda$ and $\mu$ with $[\mu] \subseteq [\lambda]$, the {\it skew partition} $\lambda/\mu$ has the diagram (or shape) $[\lambda/\mu] := [\lambda] \setminus [\mu]$. We refer to the elements of diagrams as boxes or cells, and use the English notation for drawing them (like matrices). 

For a cell $\gamma = (i,j)$ of a diagram, its {\it content} is given by $c(\gamma) := j - i$ and we also denote $\mathrm{row}(\gamma) = i, \mathrm{col}(\gamma) = j$ for its row and column indices. 

\subsubsection{r-shapes} Notice that knowing the content of a single cell 
in the shape $\lambda/\mu$, completely determines contents 
of all other cells. We use this property of shapes and define 
a \textit{refined shape} (or \textit{r-shape} for short) 
as a pair $(\lambda/\mu, r)$ for $r \in \mathbb{Z}$. 
The r-shape $(\lambda/\mu, r)$ is a shape whose content of the bottom-left cell 
is $r$. (Often we will just write $\lambda/\mu$ for an r-shape.) Informally, r-shape is a shape with `shifted' contents. 
We call the number $r$ the \textit{root content} of 
the r-shape. 
Let $c_{\mathrm{min}}, c_{\mathrm{max}}$ be the ordinary contents ($\mathrm{col}(\alpha) - \mathrm{row}(\alpha)$) of the bottom-left and top-right cells respectively. 
Then the usual shape $\lambda/\mu$ is an r-shape 
$(\lambda/\mu, c_{\mathrm{min}})$. Note that for a cell $\gamma = (i,j)$ in the r-shape $(\lambda/\mu,\ r)$, its (shifted) content is $c(\gamma) = j - i + r - c_{\mathrm{min}}$. We define the conjugate r-shape of $(\lambda, r)$ as the r-shape $(\lambda', -r')$, where $r'$ is the content of the top-right cell of $\lambda$. 

\subsection{Elementary supersymmetric functions}
For sets of variables $\mathbf{x} = (x_i), \mathbf{y} = (y_i)$, we define 
the {\it elementary and complete homogeneous supersymmetric functions} $e_n(\mathbf{x} / \mathbf{y}), h_n(\mathbf{x} / \mathbf{y})$ ($n \ge 0$) as follows: 
$$
e_n(\mathbf{x} / \mathbf{y}) := \sum_{i = 0}^n (-1)^{n - i} e_i(\mathbf{x}) h_{n - i}(\mathbf{y}), \quad
h_n(\mathbf{x} / \mathbf{y}) := \sum_{i = 0}^n (-1)^{n - i} h_i(\mathbf{x}) e_{n - i}(\mathbf{y})
$$
via the usual elementary and complete homogeneous symmetric functions $e_k(\mathbf{x}), h_k(\mathbf{x})$, where $e_0(\mathbf{x}) = h_0(\mathbf{x}) = 1$, $e_k(\mathbf{x}) = h_k(\mathbf{x}) = 0$ for $k < 0$, and for $k \ge 1$ 
$$
e_k(\mathbf{x}) = \sum_{i_1 < \ldots < i_k} x_{i_1} \cdots x_{i_k},  \qquad 
h_{k}(\mathbf{x}) = \sum_{i_1 \le \ldots \le i_k} x_{i_1} \cdots x_{i_k}. 
$$ 
In addition, we set 
$e_n(\mathbf{x} / \mathbf{y}) = h_n(\mathbf{x} / \mathbf{y}) = 0$ for $n < 0$. 
Note that we have the following generating series 
$$
\sum_{n = 0}^{\infty} e_n(\mathbf{x} / \mathbf{y}) t^n = \prod_{i} \frac{1 + x_i t}{1 + y_i t}, \qquad 
\sum_{n = 0}^{\infty} h_n(\mathbf{x} / \mathbf{y}) t^n = \prod_{i} \frac{1 - y_i t}{1 - x_i t}
$$
and the following duality 
$$
e_n (\mathbf{x} / \mathbf{y}) = (-1)^n h_n(\mathbf{y} / \mathbf{x}) = h_n(\overline{\mathbf{y}} / \overline{\mathbf{x}}).
$$

\subsection{Dual stable Grothendieck polynomials}\label{sec:dualg}
For sets of variables $\mathbf{x} = (x_1, x_2, \ldots)$, $\boldsymbol{\alpha} = (\alpha_1, \alpha_2, \ldots)$ and $\boldsymbol{\beta} = (\beta_1, \beta_2, \ldots)$ the {\it dual refined canonical stable Grothendieck polynomial} $g_{\lambda/\mu}(\mathbf{x}; \boldsymbol{\alpha}; \boldsymbol{\beta})$ is a symmetric function (in the $\mathbf{x}$ variables) which can be defined as follows: 
$$
g_{\lambda/\mu}(\mathbf{x}; \boldsymbol{\alpha}, \boldsymbol{\beta}) := 
\det \left[e_{\lambda'_i - i - \mu'_j + j}(\mathbf{x}, \overline{\boldsymbol{\alpha}}_{j - 1}, \boldsymbol{\beta}_{\lambda'_i - 1} \,/\, \overline{\boldsymbol{\alpha}}_{i - 1}, \boldsymbol{\beta}_{\mu'_j}) \right]_{1 \le i, j \le n},
$$
where $n \ge \lambda_1$. 
These functions were introduced and studied in \cite{hwang1, hwang2} (the version written here differs by $\boldsymbol{\alpha} \to \overline{\boldsymbol{\alpha}}$). Note that for $\boldsymbol{\alpha} = 0, \boldsymbol{\beta} = 0 $ these functions specialize to the Schur functions $s_{\lambda/\mu}(\mathbf{x})$; in general, $g_{\lambda/\mu}(\mathbf{x}; \boldsymbol{\alpha}; \boldsymbol{\beta}) = s_{\lambda/\mu}(\mathbf{x}) + \text{lower degree terms}$. For $\boldsymbol{\alpha} = 0, \boldsymbol{\beta} = (1, 1, \ldots)$ they specialize to the dual stable Grothendieck polynomials $g_{\lambda/\mu}(\mathbf{x})$ introduced in \cite{lp};  for $\boldsymbol{\alpha} = 0$ they specialize to the refined version introduced in \cite{ggl}; and for $\boldsymbol{\alpha} = (\alpha, \alpha, \ldots), \boldsymbol{\beta} = (\beta, \beta, \ldots)$ they specialize to canonical version 
introduced in \cite{yel17}. 
Notably,  the following duality 
$$
\omega(g_{\lambda/\mu}(\mathbf{x}; \boldsymbol{\alpha}, \boldsymbol{\beta})) = g_{\lambda'/\mu'}(\mathbf{x}; \boldsymbol{\beta}, \boldsymbol{\alpha})
$$
holds for the action of the standard involutive ring automorphism $\omega$ of the ring of symmetric functions. 
The functions $g_{\lambda/\mu}(\mathbf{x}; \boldsymbol{\alpha}, \boldsymbol{\beta})$ also have combinatorial formula using marked reverse plane partitions \cite{hwang1, hwang2}; here we will use for them another new formula using super tableaux. 

\section{Flagged supersymmetric Schur polynomials}\label{sec:flag}
We use the known notions of column and row flags of Schur functions, see e.g. \cite{wachs, macdonald1}.

\begin{definition}[Flags] Let $\lambda/\mu$ be an r-shape and $n \ge \lambda_1,\,m \ge \ell(\lambda)$. 

(Column flags) 
The vectors $\mathbf{a} = (a_i), \mathbf{b} = (b_i) \in \mathbb{Z}^n$ satisfying the following conditions 
$$
a_{i}-a_{i+1} \le \mu'_i-\mu'_{i+1}+1, \quad b_i-b_{i+1} \le \lambda'_i-\lambda'_{i+1}+1, \quad 1 \le i < n, \text{ whenever $\mu'_i < \lambda'_{i+1}$},
$$
are called {\it column flags}. (We shall refer to them as flags.)

(Row flags) The vectors $\mathbf{a} = (a_i), \mathbf{b} = (b_i) \in \mathbb{Z}^m$  
satisfying the following conditions 
$$
a_{i}-a_{i+1}\le 0, \quad b_i-b_{i+1}\le 0, \quad 1 \le i < m, \text{ whenever $\mu_i < \lambda_{i+1}$},
$$
are called {\it row flags}. 

(Conjugate flags) For flags $\mathbf{a}, \mathbf{b}$ define the conjugate flags $\mathbf{a}' = (a'_i), \mathbf{b}' = (b'_i)$ given by 
$$
a'_i = a_i + c(\gamma_i), \quad b'_i = b_i + c(\delta_i), \quad i \ge 1, 
$$
where $\gamma_i, \delta_i$ denote the top and bottom cells of $i$-th column of the diagram. (We  define conjugation this way to include refined shapes with shifted contents.) 
Note that if $\mathbf{a}, \mathbf{b}$ are column flags for the shape $\lambda/\mu$, then $\mathbf{a}', \mathbf{b}'$ are row flags for the shape $\lambda'/\mu'$ (and vice versa). 
\end{definition}

We define (column) flagged supersymmetric Schur functions as follows. 

\begin{definition}[Flagged supersymmetric Schur functions]
For a skew r-shape $\lambda/\mu$ and (doubly infinite) sets of variables $\mathbf{y} = (y_i)_{i \in \mathbb{Z}}$ and $\mathbf{z} = (z_i)_{i \in \mathbb{Z}}$ we define the skew {\it flagged supersymmetric Schur functions} as follows: 
$$
\mathsf{S}^{\mathbf{a}, \mathbf{b}}_{\lambda/\mu}\left(\mathbf{y} / \mathbf{z} \right) := 
\det\left[e_{\lambda'_i - i - \mu'_j + j}(\mathbf{y}_{a_j, b_i} \,/\, \mathbf{z}_{a'_j, b'_i}) \right]_{1 \le i,j \le n},
$$
where $n \ge \lambda_1$ and $\mathbf{a}, \mathbf{b}$ are column flags. 
\end{definition}

For $\mathbf{z} = 0$ this function specializes to flagged Schur function. Without flag restrictions (i.e. letting $a_i \to -\infty, b_i \to \infty$), this function specializes to supersymmetric Schur function. 

\vspace{0.5em}

We similarly define the \textit{row flagged supersymmetric Schur functions}  
$\overline{\mathsf{S}}^{\mathbf{a',b'}}_{\lambda/\mu}(\mathbf{y} / \mathbf{z})$ with 
row flags $\mathbf{a', b'}$ as follows:
\begin{align*}
    \overline{\mathsf{S}}^{\mathbf{a}', \mathbf{b}'}_{\lambda/\mu}(\mathbf{y} / \mathbf{z}) = \det \left[ h_{\lambda_i-\mu_j-i+j}(\mathbf{y}_{a'_j,b'_i} / \mathbf{z}_{a_j,b_i}) \right]_{1 \le i,j \le n}, 
\end{align*}
for $n \ge \ell(\lambda)$. 
Then the two functions are related via the following duality (which shows that it is enough to consider just one of them).  

\begin{proposition}[Duality for column and row flagged functions]
\label{prop:duality}
The following duality holds: 
$$
\mathsf{S}^{\mathbf{a, b}}_{\lambda/\mu}(\mathbf{y} / \mathbf{z}) = \overline{\mathsf{S}}^{\mathbf{a}', \mathbf{b}'}_{\lambda'/\mu'}(\overline{\mathbf{z}} / \overline{\mathbf{y}}). 
$$
\end{proposition}
\begin{proof}
We have 
\begin{align*}
    \mathsf{S}^{\mathbf{a}, \mathbf{b}}_{\lambda/\mu}\left(\mathbf{y} / \mathbf{z} \right) 
    &= \det\left[e_{\lambda'_i - i - \mu'_j + j}(\mathbf{y}_{a_j, b_i} \,/\, \mathbf{z}_{a'_j, b'_i}) \right]_{1 \le i,j \le n} \\
    &= \det\left[h_{\lambda'_i - i - \mu'_j + j}(\mathbf{\overline{z}}_{a'_j, b'_i} \,/\, \mathbf{\overline{y}}_{a_j, b_i}) \right]_{1 \le i,j \le n} \\ 
    &= \overline{\mathsf{S}}^{\mathbf{a}', \mathbf{b}'}_{\lambda'/\mu'}(\overline{\mathbf{z}} / \overline{\mathbf{y}})
\end{align*}
as desired. 
\end{proof}

\subsection{$g$-specialization}
We are mainly interested in the following specialization of flagged supersymmetric Schur functions which reduces them to the dual refined canonical stable Grothendieck polynomials. 

\begin{definition}[$g$-specialization]
Let $m$ be a nonnegative integer. Define the {\it $g$-specialization} as the following substitution of variables $\mathbf{y} = (y_i), \mathbf{z} = (z_i)$
$$
g : y_i \to 
\begin{cases}
 x_i, & \text{ if } i \in [m], \\
 \beta_{i - m}, & \text{ if } i > m, \\
 -\alpha_{-i + 1}, & \text{ if } i \le 0,
\end{cases} 
\qquad 
z_i \to 
\begin{cases}
 0, & \text{ if } i \in [m], \\
-\alpha_{i - m}, & \text{ if } i > m, \\
 \beta_{-i + 1}, & \text{ if } i \le 0, 
\end{cases} 
$$
where $\mathbf{x}_0 := \varnothing$. We then define the {\it flagged dual Grothendieck enumerator} 
$$
\mathsf{g}^{\mathbf{a}, \mathbf{b}}_{\lambda/\mu}(\mathbf{x}_m; \boldsymbol{\alpha}, \boldsymbol{\beta}) := \mathsf{S}^{\mathbf{a}, \mathbf{b}}_{\lambda/\mu}\left(\mathbf{y} / \mathbf{z} \right) |_g
$$
as the function under $g$-specialization of flagged supersymmetric Schur function. 
\end{definition}

\begin{proposition}
\label{prop:g-enum-ribbon}
Let $m\in \mathbb{Z}_{\ge 0}$ and $\mathbf{a} = (a_i), \mathbf{b} = (b_i)$ be flags. Let us denote $\tau(k) = -k + 1$ and $\eta(k) = k - m$. 
Then we have 
$$
\mathsf{g}^{\mathbf{a}, \mathbf{b}}_{\lambda/\mu}(\mathbf{x}_m; \boldsymbol{\alpha}, \boldsymbol{\beta}) 
= 
\det \left[e_{\lambda'_i - i - \mu'_j + j}(\mathbf{x}_{u_j,v_i}, \overline{\boldsymbol{\alpha}}_{\tau(b_i), \tau(a_j)}, \boldsymbol{\beta}_{\eta(a_j), \eta(b_i)} \,/\, \overline{\boldsymbol{\alpha}}_{\eta(a'_j), \eta(b'_i)}, \boldsymbol{\beta}_{\tau(b'_i), \tau(a'_j)}) \right]_{1 \le i,j \le n},
$$
where $n \ge \lambda_1$ and $u_j = \max(1, a_j)$, $v_i = \min(m, b_i)$ (we also set $\alpha_i = \beta_i = 0$ if $i \le 0$). 
In particular, 
when {$\lambda/\mu$ is usual shape} and $\mathbf{a} = (a_i), \mathbf{b} = (b_i) \in \mathbb{Z}^n$ are the following column flags:  
$$
a_i = -i + 2, \quad b_i = \lambda'_i + m -1, \quad i \in [n],
$$
then flagged supersymmetric Schur functions and flagged dual Grothendieck enumerator specialize to dual refined canonical stable Grothendieck polynomials  
$$
\mathsf{S}^{\mathbf{a}, \mathbf{b}}_{\lambda/\mu}\left(\mathbf{y} / \mathbf{z} \right)|_g = \mathsf{g}^{\mathbf{a}, \mathbf{b}}_{\lambda/\mu}(\mathbf{x}_m; \boldsymbol{\alpha}, \boldsymbol{\beta}) = g_{\lambda/\mu}(\mathbf{x}_m; \boldsymbol{\alpha}, \boldsymbol{\beta})
$$
under $g$-specialization. 
\end{proposition}
\begin{proof}
    Let $\mathbf{a, b}$ be any column flags. 
    Let us see how $g$-specialization applies to $\mathbf{y}_{a_j, b_i}$. If $a_j > b_i$ for some $i,j$, we have $\mathbf{y}_{a_j, b_i} = \varnothing$, and $\mathbf{x}_{u_j,v_i} = \overline{\boldsymbol{\alpha}}_{\tau(b_i), \tau(a_j)} = \boldsymbol{\beta}_{\eta(a_j), \eta(b_i)} = \varnothing$, thus there is nothing to check. If $a_j \le b_i$ there are several routine cases all of which lead to 
$$\mathbf{y}_{a_j, b_i} \rightarrow \mathbf{x}_{u_j,v_i} \cup  \overline{\boldsymbol{\alpha}}_{\tau(b_i), \tau(a_j)} \cup \boldsymbol{\beta}_{\eta(a_j), \eta(b_i)}.$$    
    The specialization of $\mathbf{z}_{a'_j, b'_i}$ is obtained by the same argument.  

    Let $\mathbf{a} = (-i + 2)_i,\ \mathbf{b} = (\lambda'_i + m - 1)_i$. The conjugate flags are $\mathbf{a}' = (-\mu'_i + 1)_i,\ \mathbf{b}' = (i + m - 1)_i$. Let us calculate the boundaries of variable sets: 
    \begin{alignat*}{4}
        &\tau(a_j) = j - 1& \qquad &\tau(a'_j) = \mu'_j& \qquad &\eta(b_i) = \lambda'_i - 1& \qquad &\eta(b'_i) = i - 1, \\
        &\tau(b_i) \le 0& \qquad &\tau(b'_i) \le 0& \qquad &\eta(a_j) \le 1& \qquad &\eta(a'_j) \le 1.
    \end{alignat*}
    Then we have 
    \begin{align*}
    \mathsf{g}^{\mathbf{a}, \mathbf{b}}_{\lambda/\mu}(\mathbf{x}_m; \boldsymbol{\alpha}, \boldsymbol{\beta}) 
    = \det \left[e_{\lambda'_i - i - \mu'_j + j}(\mathbf{x}_m, \overline{\boldsymbol{\alpha}}_{j - 1}, \boldsymbol{\beta}_{\lambda'_i - 1} \,/\, \overline{\boldsymbol{\alpha}}_{i - 1}, \boldsymbol{\beta}_{\mu'_j}) \right]_{1 \le i, j \le n} 
    = g_{\lambda/\mu}(\mathbf{x}_m; \boldsymbol{\alpha}, \boldsymbol{\beta})
    \end{align*}
    as desired. 
\end{proof}

\section{Tableaux formulas}\label{sec:tableaux}
In this section we describe combinatorial tableaux formulas for the functions 
$\mathsf{S}^{\mathbf{a}, \mathbf{b}}_{\lambda/\mu}(\mathbf{y} / \mathbf{z})$. 

\begin{definition}
    A \textit{flagged $\mathbb{Z}$-SSYT} of r-shape $\lambda/\mu$ 
    with flags 
    $\mathbf{a} = (a_i), \mathbf{b} = (b_i)$ is a semistandard tableau (i.e. weakly increasing along rows from left to right and strictly increasing along columns from top to bottom) $T = (T_{i,j})$ filled with integers 
    so that $a_j \le T_{i,j} \le b_j$. Let $ZT_{\lambda/\mu}(\mathbf{a}, \mathbf{b})$ be the set of all $\mathbb{Z}$-SSYT of the r-shape $\lambda/\mu$ with flags $\mathbf{a}, \mathbf{b}$.
\end{definition}

\begin{example}
    For $\lambda/\mu=(5,5,5,3,3)/(2,1)$  
    and flags $\mathbf{a} = (-2,-2,-1,0,0)$, $\mathbf{b} = (7,6,7,7,6)$, an 
    example of $\mathbb{Z}$-SSYT is shown in Fig.~\ref{fig:ztab1}.
    \label{ex:zssyt-1}
\end{example}

\ytableausetup{mathmode, boxsize=1.7em}
\begin{figure}
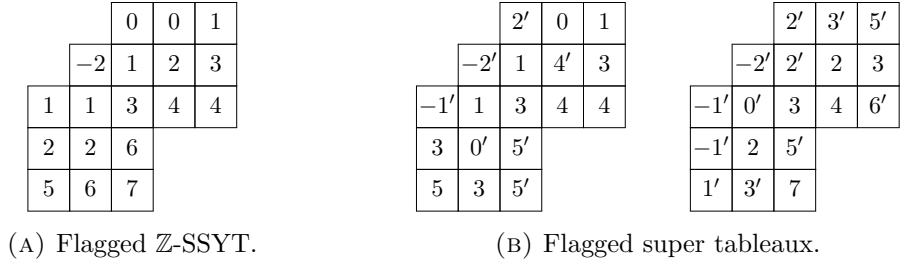

    \centering
    \begin{subfigure}[b]{0.3\textwidth}
        \centering
        \resizebox{3cm}{!}{
        \begin{ytableau}
        \none & \none & 0 & 0 & 1\\
        \none & -2 & 1 & 2 & 3 \\
        1 & 1 & 3 & 4 & 4 \\
        2 & 2 & 6 \\
        5 & 6 & 7
        \end{ytableau} 
        }
        \subcaption{Flagged $\mathbb{Z}$-SSYT.}
        \label{fig:ztab1}
    \end{subfigure}
    \begin{subfigure}[b]{0.6\textwidth}
        \centering
        \resizebox{3cm}{!}{
        \begin{ytableau}
            \none & \none & 2' & 0 & 1\\
            \none & -2' & 1 & 4' & 3 \\
            -1' & 1 & 3 & 4 & 4 \\
            3 & 0' & 5' \\
            5 & 3 & 5'
        \end{ytableau}
        }
        \hspace{1em} 
        \resizebox{3cm}{!}{
        \begin{ytableau}
            \none & \none & 2' & 3' & 5'\\
            \none & -2' & 2' & 2 & 3 \\
            -1' & 0' & 3 & 4 & 6' \\
            -1' & 2 & 5' \\
            1' & 3' & 7
        \end{ytableau}
        }
        \caption{Flagged super tableaux.}
        \label{fig:btab1}
    \end{subfigure}
    \caption{Examples of tableaux.}
\end{figure}

We now consider tableaux with extra primed entries $\mathbb{Z}' = \{\ldots  ,-2', -1',0', 1', 2', \ldots \}$. 

\begin{definition}
    \label{def:super-tableaux}
    For a given flagged $\mathbb{Z}$-SSYT $\tilde{T} = (\tilde{T}_{i,j})$, 
    we produce a 
    \textit{flagged super tableau} $T = (T_{i,j})$ of the same r-shape 
    as follows:
    \begin{align*}
        T_{ij} = \tilde{T}_{ij} \in \mathbb{Z} \quad \text{or} \quad T_{ij} = (\tilde{T}_{ij} + c({i, j}))' \in \mathbb{Z}',
    \end{align*}
    where $c(i,j)$ is a shifted content (for the given r-shape). 
    Note that every choice of ${T}$ induces a unique super tableau.   
Let $ST_{\lambda/\mu}(\mathbf{a}, \mathbf{b})$ be the set of all flagged super tableaux 
produced from the set $ZT_{\lambda/\mu}(\mathbf{a}, \mathbf{b})$ of $\mathbb{Z}$-SSYT of r-shape $\lambda/\mu$ with flags $\mathbf{a}, \mathbf{b}$.
Notice that in $T$, the vectors $\mathbf{a}', \mathbf{b}'$ 
are row flags for primed elements. 
For $r \in \mathbb{Z}$, define the weight $wt(T_{ij})$ as
\begin{equation}
    \label{eq:wt-func-bt-0}
    wt(T_{ij}) := 
    \begin{cases}
        y_r  \quad \text{if} \quad T_{ij}=r, \\ 
        -z_r  \quad \text{if} \quad T_{ij}=r', 
    \end{cases}
\end{equation}
and define the weight $wt(T)$ of the super tableau $T$ as
\begin{equation*}
    \label{eq:wt-func-bt}
    wt(T) = \prod_{ij}wt(T_{ij}).
\end{equation*}
\end{definition}

\begin{example}
    Some super tableaux produced from 
    the $\mathbb{Z}$-SSYT in Example~\ref{ex:zssyt-1} are
    shown in Fig.~\ref{fig:btab1}. 
There we have  
    \begin{align*}
    wt(T_1) &= y_0y_1^3y_3^4y_4^2y_5 (-z_{-2})(-z_{-1})(-z_0)(-z_2)(-z_4)(-z_5)^2, \\  
    wt(T_2) &= y_2^2y_3^2y_4y_7 (-z_{-2}) (-z_{-1})^2 (-z_0) (-z_1) (-z_2)^2 (-z_3)^2 (-z_5)^2 (-z_6),
    \end{align*}
    where $T_1, T_2$ are the left and right 
    tableaux, respectively.
\end{example}


We obtain the following combinatorial formula for $\mathsf{S}_{\lambda/\mu}^{\mathbf{a}, \mathbf{b}}(\mathbf{y} / \mathbf{z})$. 

\begin{theorem}[Tableaux formula for flagged supersymmetric Schur functions]
    \label{thm:signed-expansion}
    The following tableau formula holds:  
    \begin{align*}
        \label{eq:signed-expansion}
        \mathsf{S}^{\mathbf{a}, \mathbf{b}}_{\lambda/\mu}(\mathbf{y} / \mathbf{z}) = 
        \sum_{T\in ST_{\lambda/\mu}(\mathbf{a}, \mathbf{b})} wt(T).
    \end{align*}
\end{theorem}

\begin{remark}
    Note that this tableau formula becomes monomial positive for $\mathbf{S}^{\mathbf{a}, \mathbf{b}}_{\lambda/\mu}(\mathbf{y} / \overline{\mathbf{z}})$.
\end{remark}

\begin{remark}
    Without flag restrictions, supersymmetric Schur functions also have a formula via the so-called {\it bitableaux} (see \cite[\textsection\,I.5 ex.~23]{macdonald}), which are filled by the elements of the totally ordered set $\{1 < 2 < \ldots < 1' < 2' < \ldots\}$. 
    In contrast, there is no order between primed and unprimed elements in super tableaux. 
\end{remark}

We now give an analogue of Theorem~\ref{thm:signed-expansion} for row flagged supersymmetric polynomials. 
\begin{definition}
    A {\textit row flagged $\mathbb{Z}$-SSYT} with row flags $\mathbf{a}' = (a'_i), \mathbf{b}' = (b'_i)$ is 
    a semistandard tableau $T = (T_{i, j})$ filled with integers so that $a'_i \le T_{i, j} \le b'_i$. 
    
    For a given row flagged $\mathbb{Z}$-SSYT $\tilde{T} = (\tilde{T}_{i, j})$, we produce {\textit row flagged super tableau $T = (T_{i, j})$} of the same r-shape as follows:
    \begin{align*}
        T_{ij} = \tilde{T}_{ij} \in \mathbb{Z} \quad \text{or} \quad T_{ij} = (\tilde{T}_{ij} + c({i, j}))' \in \mathbb{Z}'.
    \end{align*}
    Let $\overline{ST}_{\lambda/\mu}(\mathbf{a}', \mathbf{b}')$ be the set of all row flagged super tableaux 
    produced from the set of row flagged $\mathbb{Z}$-SSYT of r-shape $\lambda/\mu$ with flags $\mathbf{a}', \mathbf{b}'$.
\end{definition}
\begin{corollary}[Tableaux formula for row flagged supersymmetric Schur functions]
    \label{cor:signed-expansion-dual}
    The following tableau formula holds:  
    \begin{align*}
       \overline{\mathsf{S}}^{\mathbf{a}', \mathbf{b}'}_{\lambda/\mu}(\mathbf{y} / \mathbf{z}) = \sum_{T\in \overline{ST}_{\lambda/\mu}(\mathbf{a}', \mathbf{b}')} wt(T).
    \end{align*}
\end{corollary}
\begin{proof}
Follows from Proposition~\ref{prop:duality} and Theorem~\ref{thm:signed-expansion}.
\end{proof}

\section{Ribbon decomposition formulas}\label{sec:ribbon}
\subsection{Outer ribbon decompositions}
In this subsection, we define the Hamel--Goulden $\#$ operation \cite{hg} on ribbons via cutting strips due to \cite{cyy}.

A {\it ribbon} is a connected skew partition whose diagram does not contain a $2\times 2$ block of cells. For a ribbon $\theta$, we call its bottom leftmost cell the \textit{head}, and its top rightmost cell the \textit{tail}.



For any skew r-shape $\lambda/\mu$, its \textit{outer ribbon decomposition} $\Theta=(\theta_1,...,\theta_k)$ is a decomposition of $\lambda/\mu$ into ribbons such that the head of every ribbon $\theta_i$ lies on the left or bottom perimeter and the tail lies on the top or right perimeter of the diagram. See Fig.~\ref{pic:young_ribbon_ex1} for an example. For a cell $\gamma \in \theta_i$, we say that it {\it goes up} if the cell adjacent above is also in $\theta_i$, or if $\gamma$ is the tail of $\theta_i$ and lies on top perimeter of $\lambda/\mu$; similarly, we say it {\it goes right} if the cell adjacent to the right is in $\theta_i$, or if $\gamma$ is the tail of $\theta_i$ and lies on the right perimeter of $\lambda/\mu$. Additionally, we define the direction for the top-right cell $\gamma$ in $\lambda/\mu$. Suppose that $\gamma \in \theta_i$. 
We say that it goes up, if the adjacent cell below $\gamma$ is also in $\theta_i$; similarly, we say it goes right, if the cell adjacent to the left is also in $\theta_i$.

\begin{remark}
    Let $\lambda/\mu$ be any skew r-shape. The r-shape $\lambda/\mu$ can be decomposed into edgewise connected components (also r-shapes) $B_1, \ldots, B_n$ which are called {\it blocks}. It follows that the outer decomposition $\Theta$ can be decomposed into $\Theta_1, \ldots, \Theta_n$, where $\Theta_i$ is an outer decomposition of $B_i$ for all $i$. Note that a flagged supersymmetric Schur function of r-shape $\lambda/\mu$ can be represented as a product of flagged supersymmetric Schur functions of r-shapes $B_i$ for $i \in [n]$. Thus, from now on we shall consider only skew r-shapes $\lambda/\mu$ with connected diagrams, which is enough for our purposes as main determinantal formulas for disconnected r-shapes can be written via blocks of connected r-shapes. 
\end{remark}

Let $\lambda/\mu$ be a connected r-shape. Every cell with content $c$ in $\lambda/\mu$ will either go up or right depending on $\Theta$. Using this property, we define the {\it cutting strip} $\Theta(\lambda/\mu)$ as a ribbon that contains cells with the same contents as the cells of $\lambda/\mu$, and every cell $\gamma \in \Theta(\lambda/\mu)$ goes up or right precisely when the cell in $\lambda/\mu$ with the content $c(\gamma)$ goes up or right. 


Let us denote the head and the tail of $\theta_i$ in the 
r-shape $\lambda/\mu$ by $\delta_i$ and $\gamma_i$, respectively. 
Each ribbon $\theta_i$ can be seen as 
a sub-diagram of $\Theta(\lambda/\mu)$, i.e. $\theta_i$ consists of cells whose contents lie in the interval $[c(\delta_i), c(\gamma_i)]$.
In general, let $p,q$ be contents of some cells in the diagram 
s.t. $p \le q$ and define 
$\Theta(p,q)$ as the sub-ribbon of $\Theta(\lambda/\mu)$ whose cell contents lie in $[p,q]$, and also
    $\Theta(p,q)=\varnothing$, if $p=q+1$, 
    and $\Theta(p,q)$ is undefined, if $p>q+1$. 
Then we define 
\[
\theta_i\#\theta_j := \Theta(c(\delta_i), c(\gamma_j)).
\]
Notice that $\theta_i \# \theta_j$ is an r-shape 
$(\theta_i \# \theta_j, c(\delta_i))$ for all $i,j$. 


\subsubsection{Pipe vectors} For subsequent proofs we will also need the following notions. 

Let $$\mathcal{P} := \{(0, 0), (0, 1), (1, 0), (1, 1)\},$$ where $0$ and $1$ mean horizontal and vertical unit directions, respectively. To each pair in $\mathcal{P}$ we match a {\it pipe}, whose direction of the first half is determined by the first value of the pair, and of the second half is determined by the second value, see Fig.~\ref{fig:possible-pipes}. 
\begin{figure}
    \centering
    \begin{tikzpicture}
        \draw[very thick] (1, 1) -- (2, 1);
    \end{tikzpicture}
    \hspace{1cm}
    \begin{tikzpicture}
        \draw[very thick] (1.5, 1.5) -- (2, 1.5) -- (2, 2);
    \end{tikzpicture}
    \hspace{1cm}
    \begin{tikzpicture}
        \draw[very thick] (1.5, 1.5) -- (1.5, 2) -- (2, 2);
    \end{tikzpicture}
    \hspace{1cm}
    \begin{tikzpicture}
        \draw[very thick] (1, 1) -- (1, 2);
    \end{tikzpicture}
    \caption{Four types of pipes corresponding to the pairs $(0, 0), (0, 1), (1, 0), (1,1)$, respectively.}
    \label{fig:possible-pipes}
\end{figure}
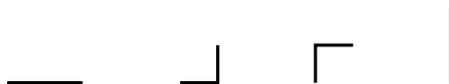

A vector $\pi = (\pi_1,..., \pi_m) \in \mathcal{P}^{m}$ where $\pi_i= (\pi^1_i, \pi^2_i)$ is called a {\it pipe vector}, if the following conditions are satisfied:  

(1) for each $i \in [m-1]$ we have $\pi^{2}_i = \pi^{1}_{i+1}$, and

(2) $\pi^{1}_1 = \pi^{2}_1$ and $\pi^{1}_m = \pi^{2}_m$. 

Let $m = c_{\mathrm{max}} - c_{\mathrm{min}} + 1$ and let $r$ be the root content of the r-shape $\lambda/\mu$. 
For the cutting strip $\Theta(\lambda/\mu)$ we define the pipe vector $\pi = (\pi_1, \ldots, \pi_m)$ as follows. Let $\gamma, \delta$ be cells with the contents $c - 1, c$ (resp.) in $\Theta(\lambda/\mu)$, such that $c \ge r + 1$. Then we define $\pi_i = (\pi_i^1, \pi_i^2)$ as follows: 
\begin{align*}
    \pi_{c - r + 1}^1 = 
    \begin{cases}
        0, & \text{ if } \gamma \text{ goes right} \\
        1, & \text{ if } \gamma \text{ goes up} 
    \end{cases} \qquad 
    \pi_{c - r + 1}^2 = 
    \begin{cases}
        0, & \text{ if } \delta \text{ goes right} \\
        1, & \text{ if } \delta \text{ goes up} 
    \end{cases}
\end{align*}
For $c = r$, let $\delta$ be the cell of $\Theta(\lambda/\mu)$ with the content $c$ (i.e. $\delta$ is the head of $\Theta(\lambda/\mu)$). Then we set $\pi_1 = (0, 0)$ if $\delta$ goes right, and $\pi_1 = (1, 1)$ if $\delta$ goes up. It is easy to see that $\pi$ is indeed a pipe vector.

Conversely, for any pipe vector $\pi \in \mathcal{P}^m$  we can construct the cutting strip $\Theta(\lambda/\mu)$, where $m = c_{\mathrm{max}} - c_{\mathrm{min}} + 1$, which in turn, defines the ribbon decomposition $\Theta$. It suffices to specify the pipe vector $\pi$ to  uniquely determine the ribbon decomposition $\Theta$. 


\subsubsection{Decompositions of ribbons} Let us now define the right-arrow $\rightarrow$ and up-arrow  $\uparrow$ operations on Young diagrams $D_1, D_2$: 

1. The diagram $D_1\rightarrow D_2$ is constructed by stacking the bottom left cell of $D_2$ to the right of the upper right cell of $D_1$ (see Fig.~\ref{fig:ex-right-arrow}). 

2. The diagram $D_1\uparrow D_2$ is constructed by stacking the bottom left cell of $D_2$ to the top of the upper right cell of $D_1$ (see Fig.~\ref{fig:ex-up-arrow}).

\ytableausetup{mathmode, boxsize=1em}
\begin{figure}
    \centering
    \begin{subfigure}[b]{0.45\textwidth}
    \centering   
    \raisebox{0.25\height}{
    \begin{ytableau}
        \ & &  \\
        \ & & \none \\
    \end{ytableau}
    }
    $\rightarrow$
    \raisebox{0.25\height}{
    \begin{ytableau}
        \ & \\
        \ & \\
    \end{ytableau}
    }
    $=$
    \raisebox{0.75\height}{
    \begin{ytableau}
        \none & \none & \none & & \\
        \ & \ & \ & & \\
        \ & \
    \end{ytableau}
    }
    \subcaption{Right arrow operation.}
    \label{fig:ex-right-arrow}
    \end{subfigure}
    \begin{subfigure}[b]{0.45\textwidth}
    \centering   
    \raisebox{0.25\height}{
    \begin{ytableau}
        \ & &  \\
        \ & & \none \\
    \end{ytableau}
    }
    $\uparrow$
    \raisebox{0.25\height}{
    \begin{ytableau}
        \ & \\
        \ & \\
    \end{ytableau}
    }
    $=$
    \raisebox{1.25\height}{
    \begin{ytableau}
        \none & \none & & \\
        \none & \none & & \\
        \ & \ & \ \\
        \ & \
    \end{ytableau}
    }
    \subcaption{Up arrow operation.}
    \label{fig:ex-up-arrow}
    \end{subfigure}
    \caption{}
\end{figure}

\ytableausetup{mathmode, boxsize=1.5em}
\begin{figure}
    \centering
    \begin{subfigure}[b]{0.5\textwidth}
    \centering
    \begin{tikzpicture}[y=-1cm, scale=0.7]
	\draw (3, 1) rectangle (4, 2);
	\draw (4, 1) rectangle (5, 2);
	\draw (2, 2) rectangle (3, 3);
	\draw (3, 2) rectangle (4, 3);
	\draw (4, 2) rectangle (5, 3);
	\draw (1, 3) rectangle (2, 4);
	\draw (2, 3) rectangle (3, 4);
	\draw (3, 3) rectangle (4, 4);
	\draw (4, 3) rectangle (5, 4);
	\draw (1, 4) rectangle (2, 5);
	\draw (2, 4) rectangle (3, 5);
	\draw [red, very thick] (1, 4.5) -- (2.5, 4.5) -- (2.5, 2.5) -- (3.5,2.5) -- (3.5, 1);
	\draw [red, very thick] (1, 3.5) -- (1.5, 3.5) -- (1.5, 3);
	\draw [red, very thick] (3.5, 4) -- (3.5, 3.5) -- (4.5, 3.5) -- (4.5, 1);
    \end{tikzpicture}
    \subcaption{An outer ribbon decomposition.}
    \label{pic:young_ribbon_ex1}
    \end{subfigure}
    \begin{subfigure}[b]{0.45\textwidth}
    \centering   
    \resizebox{2cm}{!}{
    \begin{ytableau}
         \none & \none & 3 \\
         \none & \none & 2 \\
         \none &0& 1 \\
         \none &-1& \none  \\
        -3&-2&\none\\
    \end{ytableau}
    }
    \subcaption{Cutting strip $\Theta(\lambda/\mu)$.}
    \label{fig:ex-1-stuff}
    \end{subfigure} \\[1em]
    \ytableausetup{mathmode, boxsize=0.4em}
\begin{subfigure}[b]{0.5\textwidth}
\[\arraycolsep=0.4cm\def\arraystretch{1}
\left(
\begin{array}{ccc}
    \begin{ytableau}
         \none & \none &  \\
         \none & &  \\
         \none & & \none  \\
        &&\none\\
    \end{ytableau} &     
    \begin{ytableau}
       ~ & ~
    \end{ytableau} & 
    \begin{ytableau}
         \none & \none &  \\
         \none & \none &  \\
         \none & &  \\
         \none & & \none  \\
        &&\none\\
    \end{ytableau}
    \\
    \begin{ytableau}
         \none &  \\
         &  \\
         & \none  \\
         &\none\\
    \end{ytableau} & 
    \begin{ytableau}
        \\
    \end{ytableau} &
     \begin{ytableau}
         \none &  \\
         \none &  \\
         &  \\
         & \none  \\
         &\none\\
    \end{ytableau} \\
    \begin{ytableau}
        \none & \ \\
        \ & \
    \end{ytableau} &
    \begin{ytableau}
        \none[undef]
    \end{ytableau} &
    \begin{ytableau}
        \none & \ \\
        \none & \ \\
        \ & \ 
    \end{ytableau}
    \end{array}
    \right)
    \]
    \subcaption{Ribbon matrix $(\theta_i\#\theta_j)_{i,j}^3$}
    \label{fig:ribbon-matrix-ex-1}
    \end{subfigure} 
\begin{subfigure}[b]{0.45\textwidth}
    \centering
    \ytableausetup{mathmode, boxsize=1.0em}
    \begin{ytableau}
         \none & \none &  \\
         \none & \none &  \\
         \none & &  \\
         \none & & \none  \\
        &&\none\\
    \end{ytableau}
    \raisebox{-5\height}{%
    $=$
    }
    \raisebox{-2\height}{%
    \begin{ytableau}
        \ 
    \end{ytableau}
    }
    \raisebox{-5\height}{%
    $\rightarrow$
    }
    \raisebox{-1\height}{%
    \begin{ytableau}
        \ \\
        \ \\
        \ \\
    \end{ytableau}
    }
    \raisebox{-5\height}{%
    $\rightarrow$
    }
    \raisebox{-1\height}{%
    \begin{ytableau}
        \ \\
        \ \\
        \ \\
    \end{ytableau}
    }
    \caption{Canonical decomposition of $\theta_1\#\theta_3$}
    \label{fig:canonical-ribbon-decomposition-ex-1}
    \end{subfigure}
\caption{}
\end{figure}

\ytableausetup{mathmode, boxsize=2em}

\begin{definition}[Canonical decomposition]
    For a ribbon $\theta$, the expression $\theta=\rho_1\rightarrow...\rightarrow \rho_k$ is called the \textit{canonical decomposition}, where each $\rho_i$ is a vertical strip. 
    Similarly, the expression $\theta = \rho_1 \uparrow \ldots \uparrow \rho_k$ is called an {\it up canonical decomposition}, where each $\rho_i$ is a horizontal strip.
\end{definition}

\begin{example}
    \label{exm:outer-dec}
    Let $\lambda/\mu=(4,4,4,2)/(2,1)$ be a (usual) shape. An outer ribbon decomposition and 
    the corresponding cutting strip are shown in Fig. \ref{pic:young_ribbon_ex1}, \ref{fig:ex-1-stuff}. There are three ribbons in this decomposition: 
    $\theta_1=\Theta(-3,2), \theta_2=\Theta(-2,-2), \theta_3=\Theta(0,3).$ 
    The ribbon matrix  
    $[\theta_i\#\theta_j]$ is shown in Fig.~\ref{fig:ribbon-matrix-ex-1}.
    The canonical decomposition of the ribbon $\theta_1 \# \theta_3$ is shown in Fig.~\ref{fig:canonical-ribbon-decomposition-ex-1}. 
    The corresponding pipe vector is 
    $\pi=((0,0),(0,1), (1,1),(1,0),(0,1),(1,1),(1,1)).$
\end{example}

\subsection{Flagged Hamel--Goulden formulas}
In this subsection we state Hamel--Goulden-type formulas for the functions 
$\mathsf{S}_{\lambda/\mu}^{\mathbf{a}, \mathbf{b}}
(\mathbf{y} / \mathbf{z})$. 

\vspace{0.5em}

The following is the key definition for obtaining ribbon determinantal formulas. 

\begin{definition}[Induced ribbon flags]
For a skew r-shape $\lambda/\mu$ and its outer ribbon decomposition $(\theta_1, \ldots, \theta_k)$, let $\theta_i \# \theta_j = \rho^{ij}_1\rightarrow ... \rightarrow \rho^{ij}_{k_{ij}}$ 
be the corresponding canonical decompositions.  Denote by $\delta^{ij}_r$ and $\gamma^{ij}_r$ the bottom and top cells of $\rho^{ij}_r$. Let $M^{ij}_r := \max(c(\delta^{ij}_r))$ and $m^{ij}_r := \min(c(\gamma^{ij}_r))$, where $\max(c), \min(c)$ denote the rightmost and the leftmost cells  with the content $c$ in the r-shape $\lambda/\mu$. For flags $\mathbf{a} = (a_i), \mathbf{b} = (b_i)$, define the {\it induced ribbon flags} $\mathbf{a}^{ij}=(a^{ij}_r), \mathbf{b}^{ij}=(b^{ij}_r)$ as follows:  
\begin{align*}
    a^{ij}_r = a_{\mathrm{col}(m^{ij}_r)} - 
    (\mu'_{\mathrm{col}(m^{ij}_r)} + 1) + \mathrm{row}(m^{ij}_r),
    \quad b^{ij}_r = b_{\mathrm{col}(M^{ij}_r)}-
    \lambda'_{\mathrm{col}(M^{ij}_r)} + \mathrm{row}(M^{ij}_r),
\end{align*}
with 
$\mathbf{a}^{ij} = \mathbf{b}^{ij} = \varnothing$ if $\theta_i \# \theta_j = \varnothing$ or  if $\theta_i \# \theta_j$ is undefined. 
\end{definition}
\begin{proposition}
    The induced ribbon flags $\mathbf{a}^{ij}, \mathbf{b}^{ij}$ are flags for the r-shape $\theta_i\# \theta_j$.
\end{proposition}
\begin{proof}
    Let us check this condition for $\mathbf{a}^{ij}$ (which is very similar for $\mathbf{b}^{ij}$).
    Let $\rho^{ij}_r, \rho^{ij}_{r+1}$ be vertical 
    ribbons from the canonical decomposition of 
    $\theta_i\# \theta_j$. Let $\gamma_r, \gamma_{r+1}$ be top cells 
    of $\rho^{ij}_r, \rho^{ij}_{r+1}$, respectively. Then we need to show that 
    $$a^{ij}_r - a^{ij}_{r+1} \le \mathrm{row}(\gamma_r) - \mathrm{row}(\gamma_{r+1}) + 1,$$
    which follows from the following fact: let $m_i = \min(c(\gamma_i))$ for $i=r,r+1$; by definition, we have  
    \begin{align*}
    a^{ij}_r - a^{ij}_{r+1} &= a_{\mathrm{col}(m_r)} - \mu'_{\mathrm{col}(m_r)} + \mathrm{row}(m_r) - a_{\mathrm{col}(m_{r+1})} + \mu'_{\mathrm{col}(m_{r+1})} - \mathrm{row}(m_{r+1})\\ 
    &\le \mathrm{col}(m_{r+1}) - \mathrm{col}(m_r) + \mathrm{row}(m_r) - \mathrm{row}(m_{r+1}) \\
    &= c(m_{r+1}) - c(m_{r}) \\ 
    &= c(\gamma_{r+1}) - c(\gamma_{r}) \\ 
    &= \mathrm{row}(\gamma_r) - \mathrm{row}(\gamma_{r+1}) + 1
    \end{align*}
    as needed.\footnote{ 
        It might happen that $\mathrm{col}(m_{r+1}) - \mathrm{row}(m_{r+1}) \neq c(m_{r+1})$ (as $\lambda/\mu$ is an r-shape with a shift), but $(\mathrm{col}(m_{r+1}) - \mathrm{row}(m_{r+1})) - (\mathrm{col}(m_r) - \mathrm{row}(m_r)) = c(m_{r+1}) - c(m_r)$ as all contents are shifted by the same number and difference of any two contents stays the same for any shift.
    }
\end{proof}


\begin{example}
    \label{ex:ribbon-flags-2}
    Consider the shape and outer decomposition from Example~\ref{exm:outer-dec}. Let $\mathbf{a} = (a_i), \mathbf{b} = (b_i)$ be some flags for the shape $\lambda/\mu = (4,4,4,2) / (2,1)$.
    Let us show the induced ribbon flags for the shape $\theta_2 \# \theta_1 = \rho^{21}_1 \rightarrow \rho^{21}_2$ (Fig.~\ref{ex:ribbon-flags-3a}), where $\rho^{21}_1, \rho^{21}_2$ are vertical ribbons. The cells $m^{21}_1, M^{21}_1, m^{21}_2, M^{21}_2$ are shown in Fig.~\ref{ex:ribbon-flags-3b} and we have 
    \begin{align*}
        \mathbf{a}^{21} = (a_2, a_3), \quad \mathbf{b}^{21} = (b_2, b_4), \quad \mathbf{a'}^{21} = (a_2, a_3 + 2), \quad \mathbf{b'}^{21} = (b_2 - 2, b_4 + 1),
    \end{align*}
    where $\mathbf{a'}^{21}, \mathbf{b'}^{21}$ are conjugate flags to  $\mathbf{a}^{21}, \mathbf{b}^{21}$. 
\end{example}

\begin{remark}
Notice that in the previous example, when we calculate elements of the row flags $\mathbf{a'}^{21}$ and $\mathbf{b'}^{21}$, we use the \textit{shifted} contents of the \textit{r-shape} $\theta_i \# \theta_j$. The second element of $\mathbf{a'}^{21}$ is $a_3 + 2$, where $2$ is the content of the top cell of $\rho^{21}_2$. When we consider $\theta_i \# \theta_j$ as a usual shape (not r-shape), the content of the top cell of $\rho^{21}_2$ would be $1$.    
\end{remark}

\begin{figure}
    \centering
    \begin{subfigure}[b]{0.4\textwidth}
    \centering   
    \begin{ytableau}
         \none & 2 \\
         0& 1 \\
         -1& \none  \\
         -2&\none\\
    \end{ytableau}
    \subcaption{The ribbon $\theta_2 \# \theta_1$}
    \label{ex:ribbon-flags-3a}
    \end{subfigure}
    \begin{subfigure}[b]{0.4\textwidth}
    \centering
    \begin{ytableau}
        \none & \none & m^{21}_2 & \\
        \none & m^{21}_1 & & \\
        & & & M^{21}_2 \\
        & M^{21}_1 \\
    \end{ytableau}
    \subcaption{The cells $m^{21}_r, M^{21}_r$ in $\lambda/\mu$}
    \label{ex:ribbon-flags-3b}
    \end{subfigure}
    \caption{}
\end{figure}

Our main formulas are now stated as follows.

\begin{theorem}[Hamel--Goulden formulas for flagged supersymmetric Schur functions]
    \label{thm:HG-1}
    Let $(\theta_1, \ldots, \theta_k)$ be an outer ribbon decomposition of a connected r-shape 
    $\lambda/\mu$. Then the following determinantal formulas hold: 
    \[
    \mathsf{S}_{\lambda/\mu}^{\mathbf{a}, \mathbf{b}}(\mathbf{y} / \mathbf{z}) = 
    \det \left[ \mathsf{S}^{\mathbf{a}^{ij}, \mathbf{b}^{ij}}_{\theta_i\# \theta_j}
    (\mathbf{y} / \mathbf{z}) \right]_{1 \le i,j \le k},
    \]
    where $\mathsf{S}^{\mathbf{a}^{ij}, \mathbf{b}^{ij}}_{\varnothing}
    (\mathbf{y} / \mathbf{z}) = 1$ and $\mathsf{S}^{\mathbf{a}^{ij}, \mathbf{b}^{ij}}_{undefined}
    (\mathbf{y} / \mathbf{z}) = 0$.
\end{theorem}

The proof will be presented in the next section. It generally relies on the Lindstr\"om-Gessel-Viennot (LGV) lemma \cite{lgv}. 
We construct certain weighted lattice which we call {\it super lattice}, based on the ribbon decomposition and given flags. We then show that paths correspond to super tableaux with induced ribbon flags, and establish a weight-preserving bijection between nonintersecting path systems and flagged super tableaux. 

\vspace{0.5em}

We can now state similar formulas for row flagged supersymmetric Schur functions with row flags $\mathbf{a}', \mathbf{b}'$. 

\begin{definition}[Induced row ribbon flags]
For a skew r-shape $\lambda/\mu$ and its outer ribbon decomposition $(\theta_1, \ldots, \theta_k)$, let $\theta_i \# \theta_j = \rho^{ij}_1\uparrow ... \uparrow \rho^{ij}_{k_{ij}}$ 
be the corresponding up canonical decompositions.  Denote by $\delta^{ij}_r$ and 
$\gamma^{ij}_r$ the leftmost and rightmost cells of $\rho^{ij}_r$. Let 
$M^{ij}_r := \max(c(\gamma^{ij}_r))$ and 
$m^{ij}_r := \min(c(\delta^{ij}_r))$. 
For row flags $\mathbf{a}' = (a'_i), \mathbf{b}' = (b'_i)$, define the {\it induced row ribbon flags} $\mathbf{a}'^{ij}=(a'^{ij}_r), \mathbf{b}'^{ij}=(b'^{ij}_r)$ as follows:  
\begin{align*}
    a'^{ij}_r = a'_{\mathrm{row}(m^{ij}_r)},
    \quad b'^{ij}_r = b'_{\mathrm{row}(M^{ij}_r)}, 
\end{align*}
with 
$\mathbf{a}'^{ij} = \mathbf{b}'^{ij} = \varnothing$ if $\theta_i \# \theta_j = \emptyset$ or if $\theta_i \# \theta_j$ is undefined. 
\end{definition}

\begin{corollary}[Hamel--Goulden formulas for row flagged supersymmetric Schur functions]
    \label{thm:HG-1-dual}
    Let $(\theta_1, \ldots, \theta_k)$ be an outer ribbon decomposition of connected r-shape $\lambda/\mu$. Then the following determinantal formulas hold: 
    \[
    \overline{\mathsf{S}}_{\lambda/\mu}^{\mathbf{a}', \mathbf{b}'}(\mathbf{y} / \mathbf{z}) = 
    \det \left[ \overline{\mathsf{S}}^{\mathbf{a}'^{ij}, \mathbf{b}'^{ij}}_{\theta_i\# \theta_j}
    (\mathbf{y} / \mathbf{z}) \right]_{1 \le i,j \le k}
    \]
    where $\overline{\mathsf{S}}^{\mathbf{a}'^{ij}, \mathbf{b}'^{ij}}_{\varnothing}
    (\mathbf{y} / \mathbf{z}) = 1$ and $\overline{\mathsf{S}}^{\mathbf{a}'^{ij}, \mathbf{b}'^{ij}}_{undefined}
    (\mathbf{y} / \mathbf{z}) = 0$.
\end{corollary}
\begin{proof}
    We consider the same outer decomposition for the r-shape $\lambda'/\mu'$, let us denote it $(\theta'_1, \ldots, \theta'_k)$, where each ribbon $\theta'_i$ is an r-shape $(\theta'_i,\ -c(\gamma_i))$ where $\gamma_i$ is the tail of $\theta_i$. Note that $\theta_i \# \theta_j = (\theta'_j \# \theta'_i)'$. 
    Let $\zeta_i, \xi_i$ be the leftmost and rightmost cells in $i$-th row of $\lambda/\mu$. Let $\mathbf{a}, \mathbf{b}$ be conjugate column flags to $\mathbf{a}', \mathbf{b}'$. Notice that $\mathbf{a}, \mathbf{b}$ are column flags for the r-shape $\lambda'/\mu'$, thus $\mathbf{a} = (a'_r + c(\zeta_r)),\ \mathbf{b} = (b'_r + c(\xi_r))$. Let $\mathbf{a}^{ij} = (a^{ij}_r), \mathbf{b}^{ij} = (b^{ij}_r)$ be the induced ribbon flags for $\mathbf{a}, \mathbf{b}$. Note that $\mathbf{a}^{ij}, \mathbf{b}^{ij}$ are column flags for the r-shape $\theta'_i \# \theta'_j$. We have 
    \begin{align*}
    a'^{ij}_r &= a'_{\mathrm{row}(m^{ij}_r)} = a_{\mathrm{row}(m^{ij}_r)} - c(\zeta_{\mathrm{row}(m^{ij}_r)}) + c(m^{ij}_r) - c(m^{ij}_r) \\ 
    &= a_{\mathrm{row}(m^{ij}_r)} - \mu_{\mathrm{row}(m^{ij}_r)} - 1 + \mathrm{col}(m^{ij}_r) - c(m^{ij}_r) \\ 
    &= a^{ji}_r - c(m^{ij}_r),
    \end{align*}
    which shows that $\mathbf{a}^{ji}$ is conjugate to $\mathbf{a}'^{ij}$. It can be shown that $\mathbf{b}^{ji}$ is conjugate to $\mathbf{b}'^{ij}$. We then have 
    \begin{align*}
        \overline{\mathsf{S}}_{\lambda/\mu}^{\mathbf{a}', \mathbf{b}'}(\mathbf{y} / \mathbf{z}) &= \mathsf{S}_{\lambda'/\mu'}^{\mathbf{a}, \mathbf{b}}(\overline{\mathbf{z}} / \overline{\mathbf{y}}) 
        = \det\left[ \mathsf{S}_{\theta'_j \# \theta'_i}^{\mathbf{a}^{ji}, \mathbf{b}^{ji}}(\overline{\mathbf{z}} / \overline{\mathbf{y}}) \right]_{1 \le i,j \le k} 
        = \det\left[ \overline{\mathsf{S}}_{\theta_i \# \theta_j}^{\mathbf{a}'^{ij}, \mathbf{b}'^{ij}}(\mathbf{y} / \mathbf{z}) \right]_{1 \le i,j \le k}
    \end{align*}
    as needed. 
\end{proof}


Taking the $g$-specialization we now obtain Hamel--Goulden-type formulas for dual refined canonical stable Grothendieck polynomials with determinant entries written via flagged dual Grothendieck enumerators. 

\begin{corollary}[Hamel--Goulden formulas for dual refined canonical stable Grothendieck polynomials]
    Let $(\theta_1, \ldots, \theta_k)$ be an outer ribbon decomposition of  
    $\lambda/\mu$, and 
$\mathbf{a} = (a_i), \mathbf{b} = (b_i)$ be column flags defined 
as follows:
$
    a_i =  -i + 2, 
     b_i = \lambda'_i + m - 1,  i \ge 1. 
$ 
 Then the following determinantal formulas hold:
    \begin{align*}
        g_{\lambda/\mu}(\mathbf{x}_m; \boldsymbol{\alpha},\boldsymbol{\beta}) = 
        \det \left[ \mathsf{g}_{\theta_i \# \theta_j}^{\mathbf{a}^{ij}, 
        \mathbf{b}^{ij}}(\mathbf{x}_m; \boldsymbol{\alpha},\boldsymbol{\beta}) \right]_{1 \le i, j \le k},
    \end{align*}
    where the induced ribbon flags are given by 
    \begin{align*}
    \mathbf{a}^{ij} = (-c(m^{ij}_r) - \mu'_{\mathrm{col}(m^{ij}_r)} + 1)_{r \in [k_{ij}]}, \qquad 
    \mathbf{b}^{ij} = (\mathrm{row}(M^{ij}_r) + m - 1)_{r \in [k_{ij}]}.
    \end{align*}
\end{corollary}

\begin{example}
    \label{ex:calc-S}
    Consider the shape and its outer ribbon decomposition from Example~\ref{exm:outer-dec}. 
    Let the flags be $\mathbf{a} = (1, 0, -1, -2)$, $\mathbf{b} = (6, 6, 5, 5)$.    
    Then $\mathbf{a}' = (-1, 0, 1, 1)$ $\mathbf{b}' = (3, 4, 5, 6)$ and the induced ribbon flags are the following: 
{\small
    \begin{alignat*}{4}
    \mathbf{a}^{11} &= (2,0,-1) 
    &\qquad \mathbf{b}^{11} &= (6,6,5)
    &\qquad \mathbf{a}'^{11} &= (-1,0,1) 
    &\qquad \mathbf{b}'^{11} &= (3,4,6)\\
    \mathbf{a}^{12} &= (2, 1) 
    &\qquad \mathbf{b}^{12} &= (6, 6) 
    &\qquad \mathbf{a}'^{12} &= (-1, -1) 
    &\qquad \mathbf{b}'^{12} &= (3, 4)\\
    \mathbf{a}^{13} &= (2, 0, -2) 
    &\qquad \mathbf{b}^{13} &= (6, 6, 5)
    &\qquad \mathbf{a}'^{13} &=  (-1, 0, 1) 
    &\qquad \mathbf{b}'^{13} &=  (3, 4, 6)\\
    \mathbf{a}^{21} &= (0, -1) 
    &\qquad \mathbf{b}^{21} &= (6, 5) 
    &\qquad \mathbf{a}'^{21} &= (0, 1) 
    &\qquad \mathbf{b}'^{21} &= (4, 6)\\ 
    \mathbf{a}^{22} &= (1) 
    &\qquad \mathbf{b}^{22} &= (6) 
    &\qquad \mathbf{a}'^{22} &= (-1) 
    &\qquad \mathbf{b}'^{22} &= (4)\\
    \mathbf{a}^{23} &= (0, -2) 
    &\qquad \mathbf{b}^{23} &= (6, 5)
    &\qquad \mathbf{a}'^{23} &= (0, 1) 
    &\qquad \mathbf{b}'^{23} &= (4, 6)\\ 
    \mathbf{a}^{31} &= (0, -1) 
    &\qquad \mathbf{b}^{31} &= (5, 5)
    &\qquad \mathbf{a}'^{31} &= (0, 1) 
    &\qquad \mathbf{b}'^{31} &= (5, 6)\\
    \mathbf{a}^{32} &= \varnothing 
    &\qquad \mathbf{b}^{32} &= \varnothing 
    &\qquad \mathbf{a}'^{32} &= \varnothing
    &\qquad \mathbf{b}'^{32} &= \varnothing \\
    \mathbf{a}^{33} &= (0, -2) 
    &\qquad \mathbf{b}^{33} &= (5, 5)
    &\qquad \mathbf{a}'^{33} &= (0, 1) 
    &\qquad \mathbf{b}'^{33} &=  (5, 6)
    \end{alignat*}
}
    Then we have    $$\mathsf{S}^{\mathbf{\mathbf{a},\mathbf{b}}}_{\lambda/\mu}(\mathbf{y} / \mathbf{z}) = \det \left[ \mathsf{S}^{\mathbf{a}^{ij}, 
    \mathbf{b}^{ij}}_{\theta_i \# \theta_j}(\mathbf{y} / \mathbf{z}) \right]_{1 \le i,j \le 3},$$
For $m = 3$, after applying $g$-specialization we have:
    \begin{align*}
        g_{\lambda/\mu}(\mathbf{x}_m;\boldsymbol{\alpha},\boldsymbol{\beta}) = 
        \det \left[ \mathsf{g}^{\mathbf{a}^{ij}, \mathbf{b}^{ij}}
        _{\theta_i \# \theta_j}        (\mathbf{x}_m;\boldsymbol{\alpha},\boldsymbol{\beta}) \right]_{1\le i,j \le 3}. 
    \end{align*}
\end{example}

\noindent{\bf Strict flags.} We can also slightly change the definition of ribbon flags, and get another Hamel--Goulden-type formula. 
Let us define \textit{strict flags}  
$\mathbf{\tilde{a}} = (\tilde{a}_i),$ $\mathbf{\tilde{b}} = 
(\tilde{b}_i)$ as flags 
where the flag 
$\mathbf{\tilde{a}}$ respects more 
strict inequality: $\tilde{a}_i - \tilde{a}_{i+1} \le 1$.
Then we define flags for the ribbons $\theta_i \# \theta_j$ as follows. 
Let $\theta_i \# \theta_j = \rho^{ij}_1\rightarrow ... \rightarrow \rho^{ij}_{k_{ij}}$ 
be canonical decomposition and let $\delta^{ij}_r$ and 
$\gamma^{ij}_r$ be the bottom and top cells of $\rho^{ij}_r$. 
Define the induced ribbon flags 
$\mathbf{\tilde{a}}^{ij}=(\tilde{a}^{ij}_r)_r, \mathbf{\tilde{b}}^{ij}=(\tilde{b}^{ij}_r)_r$: 
\begin{align*}
    \tilde{a}^{ij}_r = \tilde{a}_{\mathrm{col}(\min(c(\gamma^{ij}_r)))}
    \qquad 
    \tilde{b}^{ij}_r = \tilde{b}_{\mathrm{col}(\max(c(\delta^{ij}_r)))}-
    \lambda'_{\mathrm{col}(\max(c(\delta^{ij}_r)))} + \mathrm{row}(\max(c(\delta^{ij}_r)))
\end{align*}
It can be shown that $\mathbf{\tilde{a}}^{ij}, \mathbf{\tilde{b}}^{ij}$ are 
indeed flags. 
We can also similarly prove that the same Hamel--Goulden-type formulas hold for these strict flags. 


Now we state a slightly different version of Hamel--Goulden-type formulas, which can be proved similarly as for the above flags. 
\begin{theorem}[Hamel--Goulden formulas with strict flags]
    \label{thm:hg-strict-flags}
    Let $(\theta_1, \ldots, \theta_k)$ be an outer ribbon decomposition of  
    $\lambda/\mu$. Then the following determinantal formulas hold with strict flags $\mathbf{\tilde{a}},\mathbf{\tilde{b}}$: 
$$
\mathsf{S}_{\lambda/\mu}^{\mathbf{\tilde{a}},\mathbf{\tilde{b}}}(\mathbf{y} / \mathbf{z}) = 
\det \left[\mathsf{S}^{\mathbf{\tilde{a}}^{ij}, \mathbf{\tilde{b}}^{ij}}_{\theta_i\# \theta_j}
(\mathbf{y} / \mathbf{z})\right]_{1 \le i,j \le k}.
$$
\end{theorem}

\begin{remark}
One can notice that the induced ribbon flags $\mathbf{\tilde{a}}^{ij}, \mathbf{\tilde{b}}^{ij}$ need 
not be strict flags, and so the latter formula is not `recursive', i.e. 
one can not use it again to compute elements of the determinant, whereas the previously stated versions of Hamel--Goulden formulas  can be used `recursively' in that sense. 
Furthermore, the previous version is more general, as its flag conditions are weaker. We note however that the Hamel--Goulden-type formulas with strict flags give more compactly written formulas for the functions $g_{\lambda/\mu}(\mathbf{x}_m;\boldsymbol{\alpha}, \boldsymbol{\beta})$.     
\end{remark}

\section{Proofs}\label{sec:proofs}
\subsection{Summary of proofs.} 

In \textsection\,\ref{subsec:lattice} we define a lattice graph, which we call $\mathbb{Z}$-lattice, using the given r-shape $\lambda/\mu$, column flags $\mathbf{a}, \mathbf{b}$ and outer decomposition $\Theta = (\theta_1,\ldots, \theta_k)$. We consider paths on the lattice with defined starting nodes $\mathbf{C} = (C_1, \ldots,C_k)$ and ending nodes $\mathbf{D} = (D_1,\ldots, D_k)$. 

In \textsection\,\ref{subsec:bijection} we establish a bijection between non-intersecting systems of paths on the $\mathbb{Z}$-lattice and $\mathbb{Z}$-SSYT of the given r-shape $\lambda/\mu$ and flags $\mathbf{a}, \mathbf{b}$ (Lemma~\ref{thm:bijection-super-tableau-non-intersecting-paths}). We first obtain a bijection between paths from $C_i$ to $D_j$ and flagged $\mathbb{Z}$-SSYT of r-shape $\theta_i \# \theta_j$ with induced ribbon flags $\mathbf{a}^{ij}, \mathbf{b}^{ij}$ (Lemma~\ref{thm:bijection-path-ribbon}). In particular, each path $P_i$ corresponds to some $\mathbb{Z}$-SSYT of shape $\theta_i \# \theta_i = \theta_i$ for all $i$, and thus the system of paths $\mathbf{P} = (P_1,\ldots, P_k)$ corresponds to some tableau $T$ of shape $\lambda/\mu$ (by assembling all tableaux $T_{\theta_1}, \ldots, T_{\theta_k}$, where $T_{\theta_i}$ is a tableau of shape $\theta_i$ corresponding to path $P_i$ for all $i$). Next we use auxiliary lemmas~\ref{lem:path-tableau-1}, \ref{lem:head-dir}, \ref{lem:no-wrong-cells} to prove Lemma~\ref{thm:bijection-super-tableau-non-intersecting-paths}, i.e. to prove that $T$ is a $\mathbb{Z}$-SSYT with flags $\mathbf{a}, \mathbf{b}$ if and only if the paths in $\mathbf{P}$ are non-intersecting. Then using Lemma~\ref{lem:restrictions-for-heads-and-tails-ribbons} we show that the system of paths is non-intersecting if and only if each path $P_i$ in the system is a path from $C_i$ to $D_i$ (Lemma~\ref{lemma:bad-paths-intersect}), which will be necessary to further apply the LGV lemma. 

In \textsection\,\ref{subsec:enumerator} we prove theorems~\ref{thm:signed-expansion} and \ref{thm:HG-1}. We first define weighted lattice induced by the $\mathbb{Z}$-lattice, which we call {\it super lattice}. The starting and ending nodes in the super lattice are the same as in the $\mathbb{Z}$-lattice. We first establish a weight preserving bijection between paths from $C_i$ to $D_j$ on the super lattice and super tableaux 
of r-shape $\theta_i \# \theta_j$ and flags $\mathbf{a}^{ij}, \mathbf{b}^{ij}$ (Lemma~\ref{lem:bt-path-bfr}). Next we establish a weight preserving bijection between systems of paths on the super lattice and super tableaux 
of the given r-shape and flags (Lemma~\ref{lem:bijection-bt-bl}). The results from \textsection\,\ref{subsec:bijection} significantly simplify the proofs of lemmas~\ref{lem:bt-path-bfr}, \ref{lem:bijection-bt-bl} (which is the reason of introducing them first). Next we consider specific ribbon decomposition, 
where each $\theta_i$ is a vertical ribbon, and we show that the enumerator of the paths from $C_i$ to $D_j$ can be written in terms of elementary supersymmetric functions, which in combination with Lemma~\ref{lem:bijection-bt-bl} and the LGV lemma proves Theorem~\ref{thm:signed-expansion}. Next, using lemmas~\ref{lem:bt-path-bfr}, \ref{lem:bijection-bt-bl}, Theorem~\ref{thm:signed-expansion} in combination with the LGV lemma we prove the main Hamel-Goulden formulas from Theorem~\ref{thm:HG-1}. 


\subsection{$\mathbb{Z}$-lattice}
\label{subsec:lattice}
In this subsection we construct a lattice graph from the r-shape $\lambda/\mu = (\lambda/\mu,\ r)$, its outer decomposition $\Theta$ (more 
precisely, we use the corresponding pipe vector $\pi$) 
and flags $\mathbf{a}, \mathbf{b}$. For a pipe vector $\pi = (\pi_1, \ldots, \pi_m)$, let us denote $\pi(i) := \pi_{i - r + 1}$. We define the lattice graph by showing all of its directed edges. Let $\Delta = r - c_{\mathrm{min}}$. For a content $c \in \mathbb{Z}$, define the value 
\[
f_c := 
\begin{cases}
    0,\quad &\text{if}\quad c=\Delta, \\
    |\{i : \Delta < i \le c,\ \pi(i) = (1, \cdot)\}|, \quad &\text{if}\quad c>\Delta, \\
    -|\{i : c\le i < \Delta,\ \pi(i) = (\cdot, 1)\}|, \quad &\text{if}\quad c<\Delta.
\end{cases}
\]

Let $c \in \mathbb{Z}$ be some content in $\lambda/\mu$ and $m_c := \min(c)$, $M_c := \max(c)$, where $\min(c), \max(c)$ denote the leftmost and rightmost cells (resp.) with the content $c$ in the given diagram, 
and let
\begin{align*}
	\overline{a}_c := a_{\mathrm{col}(m_c)} - (\mu'_{\mathrm{col}(m_c)} + 1) + \mathrm{row}(m_c), \qquad 
	\overline{b}_c := b_{\mathrm{col}(M_c)}-\lambda'_{\mathrm{col}(M_c)}+\mathrm{row}(M_c).
\end{align*}

We now construct the lattice graph. 
We call some edges of the lattice \textit{horizontal} and others \textit{vertical}. 
\begin{definition}
    Given r-shape $\lambda/\mu = (\lambda/\mu,\ r)$, pipe vector $\pi$ and flags $\mathbf{a}, \mathbf{b}$, we define the \textit{flagged $\mathbb{Z}$-lattice} $ZL_{\lambda/\mu}(\pi, \mathbf{a}, \mathbf{b})$ as a directed graph with some nodes of $\mathbb{Z}^2$ and the set $E$ of directed edges given as follows.
    Let $m = c_{\mathrm{max}} - c_{\mathrm{min}} + 1$. For all $i\in [r,\ r + m - 1]$, the following edges are in $E$: 
\begin{enumerate}
    \item $((i,j)\rightarrow (i+1,j))$ for 
    $j\in [\overline{a}_i+f_i,\ \overline{b}_i+f_i]$, horizontal  
    \item $((i,j)\rightarrow (i,j+1))$ for $j\in \mathbb{Z}$  
    if $\pi(i) = (0, \cdot)$, vertical 
    \item $((i,j)\rightarrow (i,j-1))$ for $j\in \mathbb{Z}$ 
    if $\pi(i) = (1, \cdot)$, vertical  
\end{enumerate}
In addition, we construct the rightmost vertical edges, let $i=r+m-1$:
\begin{enumerate}
    \item $((r + m, j)\rightarrow (r + m, j+1))$ 
    for $j\in \mathbb{Z}$ 
    if $\pi(i) = (0, \cdot),$
    \item $((r + m, j)\rightarrow (r + m, j-1))$ for 
    $j\in \mathbb{Z}$
    if $\pi(i) = (1, \cdot)$.
\end{enumerate} 
\end{definition}

This lattice graph can be divided into regions of two types: \textit{$e$-region} and \textit{$h$-region}. A node $v$ with coordinates $(x, y)$ is in $e$-region if $\pi(x) = (0, \cdot)$, and it is in $h$-region if $\pi(x) = (1, \cdot)$.

To consider paths on the lattices above we define positions of starting points 
$\mathbf{C}=(C_1,...,C_k)$ and ending points $\mathbf{D}=(D_1,...,D_k)$, where $k$
is the number of ribbons in the given outer ribbon decomposition $\Theta=(\theta_1,...,\theta_k)$.
Let $\delta_i$ be the head and $\gamma_i$ be the tail of the ribbon $\theta_i$, and 
let $v_i,w_i$ be the pipes assigned to cells $\delta_i,\gamma_i$, respectively.
The points $C_i$ and $D_i$ are defined by the ribbon $\theta_i$ in the following way:
\begin{align*}
\begin{cases}
	C_i=(c(\delta_i), -\infty) \quad &\text{ if } v_i = (0, \cdot) \\
	C_i=(c(\delta_i), \infty) \quad &\text{ if } v_i = (1, \cdot) \\
	D_i=(c(\gamma_i)+1, \infty) \quad &\text{ if } w_i = (\cdot, 0) \\
	D_i=(c(\gamma_i)+1, -\infty) \quad &\text{ if } w_i = (\cdot, 1)
\end{cases}
\end{align*}

Let $P_i$ be a path from $C_i$ to $D_i$. We consider 
systems of paths $\mathbf{P} = (P_1,...,P_k)$. We say that the system 
$\mathbf{P}$ is \textit{non-intersecting} if  
$P_i,P_j$ are vertex disjoint for all $i \ne j$. It turns out that 
if the system $\mathbf{P}$ is non-intersecting, 
then it corresponds to some 
flagged $\mathbb{Z}$-SSYT $T$ from $ZT_{\lambda/\mu}(\mathbf{a}, \mathbf{b})$. 

Let us denote  by $ZP_{\lambda/\mu}(\pi, \mathbf{a}, \mathbf{b})$ the set of non-intersecting path systems 
$\mathbf{P} = (P_1, \ldots, P_k)$, where $P_i$ is a path from $C_i$ to 
$D_i$.

\begin{example}
	\label{ex:z-tableau-and-lattice}
    Let $\lambda=(5,5,4,4,2)$, $\mu = (2,1)$, with flags $\mathbf{a} = (1, 0, -1, -2, -3)$ and $\mathbf{b} = (11, 11, 10, 10, 8)$, and  
    the pipe vector: 
    $$\pi=((0,0), (0,1), (1,1), (1,0), (0,1), (1,1), (1,1), (1,1), (1,1)).$$
    The corresponding outer decomposition, example of a flagged $\mathbb{Z}$-SSYT,  
    flagged $\mathbb{Z}$-lattice and system of paths for the tableau are shown in Figures 
    \ref{pic:young_ribbon}, \ref{fig:ex-z-tableau} 
    and \ref{fig:ex-ribbon-lattice}.
\end{example}

\begin{figure}
    \centering
	\begin{subfigure}[b]{0.5\textwidth}
	\centering
		\begin{tikzpicture}[y=-1cm, scale=0.57]
		\draw (3, 1) rectangle (4, 2);
		\draw [red, very thick] (3.5, 2) -- (3.5, 1.5);
		\draw [red, very thick] (3.5, 1.5) -- (3.5, 1);
		\draw (4, 1) rectangle (5, 2);
		\draw [red, very thick] (4.5, 2) -- (4.5, 1.5);
		\draw [red, very thick] (4.5, 1.5) -- (4.5, 1);
		\draw (5, 1) rectangle (6, 2);
		\draw [red, very thick] (5.5, 2) -- (5.5, 1.5);
		\draw [red, very thick] (5.5, 1.5) -- (5.5, 1);
		\draw (2, 2) rectangle (3, 3);
		\draw [red, very thick] (2, 2.5) -- (2.5, 2.5);
		\draw [red, very thick] (2.5, 2.5) -- (2.5, 2);
		\draw (3, 2) rectangle (4, 3);
		\draw [red, very thick] (3.5, 3) -- (3.5, 2.5);
		\draw [red, very thick] (3.5, 2.5) -- (3.5, 2);
		\draw (4, 2) rectangle (5, 3);
		\draw [red, very thick] (4.5, 3) -- (4.5, 2.5);
		\draw [red, very thick] (4.5, 2.5) -- (4.5, 2);
		\draw (5, 2) rectangle (6, 3);
		\draw [red, very thick] (5.5, 3) -- (5.5, 2.5);
		\draw [red, very thick] (5.5, 2.5) -- (5.5, 2);
		\draw (1, 3) rectangle (2, 4);
		\draw [red, very thick] (1.5, 4) -- (1.5, 3.5);
		\draw [red, very thick] (1.5, 3.5) -- (1.5, 3);
		\draw (2, 3) rectangle (3, 4);
		\draw [red, very thick] (2.5, 4) -- (2.5, 3.5);
		\draw [red, very thick] (2.5, 3.5) -- (3, 3.5);
		\draw (3, 3) rectangle (4, 4);
		\draw [red, very thick] (3, 3.5) -- (3.5, 3.5);
		\draw [red, very thick] (3.5, 3.5) -- (3.5, 3);
		\draw (4, 3) rectangle (5, 4);
		\draw [red, very thick] (4.5, 4) -- (4.5, 3.5);
		\draw [red, very thick] (4.5, 3.5) -- (4.5, 3);
		\draw (1, 4) rectangle (2, 5);
		\draw [red, very thick] (1, 4.5) -- (1.5, 4.5);
		\draw [red, very thick] (1.5, 4.5) -- (1.5, 4);
		\draw (2, 4) rectangle (3, 5);
		\draw [red, very thick] (2.5, 5) -- (2.5, 4.5);
		\draw [red, very thick] (2.5, 4.5) -- (2.5, 4);
		\draw (3, 4) rectangle (4, 5);
		\draw [red, very thick] (3.5, 5) -- (3.5, 4.5);
		\draw [red, very thick] (3.5, 4.5) -- (4, 4.5);
		\draw (4, 4) rectangle (5, 5);
		\draw [red, very thick] (4, 4.5) -- (4.5, 4.5);
		\draw [red, very thick] (4.5, 4.5) -- (4.5, 4);
		\draw (1, 5) rectangle (2, 6);
		\draw [red, very thick] (1, 5.5) -- (1.5, 5.5);
		\draw [red, very thick] (1.5, 5.5) -- (2, 5.5);
		\draw (2, 5) rectangle (3, 6);
		\draw [red, very thick] (2, 5.5) -- (2.5, 5.5);
		\draw [red, very thick] (2.5, 5.5) -- (2.5, 5);
		\end{tikzpicture}
	\subcaption{Ribbon decomposition}
	\label{pic:young_ribbon}
	\end{subfigure}
	\begin{subfigure}[b]{0.45\textwidth}
        \centering
		\resizebox{3cm}{!}{
    	\begin{ytableau}
			\none & \none & *(red!70) -1 & *(green!70) -1 & *(purple!70) 2 \\
			\none & *(orange!70) 0 & *(red!70) 0& *(green!70) 2 & *(purple!70) 3 \\
			*(blue!70) 1 & *(red!70) 3 & *(red!70) 3 & *(green!70) 4 \\
			*(blue!70) 2 & *(red!70) 4& *(green!70) 5 & *(green!70) 7\\
			*(red!70) 4 & *(red!70) 6
    	\end{ytableau}
	}
    \subcaption{Example of a flagged $\mathbb{Z}$-tableau}
    \label{fig:ex-z-tableau}
    \end{subfigure}
	\caption{}
\end{figure}

\begin{figure}
    \centering
    \begin{tikzpicture}[scale=1.25]
        \foreach \x in {-4,...,5} {
            \fill[black] (\x, -2) circle (0pt) node[above] {$\x$};
        }
        \fill[red] (-4, -1) circle (2pt) node[below right] {$C_1$};
        \fill[blue] (-3, -1) circle (2pt) node[below right] {$C_2$};
        \fill[green!50!black] (-1, 12) circle (2pt) node[below right] {$C_3$};
        \fill[orange] (0, -1) circle (2pt) node[below right] {$C_4$};
        \fill[purple] (3, 12) circle (2pt) node[below right] {$C_5$};
        \fill[red] (3, -1) circle (2pt) node[below right] {$D_1$};
        \fill[blue] (-1, -1) circle (2pt) node[below right] {$D_2$};
        \fill[green!50!black] (4, -1) circle (2pt) node[below right] {$D_3$};
        \fill[orange] (1, -1) circle (2pt) node[below right] {$D_4$};
        \fill[purple] (5, -1) circle (2pt) node[below right] {$D_5$};
        \draw[line width=0.7mm, blue, -->-] (-3, -1) -- (-3, 0) -- node[above]{$2$} (-2, 0);
        \draw[line width=0.7mm, blue, -->-] (-2, 0) -- node[above]{$1$} (-1, 0);
        \draw[line width=0.7mm, blue, -->-] (-1, 0) -- (-1, -1);
        \draw[line width=0.7mm, red, -->-] (0, 3) -- node[above]{$3$} (1,3);
        \draw[line width=0.7mm, red, -->-] (1, 1) -- node[above]{$0$} (2,1);
        \draw[line width=0.7mm, red, -->-] (2, 1) -- node[above]{$-1$} (3,1);
        \draw[line width=0.7mm, red, -->-] (3, 1) -- (3, -1);
        \draw[line width=0.7mm, red, -->-] (1,3) -- (1,1);
        \draw[line width=0.7mm, red, -->-] (-1, 3) -- node[above]{$3$} (0,3);
        \draw[line width=0.7mm, red, -->-] (-2, 3) -- node[above]{$4$} (-1,3);
        \draw[line width=0.7mm, red, -->-] (-3, 4) -- node[above]{$6$} (-2,4);
        \draw[line width=0.7mm, red, -->-] (-2, 4) -- (-2,3);
        \draw[line width=0.7mm, red, -->-] (-3, 2) -- (-3,4);
        \draw[line width=0.7mm, red, -->-] (-4, -1) -- (-4,2);
        \draw[line width=0.7mm, red, -->-] (-4, 2) -- node[above]{$4$} (-3,2);

        \draw[line width=0.7mm, green!50!black, -->-] (-1, 12) -- (-1,5);
        \draw[line width=0.7mm, green!50!black, -->-] (-1, 5) -- node[above]{$5$} (0,5);
        \draw[line width=0.7mm, green!50!black, -->-] (0, 5) -- (0,7);
        \draw[line width=0.7mm, green!50!black, -->-] (0,7) -- node[above]{$7$} (1,7);
        \draw[line width=0.7mm, green!50!black, -->-] (1,7) -- (1,5);
        \draw[line width=0.7mm, green!50!black, -->-] (1,5) -- node[above]{$4$} (2,5);
        \draw[line width=0.7mm, green!50!black, -->-] (2,5) -- (2,4);
        \draw[line width=0.7mm, green!50!black, -->-] (2,4) -- node[above]{$2$} (3,4);
        \draw[line width=0.7mm, green!50!black, -->-] (3,4) -- (3,2);
        \draw[line width=0.7mm, green!50!black, -->-] (3,2) -- node[above]{$-1$} (4,2);
        \draw[line width=0.7mm, green!50!black, -->-] (4, 2) -- (4, -1) ;
        \draw[line width=0.7mm, green!50!black, -->-] (4,2) -- (4,1);

        \draw[line width=0.7mm, purple, -->-] (3,12) -- (3,6);
        \draw[line width=0.7mm, purple, -->-] (3,6) -- node[above]{$3$} (4,6);
        \draw[line width=0.7mm, purple, -->-] (4,6) -- node[above]{$2$} (5,6);
        \draw[line width=0.7mm, purple, -->-] (5,6) -- (5,-1);

        \draw[line width=0.8mm, orange, -->-] (0, -1) -- (0, 0);
        \draw[line width=0.8mm, orange, -->-] (0,0) -- node[above]{$0$} (1,0);
        \draw[line width=0.8mm, orange, -->-] (1, 0) -- (1, -1);

        \foreach \x in {-4,...,4} {
        \foreach \y in {-1,...,11} {
            \draw[dotted, thin] (\x,\y) rectangle ++(1,1);
        }
        }
        \foreach \x in {-1  ,...,12} {
            \fill[black] (-4.5,\x) circle (0pt) node[left] {$\x$};
        }
        \fill[blue, opacity=0.05] (-2,-1) rectangle (0, 12);
        \fill[blue, opacity=0.05] (1,-1) rectangle (5, 12);
        \foreach \x in {1,...,9} {
        \fill[black] (-4,\x) circle (1.5pt) node[left] {};
        }
        \foreach \x in {0,...,9} {
            \fill[black] (-3,\x) circle (1.5pt) node[left] {};
        }
        \foreach \x in {0,...,9} {
            \fill[black] (-2,\x) circle (1.5pt) node[left] {};
        }
        \foreach \x in {0,...,10} {
            \fill[black] (-1,\x) circle (1.5pt) node[left] {};
        }
        \foreach \x in {0,...,10} {
            \fill[black] (0,\x) circle (1.5pt) node[left] {};
        }
        \foreach \x in {0,...,10} {
            \fill[black] (1,\x) circle (1.5pt) node[left] {};
        }
        \foreach \x in {1,...,10} {
            \fill[black] (2,\x) circle (1.5pt) node[left] {};
        }
        \foreach \x in {1,...,11} {
            \fill[black] (3,\x) circle (1.5pt) node[left] {};
        }
        \foreach \x in {1,...,11} {
            \fill[black] (4,\x) circle (1.5pt) node[left] {};
        }
        \foreach \x in {1,...,11} {
            \fill[black] (5,\x) circle (1.5pt) node[left] {};
        }
    \end{tikzpicture}
\caption{$\mathbb{Z}$-lattice. The $e$-region is indicated by light-blue color, and $h$-region is indicated by white color. Path of some color corresponds to a filling of a ribbon with the same color in Fig.~\ref{fig:ex-z-tableau}.}
\label{fig:ex-ribbon-lattice}
\end{figure}
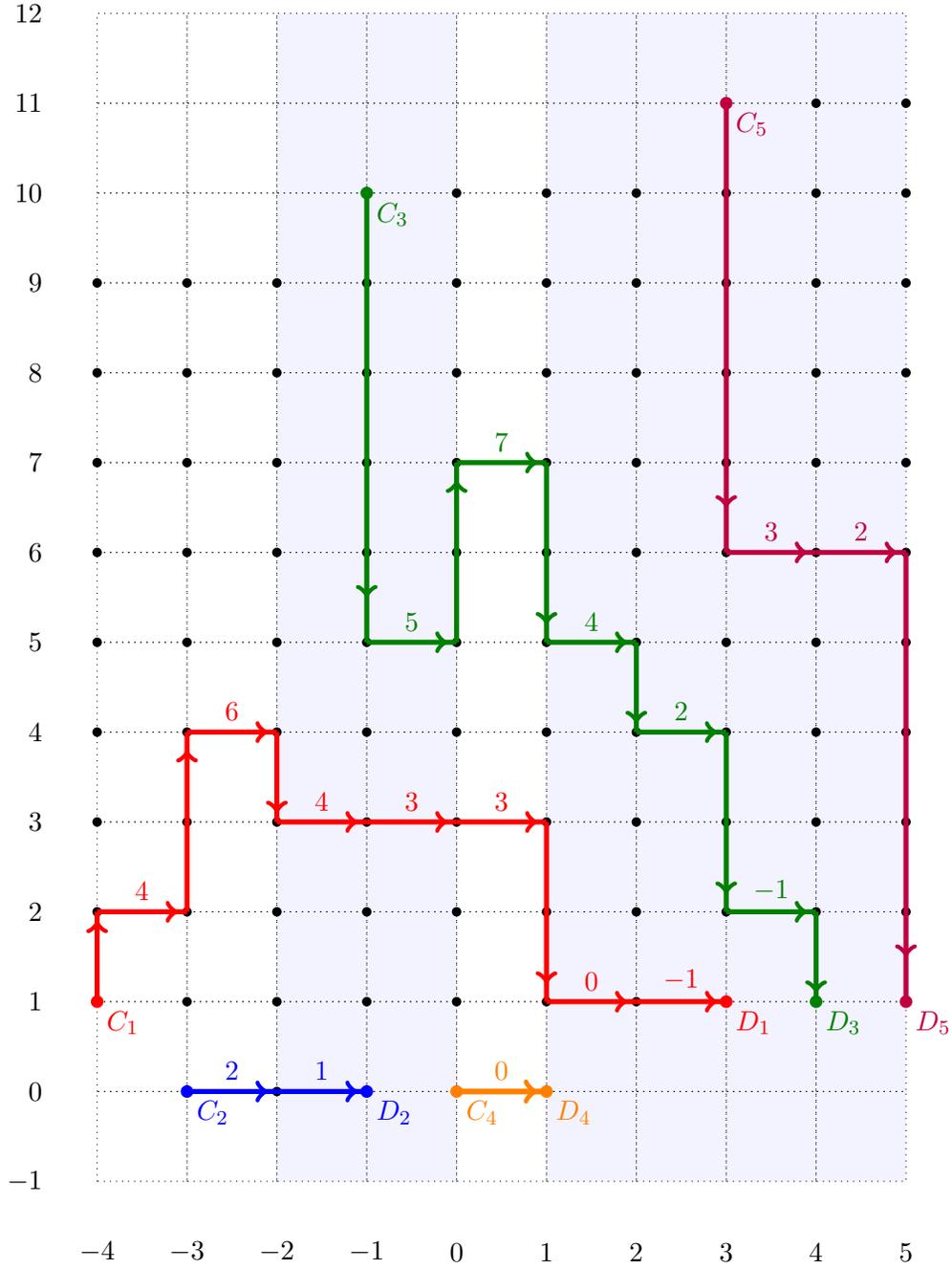

\subsection{Bijection}
\label{subsec:bijection}
In this subsection we establish a bijection 
between the sets $ZT_{\lambda/\mu}(\mathbf{a}, \mathbf{b})$ and 
$ZP_{\lambda/\mu}(\pi, \mathbf{a}, \mathbf{b})$.
Let $P(C_i, D_j)$ be the set of all paths that start from $C_i$ and 
end at $D_j$ in $ZL_{\lambda/\mu}(\pi, \mathbf{a}, \mathbf{b})$. 
Let $ZT_{\lambda/\mu}$ be the set of all tableaux (not necessarily 
semistandard) of r-shape $\lambda/\mu$ filled with integers.
Suppose that $P(C_i, D_j)$ is nonempty and define the map
\[
    \phi_{i,j} : P(C_i, D_j) \rightarrow 
    ZT_{\theta_i\# \theta_j}
\]
as follows: let $P \in P(C_i, D_j)$ and $\gamma$ be a cell in $\theta_i\# \theta_j$, suppose that $P$ contains horizontal edge $e = (c(\gamma), w + f_{c(\gamma)})\rightarrow (c(\gamma)+1, w + f_{c(\gamma)})$ for some $w$, then write $w$ in the cell $\gamma$.
We fill every cell of the $\theta_i\# \theta_j$ in the described way and define $\phi_{i,j}(P)$ as the resulting filling. 


\begin{lemma}
    \label{thm:bijection-path-ribbon} 
    The map 
    $$\phi_{i,j} : P(C_i,D_j)\rightarrow 
    ZT_{\theta_i \# \theta_j}(\mathbf{a}^{ij}, \mathbf{b}^{ij})$$ 
    is a bijection.
\end{lemma}
\begin{proof}
    Let $P \in P(C_i, D_j)$, and $Q = \phi_{i,j}(P)$.
    The filling $Q$ is unique, as for each of its cells with content $c$, 
    there is a unique corresponding edge 
    $p_c = (c, w+f_c)\rightarrow (c+1, w+f_c)$ in $P$. 
    
    Let $\gamma,\ \delta$ be adjacent cells in $Q$ with 
    values $a$ and $b$ respectively. Let $c(\gamma)=c$, then $c(\delta)=c-1$.
    There are two cases: $\gamma$ is an upper neighbor of $\delta$ 
    with $\pi(c) = (1, \cdot)$, 
    or $\gamma$ is a right neighbor of $\delta$ 
    with $\pi(c) = (0, \cdot)$. 
    By definition of the map, $P$ contains edges $(c-1, p) \rightarrow (c, p)$ and $(c, q) \rightarrow (c+1, q)$, where $p = b + f_{c-1}$, $q = a + f_c$. 
    In the first case, by definition of $ZL_{\lambda/\mu}(\pi, \mathbf{a}, \mathbf{b})$, we have $q\le p$, which gives 
    $a<b$. In the second case, by definition of $ZL_{\lambda/\mu}(\pi, \mathbf{a}, \mathbf{b})$, we have $q\ge p$, which gives
    $a\ge b$. We proved that $Q$ is semistandard. 

    Now we need to prove that $Q$ respects the flags. Let $\rho^{ij}_r$ be a vertical ribbon from the canonical decomposition of $\theta_i \# \theta_j$, 
    then from construction of $ZL_{\lambda/\mu}(\pi, \mathbf{a}, \mathbf{b})$ it follows that the maximum value 
    that the bottom cell can (possibly) be filled with is $b^{ij}_r$ 
    and the minimum value that the top cell can be filled with 
    is $a^{ij}_r$, and as the filling is semistandard it follows 
    that all cells of $\rho^{ij}_r$ are filled with values from  
    $[a^{ij}_r, b^{ij}_r]$. So, $Q \in  
    ZT_{\theta_i\# \theta_j}(\mathbf{a}^{ij}, \mathbf{b}^{ij})$. 
    
    We now define the inverse map 
    $$\phi_{i,j}^{-1}: ZT_{\theta_i \# \theta_j}(\mathbf{a}^{ij}, \mathbf{b}^{ij}) \rightarrow 
    P(C_i,D_j).$$ 
    Given $Q \in 
    ZT_{\theta_i\# \theta_j}(\mathbf{a}^{ij}, \mathbf{b}^{ij})$, 
    we construct a unique path $P = \phi_{i,j}^{-1}(Q)$ from $C_i$ to $D_j$ 
    as follows: Suppose that a  
    cell $\epsilon$ with content $r$ is filled with value $w$. Then $P$ contains 
    edge $p_r=(r, w+f_r)\rightarrow (r+1, w+f_r)$. 
    Let us check that 
    this edge exists in the $\mathbb{Z}$-lattice. 
    Let $\theta_i\#\theta_j=\rho^{ij}_1\rightarrow...\rightarrow 
    \rho^{ij}_{k_{ij}}$,
    and suppose that the cell $\epsilon$ is in $\rho^{ij}_s$. 
    Then, by definition of 
    flagged $\mathbb{Z}$-SSYT, we have $a^{ij}_s\le 
    w\le b^{ij}_s$, 
    but as $Q$ is semistandard we also have  
    \[
    a^{ij}_s+c(\gamma^{ij}_s)-r\le 
    w\le b^{ij}_s+c(\delta^{ij}_s)-r,
    \]
    where $\delta^{ij}_s, \gamma^{ij}_s$ are the bottom and top cells (resp.) of the 
    vertical ribbon $\rho^{ij}_s$.
    Then having $m_r = \min(r),\ M_r = \max(r)$ 
    from the definition of flags $\mathbf{a}, \mathbf{b}$ it follows that 
    \[
    a_{\mathrm{col}(m_r)} - (\mu'_{\mathrm{col}(m_r)} + 1) + \mathrm{row}(m_r)+f_r\le 
    w+f_r\le b_{\mathrm{col}(M_r)} - \lambda'_{\mathrm{col}(M_r)} + \mathrm{row}(M_r)+f_r,
    \]
    which shows that the edge $p_r$ is in $\mathbb{Z}$-lattice. Also, the path $P$ 
    contains vertical edges, which connect  
    end node of $p_{r-1}$ and starting node of $p_r$, for all $r$.
    Existence and uniqueness of such edges follows 
    from the definition of $ZL_{\lambda/\mu}(\pi, \mathbf{a}, \mathbf{b})$, and 
    that $Q$ is semistandard, thus $P \in P(C_i, D_j)$. Also 
    the path $P$ is unique, as every edge $p_c = (c,w+f_c)\rightarrow 
    (c+1, w+f_c)$ is uniquely defined by the cell of $Q$ of 
    content $c$, which proves that $\phi_{i,j}^{-1}$ is well defined.
    
    Finally, let 
    $P \in P(C_i, D_j)$. It is not hard to see that any horizontal edge $p_c$ in $P$ is contained in $\phi^{-1}_{i,j}(\phi_{i,j}(P))$. Also $\phi^{-1}_{i,j}(\phi_{i,j}(P))$ contains the same vertical edges as $P$, as end node and starting node of horizontal edges $p_{c-1}, p_c$ are connected uniquely, for all $c$, thus we have $\phi^{-1}_{i,j}(\phi_{i,j}(P)) = P$, which proves that $\phi^{-1}_{i,j}$ is indeed the inverse of $\phi_{i,j}$. 
\end{proof}

Let $\mathbf{P} = (P_1,...,P_k)$ be some non-intersecting 
system of paths where $P_i \in 
P(C_i,D_i)$. From Lemma~\ref{thm:bijection-path-ribbon} it follows that the path $P_i$ maps 
to some filling of $\theta_i$ for all $i$, then the system of paths maps to some 
filling of the r-shape $\lambda/\mu$. Formally, define the map 
$$
\phi_Z : ZP_{\lambda/\mu}(\pi, \mathbf{a}, \mathbf{b}) \rightarrow ZT_{\lambda/\mu}
$$
given by
$$
\phi_Z : (P_1,...,P_k) \mapsto (\phi_{1,1}(P_1),...,\phi_{k,k}(P_k)) = T,
$$
where $T$ is constructed by composing  
flagged $\mathbb{Z}$-SSYTs $\phi_{i,i}(P_i)$ of r-shape $\theta_i$ 
for all $i \in [k]$. 
Now we aim to prove that the map 
$$
\phi_Z : ZP_{\lambda/\mu}(\pi, \mathbf{a}, \mathbf{b}) \rightarrow ZT_{\lambda/\mu}(\mathbf{a}, \mathbf{b})
$$
is a bijection 
between systems of non-intersecting paths and flagged $\mathbb{Z}$-tableaux. 
The inverse map 
$$\phi^{-1}_Z: ZT_{\lambda/\mu}(\mathbf{a}, \mathbf{b}) \rightarrow ZP_{\lambda/\mu}(\pi, \mathbf{a}, \mathbf{b})$$ 
is defined 
as follows: given a flagged $\mathbb{Z}$-SSYT $T$ with flags $\mathbf{a}, \mathbf{b}$, decompose it into ribbons $(\theta_i)$. Let $T_{\theta_i}$ be part of $T$ which corresponds to the ribbon $\theta_i$. It is not hard to see that each filling $T_{\theta_i}$ can be seen as $\mathbb{Z}$-SSYT with flags $\mathbf{a}^{ii}$ and $\mathbf{b}^{ii}$ for all $i$. Thus, we can convert all fillings $T_{\theta_i}$ into paths $P_i = \phi^{-1}_{i,i}(T_{\theta_i})$ for all $i$. We set  
$\phi^{-1}_Z(T) = (P_1,...,P_k)$.


\begin{lemma}
    \label{lem:path-tableau-1}
    Let $\mathbf{P} = (P_1,...,P_k)$ such that $P_i \in P(C_i, D_i)$ for all $i$,
    and $T = \phi_Z(\mathbf{P})$. If some paths intersect 
    then $T$ is not semistandard.
\end{lemma}
\begin{proof}
    Assume the contrary: some paths intersect and $T$ is semistandard. Suppose 
    that the paths $P_i, P_j$ first intersect at node $p$ with 
    the coordinates $(c, q + f_{c - 1})$, and suppose that $P_i$ reaches the node $p$ traversing the 
    horizontal edge $(c - 1, q + f_{c - 1}) \rightarrow (c, q + f_{c - 1})$ filling the cell $\gamma$ (of content $c-1$) with the value $q$ (for some integer $q$), 
    and $P_j$ reaches $p$ traversing vertical edge. As $P_j$ intersects with $P_i$, it must 
    fill at least one of the two cells: either a cell $\delta$ with content $c-1$ or $\eta$ with 
    content $c$.
    
    There are two cases: $P_j$ reaches the intersection point from the bottom (i.e. $\pi(c) = (0, \cdot)$) or $P_j$ reaches the intersecting point from 
    above (i.e. $\pi(c) = (1, \cdot)$).
    
    In the former case (Fig.~\ref{fig:intersection-types-a} or 
    Fig.~\ref{fig:intersection-types-b}). If $\gamma$ is located bottom or bottom-right 
    of the cells $\delta$ and $\eta$ (Fig.~\ref{fig:Young-diagrams-2a}), then $\eta$ must 
    exist and must be filled with some value $a$, and 
    $a \ge q$ as 
    there is an intersection (Fig.~\ref{fig:intersection-types-a}). But it is impossible, since by the semistandard property we must have 
    $q > a$
    (notice that $\mathrm{row}(\gamma)>\mathrm{row}(\eta)$ and $\mathrm{col}(\gamma)\ge \mathrm{col}(\eta)$, as $\gamma$ is located bottom 
    or bottom-right of $\eta$).
    
    If $\gamma$ is located upper-left of the cells $\delta$ and $\eta$ 
    (Fig.~\ref{fig:Young-diagrams-2b}), then $\delta$ must exist and must be filled with 
    $b$, and $b < q$ as there is an intersection (Fig.~\ref{fig:intersection-types-b}), but it is impossible, as 
    $\mathrm{row}(\gamma)<\mathrm{row}(\delta)$ and $\mathrm{col}(\gamma)<\mathrm{col}(\delta)$, and due to semistandard property 
    we must have $b > q$.

    Now we focus on the case when $P_j$ reaches the point $p$ from above 
    (Fig.~\ref{fig:intersection-types-c} and Fig.~\ref{fig:intersection-types-d}). 
    Suppose that the cell $\gamma$ is located on the bottom-right of the cells $\delta$ and 
    $\eta$ (Fig.~\ref{fig:Young-diagrams-2c}). Then the cell $\delta$ must exist and must 
    be filled with some value $b>q$ (Fig.~\ref{fig:intersection-types-c}), but it is impossible as $\mathrm{row}(\gamma)>\mathrm{row}(\delta)$ and 
    $\mathrm{col}(\gamma)>\mathrm{col}(\delta)$, so $q>b$ must hold due to semistandard property. 
    Contradiction.

    If the cell $\gamma$ is located on left or upper-left of the cells 
    $\delta$ and $\eta$ (Fig.~\ref{fig:Young-diagrams-2d}), then $\eta$ must exist, 
    and filled with some value $a$ such that $a < q$ as there is an intersection (Fig.~\ref{fig:intersection-types-d}), but it is impossible, 
    since by the semistandard property we must have 
    $a\ge q$.
\end{proof}

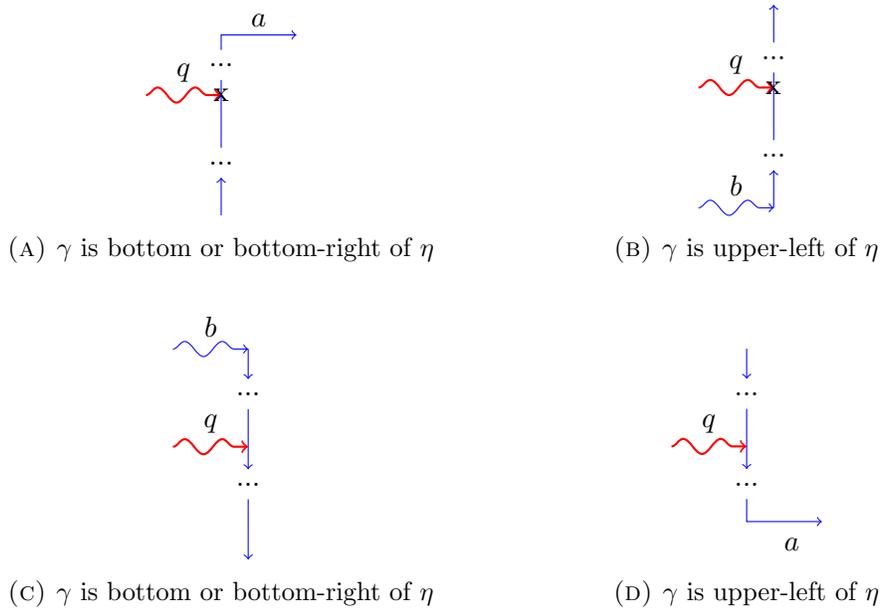
\begin{figure}
    \centering
    \begin{subfigure}[b]{0.5\textwidth}
    \centering
    \begin{tikzpicture}
        \draw[decorate, decoration={snake, amplitude=1mm, segment length=5mm}, ->, red, thick] (0, 1) -- (1, 1);
        \draw[->, blue] (1, 0.3) -- (1, 1);
        \fill[black] (1, 0.1) node {...};
        \draw[->, blue] (1, -0.6) -- (1, -0.1);
        \draw[blue] (1, 1) -- (1, 1.2);
        \fill (1, 1.4) node {...};
        \draw[->, blue] (1, 1.6) -- (1, 1.8) -- (2, 1.8);
        \fill (1.5, 2) node {$a$};
        \fill (0.5, 1.3) node {$q$};
        \fill (1, 1) node {x};
    \end{tikzpicture}
    \subcaption{$\gamma$ is bottom or bottom-right of $\eta$}
    \label{fig:intersection-types-a}
    \end{subfigure}
    \begin{subfigure}[b]{0.4\textwidth}
    \centering
    \begin{tikzpicture}
        \draw[decorate, decoration={snake, amplitude=1mm, segment length=5mm}, ->, red, thick] (0, 1) -- (1, 1);
        \draw[->, blue] (1, 0.3) -- (1, 1);
        \fill[black] (1, 0.1) node {...};
        \draw[->, blue] (1, -0.6) -- (1, -0.1);
        \draw[decorate, decoration={snake, amplitude=1mm, segment length=5mm}, ->, blue] (0, -0.6) -- (1, -0.6);
        \fill (0.5, -0.3) node {$b$};
        \draw[blue] (1, 1) -- (1, 1.2);
        \fill (1, 1.4) node {...};
        \draw[->, blue] (1, 1.6) -- (1, 2.1);;
        \fill (0.5, 1.3) node {$q$};
        \fill (1, 1) node {x};
    \end{tikzpicture}
    \subcaption{$\gamma$ is upper-left of $\eta$}
    \label{fig:intersection-types-b}
    \end{subfigure}\\[0.5cm]
    \begin{subfigure}[b]{0.5\textwidth}
    \centering
    \begin{tikzpicture}
        \draw[->, blue] (1, 0.5) -- (1, 0) -- (1, -0.3);
        \draw[decorate, decoration={snake, amplitude=1mm, segment length=5mm}, ->, red, thick] (0, 0) -- (1, 0);
        \fill (0.5, 0.3) node {$q$};
        \fill (1, 0.7) node {...};
        \draw[decorate, decoration={snake, amplitude=1mm, segment length=5mm}, ->, blue] (0, 1.3) -- (1, 1.3);
        \draw[->, blue] (1, 1.3) -- (1, 0.9);
        \fill (0.5, 1.6) node {$b$};
        \fill (1, -0.5) node {...};
        \draw[->, blue] (1, -0.7) -- (1, -1.5);
    \end{tikzpicture}    
    \subcaption{$\gamma$ is bottom or bottom-right of $\eta$}
    \label{fig:intersection-types-c}
\end{subfigure}
\begin{subfigure}[b]{0.4\textwidth}
    \centering
    \begin{tikzpicture}
        \draw[->, blue] (1, 0.5) -- (1, 0) -- (1, -0.3);
        \draw[decorate, decoration={snake, amplitude=1mm, segment length=5mm}, ->, red, thick] (0, 0) -- (1, 0);
        \fill (0.5, 0.3) node {$q$};
        \fill (1, 0.7) node {...};
        \draw[->, blue] (1, 1.3) -- (1, 0.9);
        \fill (1, -0.5) node {...};
        \draw[->, blue] (1, -0.7) -- (1, -1) -- (2, -1);
        \fill (1.6, -1.3) node {$a$};
    \end{tikzpicture}
    \subcaption{$\gamma$ is upper-left of $\eta$}
    \label{fig:intersection-types-d}
    \end{subfigure}
    \caption{Intersection types}
    \label{fig:intersection-types}
\end{figure}

\begin{figure}[!h]
    \centering
    \begin{subfigure}[b]{0.4\textwidth}
    \centering
    \resizebox{3cm}{!}{
    \begin{ytableau}
            *(red!20) \delta & \eta \\
            \none \\
            \none & \none & \none[\ddots] \\
            \none & \none & \none & \none & \gamma
    \end{ytableau}
    }
    \subcaption{}
    \label{fig:Young-diagrams-2a}
    \end{subfigure}
    \begin{subfigure}[b]{0.4\textwidth}
    \centering
    \resizebox{3cm}{!}{
    \begin{ytableau}
            \none & \gamma \\
            \none \\
            \none & \none & \none[\ddots] \\
            \none & \none & \none & \none & \delta & *(red!20) \eta
    \end{ytableau}
    }
    \subcaption{}
    \label{fig:Young-diagrams-2b}
    \end{subfigure} \\[1cm]
    \begin{subfigure}[b]{0.4\textwidth}
    \centering
    \resizebox{3cm}{!}{
    \begin{ytableau}
            *(red!20) \eta\\
            \delta \\
            \none \\
            \none & \none & \none[\ddots] \\
            \none & \none & \none & \none & \gamma
    \end{ytableau}
    }
    \subcaption{}
    \label{fig:Young-diagrams-2c}
    \end{subfigure}
    \begin{subfigure}[b]{0.4\textwidth}
    \centering
    \resizebox{3cm}{!}{
    \begin{ytableau}
            \none & \gamma \\
            \none \\
            \none & \none & \none[\ddots] \\
            \none & \none & \none & \none & \eta \\
            \none & \none & \none & \none  & *(red!20)\delta
    \end{ytableau}
    }
    \subcaption{}
    \label{fig:Young-diagrams-2d}
    \end{subfigure} \\[0.5cm]

    \caption{}
    \label{fig:Young-diags-2}
\end{figure}

\begin{lemma}
    \label{lem:head-dir}
    Let $\theta_i, \theta_j$ be ribbons with the heads $\delta_i, \delta_j$ in the outer ribbon decomposition $\Theta$ such that 
    there are cells $\xi \in \theta_i, \zeta\in \theta_j$, where $\xi$ is located to the 
    upper-left of the cell $\zeta$, i.e. $\mathrm{row}(\xi)=\mathrm{row}(\zeta)-1, \mathrm{col}(\xi)=\mathrm{col}(\zeta)-1$. 
    If $c(\delta_i)<c(\delta_j)$ then $\pi({c(\delta_j)}) = (1, \cdot)$, and if  
    $c(\delta_i)>c(\delta_j)$ then $\pi({c(\delta_i)}) = (0, \cdot)$. 
\end{lemma}
\begin{proof}
    Let us prove the case when $c(\delta_i)<c(\delta_j)$. 
    Suppose the statement is not true, i.e. $\pi({c(\delta_j)}) = (0, \cdot)$, 
    which means the cell $\delta_j$ is on the left perimeter of the diagram. As $\xi$ is located to the 
    upper-left of the cell $\zeta$ and $\theta_i, \theta_j$ do not intersect, the cell $\delta_j$ must be lower than $\delta_i$, i.e. $\mathrm{row}(\delta_j) > \mathrm{row}(\delta_i)$. 
    But as $c(\delta_i)<c(\delta_j)$, $\mathrm{col}(\delta_j) > \mathrm{col}(\delta_i)$ must hold,  
    which is impossible due to the fact that for all cells $\eta$ with $\mathrm{row}(\eta) < \mathrm{row}(\delta_j)$, we have $\mathrm{col}(\eta) \ge \mathrm{col}(\delta_j)$ by definition of diagram. In other words, all cells above $\delta_j$ must lie either directly above it or in an upper-right position to $\delta_j$, as $\delta_j$ is on the left perimeter.
    Now for the case $c(\delta_i)>c(\delta_j)$, suppose that $\pi(c(\delta_i)) = (1, \cdot)$ (i.e. $\delta_i$ on the bottom perimeter of the diagram). As the ribbons $\theta_i, \theta_j$ do not intersect, we have cells $\xi, \zeta$ with common content, and $c(\delta_i) > c(\delta_j)$. It follows that $\theta_j$ contains the cell $\eta$ of content $c(\delta_i)$ such that $\eta$ is below and to the right of $\delta_i$, which contradicts the definition of diagram. 
\end{proof}

\begin{lemma}
    \label{lem:no-wrong-cells}
    Let $\mathbf{P} = (P_1,...,P_k)$ where $P_i \in P(C_i, D_i)$ for all $i \in [k]$,
    and $T = \phi_Z(\mathbf{P})$. 
    Let $\xi, \zeta$ be cells located in the ribbons $\theta_i,\theta_j$ 
    (with the heads $\delta_i,\delta_j$), respectively. If one of the following 
    statements holds then some paths intersect:
    \begin{enumerate}
        \item $\xi$ is the left neighbor of $\zeta$ and $T_{\xi}>T_{\zeta}$;
        \item $\xi$ is the upper neighbor of $\zeta$ and $T_{\xi}\ge T_{\zeta}$.
    \end{enumerate}
\end{lemma}
\begin{proof}
    Let us show (1). If 
    $\zeta$ is the head of $\theta_j$, then there will be intersection, which is 
    shown in Fig.~\ref{fig:intersection-types-2a},  
    where $a=T_{\xi}, b=T_{\zeta}$. Now suppose that $\zeta$ is not the head of 
    $\theta_j$, which means that there is a cell $\eta$ under $\zeta$. Suppose that 
    $c(\delta_i)<c(\delta_j)$, then by Lemma~\ref{lem:head-dir}, $\pi({c(\delta_j)}) = (1, \cdot)$, 
    which means $C_j$ will be placed on the top. Then as 
    $c(\delta_i)<c(\delta_j)$ it follows that $C_i$ should be placed to the left of $C_j$ (i.e. $C_i^1 < C_j^1$ for their first coordinates). But then the path $P_i$ should intersect somewhere the path 
    $P_j$ as $a>b$ (see Fig.~\ref{fig:head-inter-1a} when $C_i$ is located at the bottom). 
    If $c(\delta_i)>c(\delta_j)$,  then by Lemma~\ref{lem:head-dir}, $\pi({c(\delta_i)}) = (0, \cdot)$, 
    which means $C_i$ will be placed at the bottom. Then as 
    $c(\delta_i)>c(\delta_j)$ it follows that $C_j$ is placed to the left of $C_i$ (i.e. $C_j^1 < C_i^1$). But then the path $P_j$ should intersect somewhere the path $P_i$ 
    as $a>b$ (see Fig.~\ref{fig:head-inter-1b} when $C_j$ located at the bottom). 
    
    In (2), by arguments similar to those in the previous case, the paths $P_i, P_j$ inevitably intersect.
    The reason why strict inequality changes to weak is that $f_{c(\xi)} = f_{c(\zeta)}$; which follows from the fact that $\xi, \zeta$ in different ribbons and $\pi(c(\xi)) = (0, \cdot)$. 
    The edges that fill cells $\zeta, \xi$ are $e = (c(\zeta), b + f_{c(\zeta)}) \rightarrow (c(\zeta) + 1, b + f_{c(\zeta)}),\ e' = (c(\xi), a + f_{c(\zeta)}) \rightarrow (c(\xi) + 1, a + f_{c(\zeta)})$; it follows that the paths $P_i, P_j$ instersect if $a = b$, as the edges $e, e'$ intersect. 
\end{proof}

\begin{figure}
    \centering
    \begin{subfigure}[b]{0.4\textwidth}
    \centering
    \begin{tikzpicture}
        \draw[->, red] (1, 0.5) -- (1, 0) -- (1, -0.3);
        \draw[->, blue] (0, 0) -- (1, 0);
        \fill (0.5, 0.2) node {$a$};
        \fill (1, 0.7) node {...};
        \draw[->, red] (1, 1.3) -- (1, 0.9);
        \fill (1, -0.5) node {...};
        \draw[->, red] (1, -0.7) -- (1, -1) -- (2, -1);
        \fill (1.6, -1.3) node {$b$};
    \end{tikzpicture}
    \subcaption{}
    \label{fig:intersection-types-2a}
    \end{subfigure}
    \begin{subfigure}[b]{0.4\textwidth}
    \centering
    \begin{tikzpicture}
        \draw[->, red] (0, 1) -- (1, 1);
        \draw[->, blue] (1, 0.3) -- (1, 1);
        \fill[black] (1, 0.1) node {...};
        \draw[->, blue] (1, -0.6) -- (1, -0.1);
        \draw[blue] (1, 1) -- (1, 1.2);
        \fill (1, 1.4) node {...};
        \draw[->, blue] (1, 1.6) -- (1, 1.8) -- (2, 1.8);
        \fill (1.5, 2) node {$a$};
        \fill (0.5, 1.2) node {$b$};
    \end{tikzpicture}
    \subcaption{}
    \label{fig:intersection-types-2c}
    \end{subfigure}
    \caption{}
\end{figure}

\begin{figure}
    \centering
    \begin{subfigure}[b]{0.5\textwidth}
        \centering
        \resizebox{4.5cm}{!}{
        \begin{tikzpicture}
            \draw (1,2) grid (3, 5);
            \draw[->, line width=0.5mm, blue] (1, 3) -- (2, 3);
            \foreach \i in {-1,...,5}
                \fill[black] (1, \i) circle (2pt) node[left] {};
            \fill (2, 1.5) node {...};
            \draw (1, -1) grid (3, 1);
            \draw[red, line width=0.5mm] plot [smooth] coordinates {(-2, 5) (-1, 4.3) (0.5, 1.1) (2, 0)};
            \draw[blue, line width=0.5mm] plot [smooth] coordinates {(-3, -1) (-2, 4) (-0.5, 3.5) (1, 3)};
            \draw[->, line width=0.5mm, red] (2, 0) -- (3, 0);
            \fill (1, -1) circle node[left] {$b$};
            \fill (1, 0) circle node[left] {$b+1$};
            \fill (1, 1) node[left] {$b+2$};
            \fill (1, 2) node[left] {$a-1$};
            \fill (1, 3) node[left] {$a$};
            \fill (1, 4) node[left] {$a+1$};
            \fill (1, 5) node[left] {$a+2$};
            \fill[blue] (-3, -1) node[below] {$C_i$};
            \fill[red] (-2, 5) node[above] {$C_{j}$};
        \end{tikzpicture}
        }
        \caption{$c(\delta_i)<c(\delta_j)$}
        \label{fig:head-inter-1a}
    \end{subfigure}
    \begin{subfigure}[b]{0.49\textwidth}
        \centering
        \resizebox{4.5cm}{!}{
        \begin{tikzpicture}
            \draw (1,2) grid (3, 5);
            
            \draw[->, line width=0.5mm, blue] (1, 3) -- (2, 3);
            \foreach \i in {-1,...,5}
                \fill[black] (1, \i) circle (2pt) node[left] {};
            \fill (2, 1.5) node {...};
            \draw (1, -1) grid (3, 1);
            \draw[red, line width=0.5mm] plot [smooth] coordinates {(-3, -1) (-1, 2.3) (0.5, 1.1) (2, 0)};
            \draw[blue, line width=0.5mm] plot [smooth] coordinates {(-2, -1) (-1, 4.3) (0.5, 3.5) (1, 3)};
            \draw[->, line width=0.5mm, red] (2, 0) -- (3, 0);
            \fill (1, -1) circle node[left] {$b$};
            \fill (1, 0) circle node[left] {$b+1$};
            \fill (1, 1) node[left] {$b+2$};
            \fill (1, 2) node[left] {$a-1$};
            \fill (1, 3) node[left] {$a$};
            \fill (1, 4) node[left] {$a+1$};
            \fill (1, 5) node[left] {$a+2$};
            \fill[red] (-3, -1) node[below] {$C_j$};
            \fill[blue] (-2, -1) node[below] {$C_{i}$};
        \end{tikzpicture}
        }
        \caption{$c(\delta_i)>c(\delta_j)$}
        \label{fig:head-inter-1b}
    \end{subfigure}
    \caption{}
\end{figure}

\begin{lemma}
    \label{thm:bijection-super-tableau-non-intersecting-paths}
    The map 
    $$
    \phi_Z : ZP_{\lambda/\mu}(\pi, \mathbf{a}, \mathbf{b}) \rightarrow ZT_{\lambda/\mu}(\mathbf{a}, \mathbf{b})
    $$
    is a well defined bijection.
\end{lemma}
\begin{proof}
    Let $\mathbf{P} = (P_1,...,P_k) \in ZP_{\lambda/\mu}(\pi, \mathbf{a}, \mathbf{b})$ 
    be non-intersecting path system. 
    We first prove that 
    $T = \phi_Z(\mathbf{P}) \in ZT_{\lambda/\mu}(\mathbf{a}, \mathbf{b})$.
    By Lemma~\ref{lem:no-wrong-cells} it follows that $T$ is semistandard. 
    Now we have to prove that $T$ respects the flags $\mathbf{a}, \mathbf{b}$, 
    i.e. that the values in $j$-th column of $T$ are in the interval 
    $[a_j, b_j]$. From 
    the construction of $ZL_{\lambda/\mu}(\pi, \mathbf{a}, \mathbf{b})$: 
    taking the bottommost cell 
    $\gamma$ in $j$-th column, it is filled by the edge 
    $e = (c(\gamma),p) \rightarrow (c(\gamma)+1,p)$ for some $p$. 
    Then the maximum possible value of $p$ is $b_j + f_{c(\gamma)}$, and so
    maximum value the cell $\gamma$ can be filled with 
    is $b_j$. 
    Using the same idea, it follows that the minimum possible 
    value that the top cell of $j$-th column can be filled with 
    is $a_j$.  
    As $T$ is semistandard, 
    all values in $j$-th column are between $a_j$ and 
    $b_j$, which proves that $T \in ZT_{\lambda/\mu}(\mathbf{a}, \mathbf{b})$.

    Now we need to prove that the inverse map $\phi^{-1}_Z$ is well defined. 
    Let $T \in ZT_{\lambda/\mu}(\mathbf{a}, \mathbf{b})$. Then it follows that
    $\phi^{-1}_Z(T) = (\phi^{-1}_{i,i}(T_{\theta_i}))_i$ is a 
    non-intersecting system of paths, as otherwise by 
    Lemma~\ref{lem:path-tableau-1} we have 
    $(\phi_{i,i}(\phi^{-1}_{i,i}(T_{\theta_i})))_i = T$ is not  
    semistandard, which is a contradiction, and so 
    $\phi^{-1}_Z(T) \in ZP_{\lambda/\mu}(\pi, \mathbf{a}, \mathbf{b})$.
    
    Finally, the map $\phi_Z^{-1}$ is indeed the inverse of $\phi_Z$, which follows 
    from the fact that $\phi^{-1}_{ii}$ is the inverse of $\phi_{ii}$ for 
    all $i$.
\end{proof}

\begin{lemma}
    \label{lem:restrictions-for-heads-and-tails-ribbons}
    Let $\theta_i,\theta_j$ be ribbons from the outer ribbon decomposition 
    $\Theta$ and let $\delta_r, \gamma_r$ be the head and tail (resp.) of $\theta_r$ for $r = i, j$. If $c(\delta_i) < c(\delta_j) < c(\gamma_i)$, $\pi({c(\delta_i)})=\pi({c(\delta_j)})=(0, \cdot)$ and $\pi(c(\gamma_i)) = (\cdot, 1)$, then $c(\gamma_i)>c(\gamma_j)$.
\end{lemma}
\begin{proof}
    The heads $\delta_i, \delta_j$ are on the left perimeter of the r-shape as $\pi({c(\delta_i)})=\pi({c(\delta_j)})=(0, \cdot)$. It follows that $\mathrm{row}(\delta_j) < \mathrm{row}(\delta_i)$ as $c(\delta_i) < c(\delta_j)$ and $\delta_i, \delta_j$ are on the left perimeter. As the ribbons are connected by definition, and $\gamma_i$ is on the top perimeter ($\pi(c(\gamma_i)) = (\cdot, 1)$), it follows that $\theta_i$ divides the shape into two parts: cells which are above the ribbon $\theta_i$, and cells which are below the ribbon $\theta_i$. As the ribbons $\theta_i, \theta_j$ do not intersect, and $\delta_j$ is above $\theta_i$, it follows that $\gamma_j$ is also above $\theta_i$. Hence we have $\mathrm{row}(\gamma_j) \ge \mathrm{row}(\gamma_i)$ and $\mathrm{col}(\gamma_j) < \mathrm{col}(\gamma_i)$, as needed. 
\end{proof}

The next lemma will be useful for applying the LGV lemma in the next subsection.
\begin{lemma}
    \label{lemma:bad-paths-intersect}
    Consider the flagged $\mathbb{Z}$-lattice $ZL_{\lambda/\mu}(\pi, \mathbf{a}, \mathbf{b})$ and outer ribbon decomposition 
    $\Theta=(\theta_1,...,\theta_k)$ of $\lambda/\mu$. Let $\sigma \in S_k$ be a nonidentity permutation. Consider a system of paths $(P_1,...,P_k)$ on $ZL_{\lambda/\mu}(\pi, \mathbf{a}, \mathbf{b})$ such that $P_i$ is a path from $C_i$ to $D_{\sigma(i)}$. Then some paths are intersecting.
\end{lemma}
\begin{proof}
    Let $\theta_i, \theta_j \in \Theta$ such that there are cells $\alpha \in \theta_i, \beta \in \theta_j$ with $c(\alpha) = c(\beta)$. 
    We say that $\theta_i$ is {\it above} $\theta_j$ or $\theta_j$ is {\it below} $\theta_i$ if $\mathrm{row}(\alpha) < \mathrm{row}(\beta)$.
    Let $\delta_i, \gamma_i$ be the head and tail of $\theta_i$. Let $C_i$ be some starting node. We say that {\it it is placed on lower boundary} if $\pi({c(\delta_i)}) = (0, \cdot)$ and {\it on upper boundary} if $\pi({c(\delta_i)}) = (1, \cdot)$. 
    Similarly, we say that an end-node $D_i$ {\it is placed on lower boundary if $\pi({c(\gamma_i)}) = (\cdot, 1)$ and on upper boundary if $\pi({c(\gamma_i)}) = (\cdot, 0)$}. 

    Let $\delta_i,\gamma_i$ be the head and tail (respectively) of the ribbon 
    $\theta_i$, and $C_i = (C_i^1, C_i^2),\ D_i = (D_i^1, D_i^2)$, $i \in [k]$. We assume that 
    $C_i^1\le D_{\sigma(i)}^1$ for all $i\in [k]$, as otherwise 
    there is no path in $P(C_i,D_{\sigma(i)})$. 
    We order ribbons $(\theta_i)$ by $c(\delta_i)$ (i.e. by the coordinates $C_i^1$) in ascending order, i.e. if $i<j$ then $c(\delta_i)<c(\delta_j)$ ($C_i^1<C_j^1$). Let $j\in [k]$ be the smallest number s.t. $j\neq \sigma(j)$ and $r$ be the length of the cycle of $\sigma$ that contains $j$ (note that $\sigma^r(j) = j$). 
    Consider the path $P_j$. Notice that $C_j^1\le D_{\sigma(j)}^1-1$, as otherwise it will contradict the choice of the number $j$. From our assumption and the choice of $j$, it follows that 
    $C_j^1<C_{\sigma(j)}^1<D_{\sigma(j)}^1$. Note that the path started 
    from $C_{\sigma(j)}$ goes to $D_{\sigma^2(j)}$. 
    
    
    Assume that the starting and ending nodes $C_j,C_{\sigma(j)},D_{\sigma(j)}$ 
    are on lower boundary. In this case, $D_{\sigma^2(j)}^1<D_{\sigma(j)}^1$ and 
    $D_{\sigma^2(j)}$ is placed on lower boundary, as otherwise the paths 
    $P_j$ and $P_{\sigma(j)}$ will inevitably intersect 
    (see Fig.~\ref{fig:wrong-paths-intersect-a}). Then it follows that 
    $C_{\sigma^2(j)}$ is on lower boundary too, and it must be to the right of 
    $C_{\sigma(j)}$ (i.e. $C_{\sigma^2(j)}^1 > C_{\sigma(j)}^1$), as otherwise the ribbons $\theta_{\sigma(j)}, \theta_{\sigma^2(j)}$ are impossible by construction: for example, suppose that $C_{\sigma^2(j)}^1$ is placed on lower boundary and $C_{\sigma^2(j)}^1 < C_{\sigma(j)}^1$, then by Lemma~\ref{lem:restrictions-for-heads-and-tails-ribbons} 
    it is not possible to have ribbons $\theta_{\sigma(j)},\theta_{\sigma^2(j)}$ in $\Theta$, as $c(\gamma_{\sigma^2(j)}) < c(\gamma_{\sigma(j)})$. Next, $D_{\sigma^3(j)}$ 
    is placed to the left of $D_{\sigma^2(j)}$ to avoid intersection, and
    so $C_j^1<C_{\sigma^2(j)}^1<D_{\sigma^2(j)}^1<D_{\sigma(j)}^1$. By continuing 
    this procedure inductively one ends up with locating $C_{\sigma^{r-1}(j)}$ on lower boundary, 
    to the right of $C_{\sigma(j)}$ and to the left of $D_{\sigma(j)}$. Then $D_{\sigma^r(j)}=D_j$ must 
    be located to the right of $C_{\sigma^{r-1}(j)}$ and to the left of $D_{\sigma(j)}$ 
    (see Fig.~\ref{fig:wrong-paths-intersect-b}), which in its turn contradicts Lemma~\ref{lem:restrictions-for-heads-and-tails-ribbons} 
    (choosing $\theta_j$ and $\theta_{\sigma(j)}$).
    Suppose that $C_j, C_{\sigma(j)}$ are on lower boundary and $D_{\sigma(j)}$ is on upper boundary; then $\pi(\delta_j) = \pi(\delta_{\sigma(j)}) = (0, \cdot)$ and $\pi(\gamma_{\sigma(j)}) = (\cdot, 1)$. As $C^1_j < C^1_{\sigma(j)} < D^1_{\sigma(j)}$ the ribbon $\theta_j$ is below $\theta_{\sigma(j)}$, thus $c(\gamma_{\sigma(j)}) > c(\gamma_j)$ and $D^1_{\sigma(j)} > D^1_j$; hence $D_{\sigma(j)}$ is on upper boundary. Now consider the path $P_{\sigma^{r - 1}(j)}$; recall that $C^1_j < C^1_{\sigma^{r - 1}(j)}$, thus to avoid intersection with $P_j$, the endpoint $C_{\sigma^{r - 1}(j)}$ must be located on upper boundary. Therefore, $\pi(c(\delta_{\sigma^{r - 1}(j)})) = (1, \cdot)$ and the ribbon $\theta_{\sigma^{r - 1}(j)}$ is below $\theta_j$, thus $\pi(c(\gamma_{\sigma^{r - 1}(j)})) = (\cdot, 0)$ and the endpoint $D_{\sigma^{r - 1}(j)}$ is on upper boundary with $D^1_{\sigma^{r - 1}(j)} < D^1_j$. Continuing this procedure, we get that the endpoints $C_{\sigma^{r - t - 1}}, D_{\sigma^{r - t}}$ for $t \in [r - 3]$ together with $D_{\sigma^2(j)}$ are `trapped' between endpoints $C_{\sigma^{r - 1}(j)}$ and $D_j$ on upper boundary. As $C_{\sigma(j)}$ is on lower boundary, it follows that any path from $C_{\sigma(j)}$ to $D_{\sigma^2(j)}$ intersects the path $P_{\sigma^{r - 1}(j)}$, a contradiction. 

    Suppose that $C_j, D_{\sigma(j)}$ are on lower boundary and $C_{\sigma(j)}$ is on upper boundary. By similar arguments as in the previous case, it can be shown that the endpoint $D_{\sigma^2(j)}$ is `trapped' between the endpoints $C_j, D_{\sigma(j)}$ on lower boundary, and so the paths $P_j$ and $P_{\sigma(j)}$ must intersect, which is a contradiction. 

    Suppose that $C_j$ is on lower boundary and $C_{\sigma(j)}, D_{\sigma(j)}$ are on upper boundary; then $\theta_j$ is above $\theta_{\sigma(j)}$ by similar arguments as in previous cases. The endpoint $D_{\sigma^2(j)}$ must be on upper boundary, as the paths $P_{\sigma(j)}$ and $P_j$ do not intersect; hence the ribbon $\theta_{\sigma^2(j)}$ is below the ribbon $\theta_{\sigma(j)}$, and $C_{\sigma^2(j)}$ is on upper boundary with $C^1_{\sigma^2(j)} > C^1_{\sigma(j)}$. By continuing this procedure, it follows that the endpoint $D_j = D_{\sigma^r(j)}$ is between the endpoints $C_{\sigma(j)}, D_{\sigma(j)}$ on upper boundary, from which it follows that $\pi(c(\gamma_j)) = \pi(c(\gamma_{\sigma(j)})) = (\cdot, 0)$ (i.e. $\gamma_j, \gamma_{\sigma(j)}$ are on the right perimeter of the shape) and $c(\gamma_j) < c(\gamma_{\sigma(j)})$; hence $\mathrm{col}(\gamma_j) < \mathrm{col}(\gamma_{\sigma(j)})$, which contradicts the definition of shape. 

    There are four other cases where $C_j$ in on upper boundary, but these are symmetric to the cases we already considered; more precisely, we can flip vertically the lattice (so that, now $C_j$ is on lower boundary), and apply the corresponding constructions of paths on the lattice that led to contradictions.
\end{proof}

\begin{figure}
    \centering
    \begin{subfigure}[b]{0.5\textwidth}
    \centering
    \resizebox{1\textwidth}{!}{
    \begin{tikzpicture}
        \draw[red, line width=0.5mm] plot [smooth] coordinates {(0,0) (1,2) (2,3) (3,3) (4,2.5) (5,1) (6,0)};
        \draw[blue, line width=0.5mm] plot [smooth] coordinates { (1,0) (2,2) (3, 2) (4, 1) (5, 0)};
        \fill[red] (0,0) circle (2pt) node[below] {$C_j$};
        \fill[red] (6,0) circle (2pt) node[below] {$D_{\sigma(j)}$};
        \fill[blue] (1,0) circle (2pt) node[below] {$C_{\sigma(j)}$};
        \fill[blue] (5,0) circle (2pt) node[below] {$D_{\sigma^2(j)}$};
        \fill[purple] (5, 3) circle (2pt) node[above] {$\xcancel{D_{\sigma^2(j)}}$};
        \fill[purple] (7, 3) circle (2pt) node[above] {$\xcancel{D_{\sigma^2(j)}}$};
        \fill[purple] (7, 0) circle (2pt) node[below] {$\xcancel{D_{\sigma^2(j)}}$};
    \end{tikzpicture}
    }
    \subcaption{}
    \label{fig:wrong-paths-intersect-a}
    \end{subfigure}
    \begin{subfigure}[b]{0.49\textwidth}
    \centering
    \resizebox{1\textwidth}{!}{
    \begin{tikzpicture}
        \draw[red, line width=0.5mm] plot [smooth] coordinates {(0,0) (1,2) (2,3) (3,3) (4,2.5) (5,1) (6,0)};
        \draw[blue, line width=0.5mm] plot [smooth] coordinates { (0.5,0) (2,2.7) (3, 2.7) (4, 2) (5.5, 0)};
        \draw[green!50!black, line width=0.5mm] plot [smooth] coordinates { (1,0) (1.5,1.5) (2.5, 2.5) (3.5,2) (5, 0)};
        \draw[black, line width=0.5mm] plot [smooth] coordinates { (2,0) (2.5, 1.5) (3,2) (3.5,1.5) (4, 0)};
        \fill[red] (0,0) circle (2pt) node[left] {$C_j$};
        \fill[red] (6,0) circle (2pt) node[right] {$D_{\sigma(j)}$};
        \fill[blue] (0.5,0) circle (2pt) node[below] {$C_{\sigma(j)}$};
        \fill[blue] (5.5,0) circle (2pt) node[below] {$D_{\sigma^2(j)}$};
        \fill[green!50!black] (1,0) circle (2pt) node[below] {};
        \fill[green!50!black] (5,0) circle (2pt) node[below] {};
        \fill[green!50!black] (1,-0.5) node[below] {$C_{\sigma^2(j)}$};
        \fill[green!50!black] (5,-0.5) node[below] {$D_{\sigma^3(j)}$};
        \fill[black] (2,0) node[below] {$C_{\sigma^{r-1}(j)}$};
        \fill[black] (4,0) node[below] {$D_{\sigma^{r}(j)}$};
        \fill (1.5, 0.5) node {$...$};
        \fill (4.2, 0.5) node {$...$};
    \end{tikzpicture}
    }
    \subcaption{}
    \label{fig:wrong-paths-intersect-b}
    \end{subfigure}
    \caption{System of paths for case in the proof.}
    \label{fig:wrong-paths-intersect}
\end{figure}
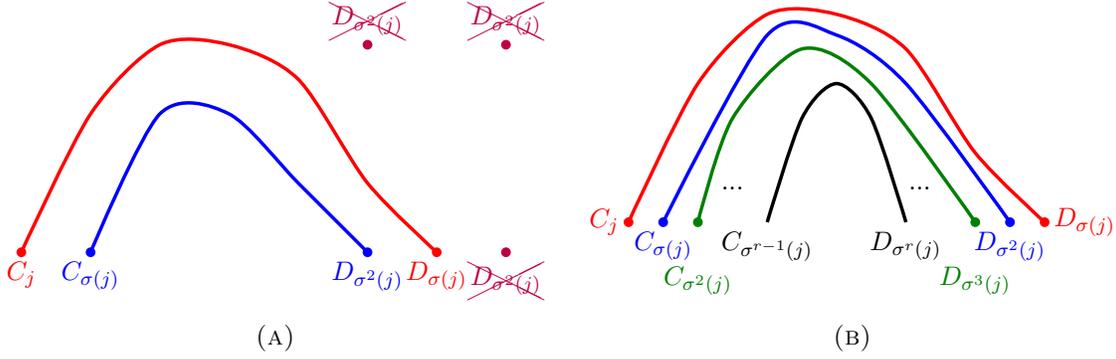

\subsection{Super lattice and enumerators}
\label{subsec:enumerator}

We are going to construct \textit{weighted lattice graph} on the  $\mathbb{Z}$-lattice discussed above. 
\begin{definition}
Given an edge $e$, let us denote its weight as $wt(e)$. Let $E$ be the set of edges of $ZL_{\lambda/\mu}(\pi, \mathbf{a}, \mathbf{b})$, we 
define a new multiset of edges $E'$ as follows: 
let $e = (i, j + f_i) \rightarrow (i+1, j + f_i) \in E$, then the set $E'$ contains 
two copies of $e$: $e_1$ and $e_2$ with different weights $wt(e_1) = y_j
, wt(e_2) = -z_{j+i}$. We call the edges $e_1, e_2$ as \textit{representations} of the edge $e$. 
All vertical edges $e$ in $E$ are in $E'$ as well and their weights are defined to be $1$.
    We define the flagged {\it super lattice} $SL_{\lambda/\mu}(\pi, \mathbf{a}, \mathbf{b})$ 
    as an edge-induced subgraph of $\mathbb{Z}^2$ by the multiset $E'$.
\end{definition}
Let $\Theta = (\theta_1, \ldots, \theta_k)$ be an outer ribbon decomposition of the r-shape 
$\lambda/\mu$. We consider systems of paths on 
$SL_{\lambda/\mu}(\pi, \mathbf{a}, \mathbf{b})$, 
whose starting and ending points $\mathbf{C} = (C_1,\ldots, C_k), \mathbf{D} =  (D_1,\ldots, D_k)$ are placed in 
the same positions as in 
$ZL_{\lambda/\mu}(\pi, \mathbf{a}, \mathbf{b})$ shown before. 
The weight $wt(P)$ of a path $P$ is defined as:
$$
wt(P) = \prod_{\text{edge } e \in P} wt(e).
$$

As any edge $e$ in $E$ is also in $E'$, 
any path $P$ on $ZL_{\lambda/\mu}(\pi, \mathbf{a}, \mathbf{b})$ 
can be represented on $SL_{\lambda/\mu}(\pi, \mathbf{a}, \mathbf{b})$ using the same edges, 
but as some nodes are connected by 
several edges, there might be 
multiple paths $P'_1,...,P'_s$ which represent the $P$ 
on the super lattice. We call the paths $P'_1,...,P'_s$ 
\textit{representations} of 
the path $P$.
(Note that all the paths $P'_i$ are the same, but their weights are different.) 
Consequently, for a system of paths $\mathbf{P}$ on  
$ZL_{\lambda/\mu}(\pi, \mathbf{a}, \mathbf{b})$, there are 
its copies $\mathbf{P}'_1,...,\mathbf{P}'_r$ 
on $SL_{\lambda/\mu}(\pi, \mathbf{a}, \mathbf{b})$, and similarly, 
we call them \textit{representations} of the system $\mathbf{P}$. Let us define the weight $wt(\mathbf{P})$ of system of paths $\mathbf{P} = (P_1,\ldots, P_k)$ as 
$$
wt(\mathbf{P}) = \prod_{i = 1}^k wt(P_i).
$$

Let us recall that we established bijection between the sets 
$ZP_{\lambda/\mu}(\pi, \mathbf{a}, \mathbf{b})$ and $ZT_{\lambda/\mu}(\mathbf{a}, \mathbf{b})$.
Let now $SP_{\lambda/\mu}(\pi, \mathbf{a}, \mathbf{b})$ be the set of systems of 
non-intersecting paths $\mathbf{P}=(P_1,...,P_k)$, where 
$P_i$ is a path from $C_i$ to $D_i$ on $SL_{\lambda/\mu}(\pi, \mathbf{a}, \mathbf{b})$. We are going to construct  
a weight-preserving bijection $\psi$ between the sets 
$SP_{\lambda/\mu}(\pi, \mathbf{a}, \mathbf{b})$ and 
$ST_{\lambda/\mu}(\mathbf{a}, \mathbf{b})$.

\begin{remark} 
    From Lemma~\ref{lemma:bad-paths-intersect} it follows that on the $\mathbb{Z}$-lattice, a system of paths $\mathbf{P} = (P_i)$ is non-intersecting only if $P_i \in P(C_i, D_i)$ for all $i$. This is also true for systems of paths on the super lattice, as every path is some  representation of the same path on the $\mathbb{Z}$-lattice. Thus, the set $SP_{\lambda/\mu}(\pi, \mathbf{a}, \mathbf{b})$ is indeed the set of all possible non-intersecting systems of paths on the super lattice $SL_{\lambda/\mu}(\pi, \mathbf{a}, \mathbf{b})$.
\end{remark}

Let $P_S(C_i, D_j)$ be the set of paths from $C_i$ to $D_j$ on $SL_{\lambda/\mu}(\pi, \mathbf{a}, \mathbf{b})$, and let $ST_{\lambda/\mu}$ be the 
set of all tableaux of shape $\lambda/\mu$ with entries from $\mathbb{Z} \cup \mathbb{Z}'$. 
We define the map 
$$
    \psi_{i,j} : P_S(C_i, D_j) \rightarrow ST_{\theta_i \# \theta_j}
$$
as follows. Let $P\in P_S(C_i, D_j)$. 
We construct the tableau $Q \in 
ST_{\theta_i \# \theta_j}$ as follows: for all contents 
$c \in [c(\delta_i), c(\gamma_j)]$ (where $\delta_i, \gamma_j$ 
are the head and tail of the ribbon $\theta_i \# \theta_j$), 
let $p_c = (c, w + f_c) \rightarrow (c + 1, w + f_c)$ for some 
$w$, be an 
edge in $P$. Then the cell $\zeta$ of $Q$ with the content $c$ 
has the value $w$ if $wt(p_c) = y_w$, or it has the value 
$(w + c)'$ if $wt(p_c) = -z_{w + c}$.

\begin{figure}
    \centering
    \resizebox{4cm}{!}{
    \begin{ytableau}
        \none & \none & *(red!70) 1' & *(green!70) -1 & *(purple!70) 2 \\
        \none & *(orange!70) 0' & *(red!70) 0 & *(green!70) 4' & *(purple!70) 6' \\
        *(blue!50) 1 & *(red!70) 3 & *(red!70) 3 & *(green!70) 5' \\
        *(blue!50) -1' & *(red!70) 2' & *(green!70) 5 & *(green!70) 7\\
        *(red!70) 0' & *(red!70) 3'
    \end{ytableau}
    }
    \caption{Example of a flagged super tableau.}
    \label{fig:ex-s-tableau}
\end{figure}

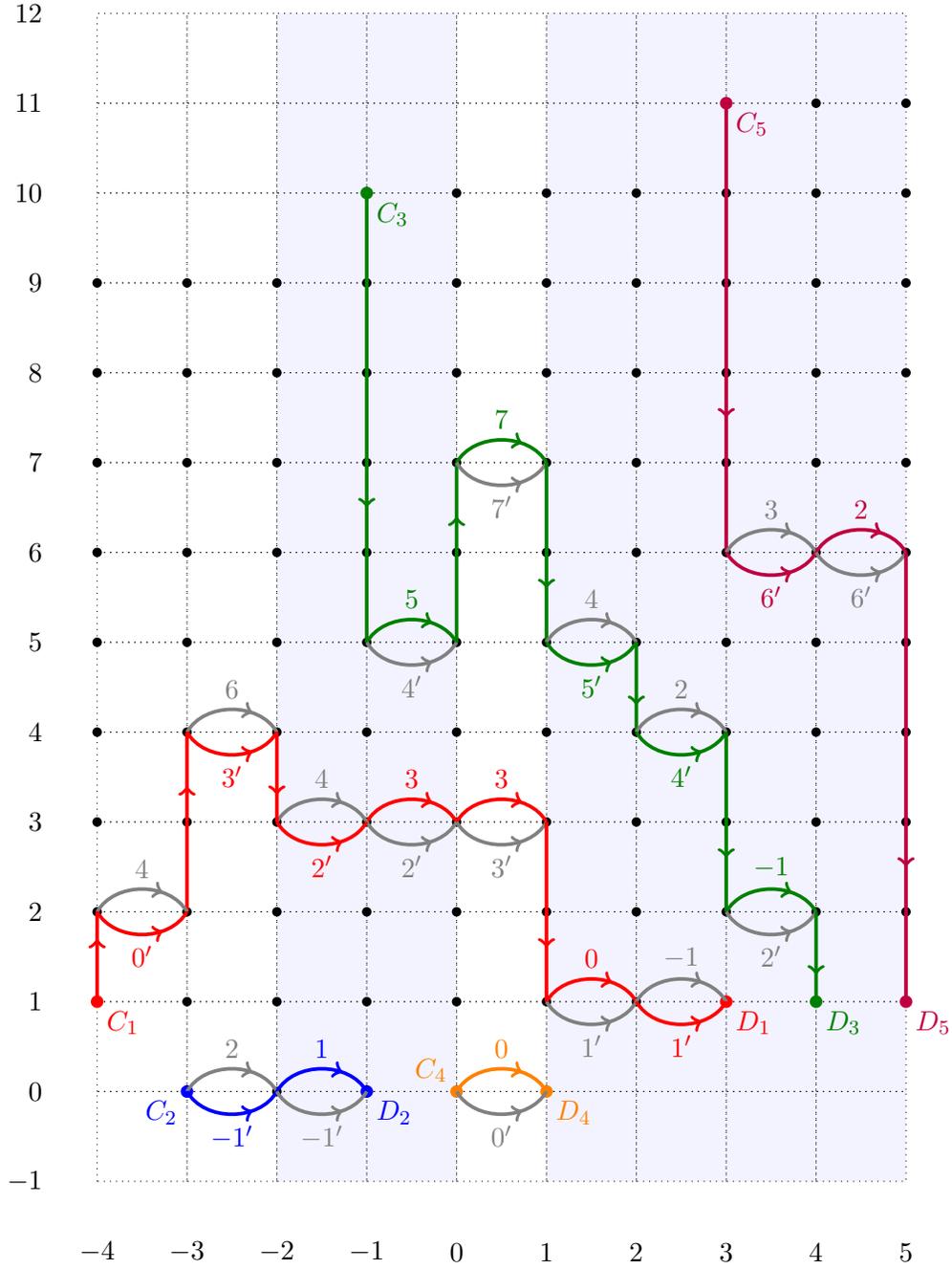
\begin{figure}
    \centering
    \begin{tikzpicture}[scale=1.25]
        \foreach \x in {-4,...,5} {
            \fill[black] (\x, -2) circle (0pt) node[above] {$\x$};
        }
        \fill[red] (-4, -1) circle (2pt) node[below right] {$C_1$};
        \fill[blue] (-3, -1) circle (2pt) node[below left] {$C_2$};
        \fill[green!50!black] (-1, 12) circle (2pt) node[below right] {$C_3$};
        \fill[orange] (0, -1) circle (2pt) node[above left] {$C_4$};
        \fill[purple] (3, 12) circle (2pt) node[below right] {$C_5$};
        \fill[red] (3, -1) circle (2pt) node[below right] {$D_1$};
        \fill[blue] (-1, -1) circle (2pt) node[below right] {$D_2$};
        \fill[green!50!black] (4, -1) circle (2pt) node[below right] {$D_3$};
        \fill[orange] (1, -1) circle (2pt) node[below right] {$D_4$};
        \fill[purple] (5, -1) circle (2pt) node[below right] {$D_5$};
        \draw[line width=0.5mm, gray, ->-] (-3, 0) to[out=60, in=120] node[above]{$2$} (-2, 0);
        \draw[line width=0.5mm, blue, ->-] (-3, -1) -- (-3, 0);
        \draw[line width=0.5mm, blue, ->-] (-1, 0) -- (-1, -1);
        \draw[line width=0.5mm, blue, ->-] (-3, 0) to[out=-60, in=-120] node[below]{$-1'$} (-2, 0);
        \draw[line width=0.5mm, blue, ->-] (-2, 0) to[out=60, in=120] node[above]{$1$} (-1, 0);
        \draw[line width=0.5mm, gray, ->-] (-2, 0) to[out=-60, in=-120] node[below]{$-1'$} (-1, 0);
        \draw[line width=0.5mm, red, ->-] (0, 3) to[out=60, in=120] node[above]{$3$} (1,3);
        \draw[line width=0.5mm, gray, ->-] (0, 3) to[out=-60, in=-120] node[below]{$3'$} (1,3);
        \draw[line width=0.5mm, red, ->-] (1, 1) to[out=60, in=120] node[above]{$0$} (2,1);
        \draw[line width=0.5mm, gray, ->-] (1, 1) to[out=-60, in=-120] node[below]{$1'$} (2,1);
        \draw[line width=0.5mm, gray, ->-] (2, 1) to[out=60, in=120] node[above]{$-1$} (3,1);
        \draw[line width=0.5mm, red, ->-] (3, 1) -- (3, -1);
        \draw[line width=0.5mm, red, ->-] (2, 1) to[out=-60, in=-120] node[below]{$1'$} (3,1);
        \draw[line width=0.5mm, red, ->-] (1,3) -- (1,1);
        \draw[line width=0.5mm, red, ->-] (-1, 3) to[out=60, in=120] node[above]{$3$} (0,3);
        \draw[line width=0.5mm, gray, ->-] (-1, 3) to[out=-60, in=-120] node[below]{$2'$} (0,3);
        \draw[line width=0.5mm, gray, ->-] (-2, 3) to[out=60, in=120] node[above]{$4$} (-1,3);
        \draw[line width=0.5mm, red, ->-] (-2, 3) to[out=-60, in=-120] node[below]{$2'$} (-1,3);
        \draw[line width=0.5mm, gray, ->-] (-3, 4) to[out=60, in=120] node[above]{$6$} (-2,4);
        \draw[line width=0.5mm, red, ->-] (-3, 4) to[out=-60, in=-120] node[below]{$3'$} (-2,4);
        \draw[line width=0.5mm, red, ->-] (-2, 4) -- (-2,3);
        \draw[line width=0.5mm, red, ->-] (-3, 2) -- (-3,4);
        \draw[line width=0.5mm, red, ->-] (-4, -1) -- (-4,2);
        \draw[line width=0.5mm, gray, ->-] (-4, 2) to[out=60, in=120] node[above]{$4$} (-3,2);
        \draw[line width=0.5mm, red, ->-] (-4, 2) to[out=-60, in=-120] node[below]{$0'$} (-3,2);

        \draw[line width=0.5mm, green!50!black, ->-] (-1, 12) -- (-1,5);
        \draw[line width=0.5mm, green!50!black, ->-] (-1, 5) to[out=60, in=120] node[above]{$5$} (0,5);
        \draw[line width=0.5mm, gray, ->-] (-1, 5) to[out=-60, in=-120] node[below]{$4'$} (0,5);
        \draw[line width=0.5mm, green!50!black, ->-] (0, 5) -- (0,7);
        \draw[line width=0.5mm, green!50!black, ->-] (0,7) to[out=60, in=120] node[above]{$7$} (1,7);
        \draw[line width=0.5mm, gray, ->-] (0,7) to[out=-60, in=-120] node[below]{$7'$} (1,7);
        \draw[line width=0.5mm, green!50!black, ->-] (1,7) -- (1,5);
        \draw[line width=0.5mm, gray, ->-] (1,5) to[out=60, in=120] node[above]{$4$} (2,5);
        \draw[line width=0.5mm, green!50!black, ->-] (1,5) to[out=-60, in=-120] node[below]{$5'$} (2,5);
        \draw[line width=0.5mm, green!50!black, ->-] (2,5) -- (2,4);
        \draw[line width=0.5mm, gray, ->-] (2,4) to[out=60, in=120] node[above]{$2$} (3,4);
        \draw[line width=0.5mm, green!50!black, ->-] (2,4) to[out=-60, in=-120] node[below]{$4'$} (3,4);
        \draw[line width=0.5mm, green!50!black, ->-] (3,4) -- (3,2);
        \draw[line width=0.5mm, green!50!black, ->-] (3,2) to[out=60, in=120] node[above]{$-1$} (4,2);
        \draw[line width=0.5mm, gray, ->-] (3,2) to[out=-60, in=-120] node[below]{$2'$} (4,2);
        \draw[line width=0.5mm, green!50!black, ->-] (4,2) -- (4,-1);

        \draw[line width=0.5mm, purple, ->-] (3,12) -- (3,6);
        \draw[line width=0.5mm, gray, ->-] (3,6) to[out=60, in=120] node[above]{$3$} (4,6);
        \draw[line width=0.5mm, purple, ->-] (3,6) to[out=-60, in=-120] node[below]{$6'$} (4,6);
        \draw[line width=0.5mm, purple, ->-] (4,6) to[out=60, in=120] node[above]{$2$} (5,6);
        \draw[line width=0.5mm, gray, ->-] (4,6) to[out=-60, in=-120] node[below]{$6'$} (5,6);
        \draw[line width=0.5mm, purple, ->-] (5,6) -- (5,-1);

        \draw[line width=0.5mm, orange, ->-] (0, -1) -- (0, 0);
        \draw[line width=0.5mm, orange, ->-] (1, 0) -- (1, -1); 
        \draw[line width=0.5mm, orange, ->-] (0,0) to[out=60, in=120] node[above]{$0$} (1,0);
        \draw[line width=0.5mm, gray, ->-] (0,0) to[out=-60, in=-120] node[below]{$0'$} (1,0);

        \foreach \x in {-4,...,4} {
        \foreach \y in {-1,...,11} {
            \draw[dotted, thin] (\x,\y) rectangle ++(1,1);
        }
        }
        \foreach \x in {-1  ,...,12} {
            \fill[black] (-4.5,\x) circle (0pt) node[left] {$\x$};
        }
        \fill[blue, opacity=0.05] (-2,-1) rectangle (0, 12);
        \fill[blue, opacity=0.05] (1,-1) rectangle (5, 12);
        \foreach \x in {1,...,9} {
        \fill[black] (-4,\x) circle (1.5pt) node[left] {};
        }
        \foreach \x in {0,...,9} {
            \fill[black] (-3,\x) circle (1.5pt) node[left] {};
        }
        \foreach \x in {0,...,9} {
            \fill[black] (-2,\x) circle (1.5pt) node[left] {};
        }
        \foreach \x in {0,...,10} {
            \fill[black] (-1,\x) circle (1.5pt) node[left] {};
        }
        \foreach \x in {0,...,10} {
            \fill[black] (0,\x) circle (1.5pt) node[left] {};
        }
        \foreach \x in {0,...,10} {
            \fill[black] (1,\x) circle (1.5pt) node[left] {};
        }
        \foreach \x in {1,...,10} {
            \fill[black] (2,\x) circle (1.5pt) node[left] {};
        }
        \foreach \x in {1,...,11} {
            \fill[black] (3,\x) circle (1.5pt) node[left] {};
        }
        \foreach \x in {1,...,11} {
            \fill[black] (4,\x) circle (1.5pt) node[left] {};
        }
        \foreach \x in {1,...,11} {
            \fill[black] (5,\x) circle (1.5pt) node[left] {};
        }
    \end{tikzpicture}
\caption{Super lattice. The $e$-region is shown in light-blue color, and $h$-region is shown in white color. A path of some color corresponds to a filling of a ribbon with the same color in Fig.~\ref{fig:ex-s-tableau}.}
\label{fig:ex-ribbon-lattice-super}
\end{figure}

\begin{example}
    Consider the $\mathbb{Z}$-tableau $T'$ and the system of paths $\mathbf{P}'$ from Example~\ref{ex:z-tableau-and-lattice}. We give a representation of $T'$ in Fig.~\ref{fig:ex-s-tableau} and a representation of $\mathbf{P}'$ in the super lattice in Fig.~\ref{fig:ex-ribbon-lattice-super}.
\end{example}

\begin{lemma}
    \label{lem:bt-path-bfr}
    The map 
    $$
    \psi_{i,j} : P_S(C_i, D_j) \rightarrow ST_{\theta_i \# \theta_j}
    (\mathbf{a}^{ij}, \mathbf{b}^{ij})
    $$ is a weight-preserving bijection.
\end{lemma}
\begin{proof}
    Let $P\in P_S(C_i, D_j)$ be a representation of a path 
    $P' \in P(C_i, D_j)$.  
    Let $Q' = \phi_{i,j}(P')$ and $Q = \psi_{i,j}(P)$. Then for all 
    contents $c \in [c(\delta_i), c(\gamma_j)]$: let an edge 
    $p'_c = (c, q + f_c) \rightarrow (c + 1, q + f_c)$ be in $P'$ for some $q$, and it follows that the 
    cell $\zeta$ of $Q'$ with  
    content $c$ contains the value $q$, i.e. $Q'_{\zeta} = q$. Let $p_c = (c, q + f_c) \rightarrow (c + 1, q + f_c) \in P$ be a representation of the edge $p'_c$ on the super lattice, then 
    by definition of $\psi_{i,j}$, we have $Q_{\zeta} = q$ if $wt(p_c) = y_q$, and $Q_{\zeta} = 
    (q + c)'$ if $wt(p_c) = -z_{q + c}$. Notice that $wt(Q_{\zeta}) = wt(p_c)$. It follows that $Q$ is a representation 
    of $Q'$, and so $Q \in ST_{\theta_i \# \theta_j}(\mathbf{a}^{ij}, \mathbf{b}^{ij})$, and moreover $wt(P) = wt(Q)$.

    Now we define the inverse function $\psi_{i,j}^{-1}$. 
    Let $Q \in ST_{\theta_i \# \theta_j}(\mathbf{a}^{ij}, \mathbf{b}^{ij})$ be a 
    representation of $Q' \in ZT_{\theta_i \# \theta_j}(\mathbf{a}^{ij}, \mathbf{b}^{ij})$. Let $P' = \phi_{i,j}^{-1}(Q')$ 
    and we choose $P := \psi_{i,j}^{-1}(Q)$ to be a representation 
    of $P'$ such that $wt(P) = wt(Q)$. The path $P$ exists and 
    it is unique as 
    for any content $c \in [c(\delta_i), c(\gamma_j)]$, there 
    is a unique edge $p_c = (c, q + f_c) \rightarrow (c + 1, q + f_c)$ 
    with the weight $wt(p_c) = wt(Q_{\zeta})$ on  
    $SL_{\lambda/\mu}(\pi, \mathbf{a}, \mathbf{b})$, where 
    $\zeta$ is a cell in $\theta_i \# \theta_j$ with the content $c$ and  
    $q = Q'_{\zeta}$. Notice that $p_c$ is also a representation of an edge $p'_c = (c, q + f_c) \rightarrow (c + 1, q + f_c)$ inside $P'$, and so $\psi^{-1}_{i,j}(Q)$ is a representation of $P'$, which gives that $\psi^{-1}_{i,j}(\psi_{i,j}(P)) = P$.
\end{proof}

Let $\mathbf{P} = (P_1,...,P_k) \in SP_{\lambda/\mu}(\pi, \mathbf{a}, \mathbf{b})$.
We can convert 
every path $P_i$ into a super tableau $T_{\theta_i} = \psi(P_i)$ 
of r-shape $\theta_i$. Assembling all tableaux $T_{\theta_i}$ for all $i$ gives a filling of the r-shape $\lambda/\mu$. Let us define the map
$$
    \psi : SP_{\lambda/\mu}(\pi, \mathbf{a}, \mathbf{b}) \rightarrow ST_{\lambda/\mu}
$$
given by 
$$
    \psi: (P_1,...,P_k) \mapsto (\psi_{1,1}(P_1),...,\psi_{k,k}(P_k)) = T.
$$

\begin{lemma}
    \label{lem:bijection-bt-bl}
    The map 
    $$
    \psi : SP_{\lambda/\mu}(\pi, \mathbf{a}, \mathbf{b}) \rightarrow ST_{\lambda/\mu}
    (\mathbf{a}, \mathbf{b})
    $$
    is a weight-preserving bijection.
\end{lemma}
\begin{proof}
    Let $\mathbf{P}' = (P'_1, \ldots, P'_k) \in ZP_{\lambda/\mu}(\pi, \mathbf{a}, \mathbf{b})$ 
    and $T' = \phi_Z(\mathbf{P}') \in ZT_{\lambda/\mu}(\mathbf{a}, \mathbf{b})$.
    
    Let $\mathbf{P} \in SP_{\lambda/\mu}(\pi, \mathbf{a}, \mathbf{b})$, and 
    $T = \psi(\mathbf{P}) = (T_{\theta_i})_{i \in [k]}$, 
    where $T_{\theta_i} = \psi_{i,i}(P_i)$. Since each 
    $T_{\theta_i}$ is a representation of $T'_{\theta_i}$, 
    and $(T'_{\theta_i})_i^k = T'$,  
    it follows that $T \in ST_{\lambda/\mu}(\mathbf{a}, \mathbf{b})$. 
    As $wt(P_i) = wt(T_{\theta_i})$ we have 
    $wt(\mathbf{P}) = wt(T)$.

    Conversely, let $T \in ST_{\lambda/\mu}( \mathbf{a}, \mathbf{b})$ be a representation of $T'$.
    We define the inverse map $\psi^{-1}$ as follows: 
    $\psi^{-1}(T) = (\psi^{-1}_{i, i}(T_{\theta_i}))_{i \in [k]} = \mathbf{P}$. 
    As $\psi^{-1}_{i,i}(T_{\theta_i})$ is a representation of 
    $P'_i$ for all $i \in [k]$ and $\mathbf{P}'$ is non-intersecting, it 
    follows that $\mathbf{P}$ is non-intersecting too. This 
    completes the proof.
\end{proof}

Next we consider the decomposition of the diagram into 
vertical ribbons, i.e. $\Theta = (\theta_1,..., \theta_{\lambda_1})$, 
where each $\theta_j$ is a $j$-th column of the diagram.
We write an enumerator for each path 
from $C_i$ to $D_j$ on $SL_{\lambda/\mu}(\pi, \mathbf{a}, \mathbf{b})$ in terms 
of the elementary supersymmetric functions $e_{n}$. Let $wt(C_i, D_j)$ be 
an enumerator for paths on $SL_{\lambda/\mu}(\pi, \mathbf{a}, \mathbf{b})$ 
from $C_i$ to $D_j$, i.e. 
$$ wt(C_i, D_j) := \sum_{P:C_i \rightarrow D_j} wt(P).$$

Let $T \in ST_{\lambda/\mu}(\mathbf{a}, \mathbf{b})$.
We define the \textit{index of $T_{ij}$} 
denoted by $\chi(T_{ij}, c)$ for $c=c(i, j)$ as follows: for 
some integer $r$,
\[
\chi(T_{ij},c) = 
    \begin{cases}
        r \quad \text{if} \quad T_{ij}=r \\ 
        r - c \quad \text{if} \quad T_{ij}=r'
    \end{cases}
\]
We can restore the flagged $\mathbb{Z}$-SSYT $T'$ (taking $T'_{ij} = \chi(T_{ij}, c(i, j))$ 
for all $i,j$), from which $T$ was derived. We 
call it the \textit{index tableau} of $T$ and denote it by $\chi(T)$.

Now we define an `inverse' of the weight function defined in eq. \eqref{eq:wt-func-bt-0}:  
\begin{equation}
    wt^{-1}(w) = 
    \begin{cases}
        r \quad \text{if} \quad w=y_r \\
        r' \quad \text{if} \quad w=-z_r
    \end{cases}
\end{equation}

\begin{lemma}
    \label{lem:vert-rib-enum}
    Let us decompose the r-shape $\lambda/\mu$ into vertical ribbons. Then on 
    $SL_{\lambda/\mu}(\pi, \mathbf{a}, \mathbf{b})$ we have 
    $$wt(C_i, D_j) = e_{\lambda'_i-\mu'_j-i+j}(\mathbf{y}_{a_j,b_i} / 
    \mathbf{z}_{a'_j,b'_i}).$$
\end{lemma}
\begin{proof}
    By Lemma~\ref{lem:bt-path-bfr} it follows that there is 
    a weight-preserving bijection between paths $P_S(C_i, D_j)$ and 
    flagged super tableaux of r-shape $(1^{\lambda'_i-\mu'_j-i+j}, i - \lambda'_i + \Delta)$ 
    with flags $\mathbf{a}^{ij}=(a_j), \mathbf{b}^{ij}=(b_i)$.  
    Thus, we aim to show that for every term 
    $w = w_1 \cdots w_{\lambda'_i - i -\mu'_j + j}$ in r.h.s. there is a 
    unique flagged super tableau $Q$ of r-shape 
    $\theta_i\# \theta_j = (1^{\lambda'_i-\mu'_j-i+j}, i - \lambda'_i + \Delta)$ 
    with flags $\mathbf{a}^{ij}=(a_j), \mathbf{b}^{ij}=(b_i)$ of weight $w$.
    
    Let $U=\{t_1,...,t_{\lambda'_i-\mu'_j-i+j}\}$ be a multiset 
    where $t_i = wt^{-1}(w_i)$. Then we construct $Q$ starting from the bottommost cell. We proceed as follows: for content $c\in [i-\lambda'_i + \Delta, j-(\mu'_j+1) + \Delta]$ 
    starting from $c = i - \lambda'_i + \Delta$ and proceeding in increasing  order, let $\tau_c$ be an element in $U$ with maximal index, i.e. 
    $\tau_c = \argmax\limits_{t_i\in U}\chi(t_i, c)$. There might be several elements $s_1,\ldots, s_k \in U$ with maximal index but only one of them can be unprimed (as unprimed elements correspond to variables from the alphabet $\mathbf{y}$ in $e_n(\mathbf{y}/\mathbf{z})$), others are primed and all primed elements are the same. If there is an unprimed element with maximal index, we choose $\tau_c$ as this element. Otherwise $s_1 = \ldots = s_k = r'$, for some $r$, and we choose $\tau_c = r'$. Next we remove $\tau_c$ from $U$, i.e. update $U \to U \setminus \{\tau_c\}$. We fill the cell of content $c$ with the value $\tau_c$. 
    
    At the end of the described process we obtain some tableau $Q$ filled with primed and unprimed integers. Notice that $b_i \ge \chi(\tau_c,c) > \chi(\tau_{c + 1}, c + 1) \ge a_j$ for all $c$; the leftmost and the rightmost inequalities hold as $w_q \in \mathbf{y}_{a_j, b_i} \cup \mathbf{z}_{a'_j, b'_i}$ for all $q$; for strict inequality, if $\chi(\tau_c, c) > \chi(\tau_{c + 1}, c)$ then the inequality is obvious; if $\chi(\tau_c, c) = \chi(\tau_{c + 1}, c)$ the inequality holds as well, as by the algorithm $\tau_{c + 1}$ must be primed. Consequently, the $Q$ is semistandard, and it respects the flags $\mathbf{a}^{ij}, \mathbf{b}^{ij}$, thus $Q \in ST_{\theta_i \# \theta_j}(\mathbf{a}^{ij}, \mathbf{b}^{ij})$, and also note that $wt(Q) = w$. Now we show that $Q$ is unique. 
    If in addition to the tableau $Q = (q_{i-\lambda'_i + \Delta},\ldots,q_{j-(\mu'_j+1) + \Delta})$ 
    constructed by the algorithm above, there 
    is another super tableau $R  = (r_{i-\lambda'_i + \Delta},\ldots,r_{j-(\mu'_j+1) + \Delta})$ 
    with the same weight,
    where $q_c, r_c$ are values inside the cells with content $c$ 
    of $Q$ and $R$. Then there must be an integer $\ell$, such that $q_c = r_c$ for all $c < \ell$ and $q_{\ell} \neq r_{\ell}$, which means that on $\ell$-th step of the algorithm $r_{\ell}$ was chosen instead of $q_{\ell}$ and there are two possibilities: (1) $\chi(r_{\ell}, \ell) < \chi(q_{\ell}, \ell)$, and (2) $\chi(r_{\ell}, \ell) = \chi(q_{\ell}, \ell)$ but $q_{\ell}$ is unprimed. 
    
    In (1), as $q_{\ell}$ lies above  
    $r_{\ell}$ in $R$, it follows that $\chi(q_{\ell}, \ell) > \chi(r_{\ell}, \ell) \ge 
    \chi(q_{\ell}, \ell + p) + p$, for some positive $p$, but then 
    $\chi(q_{\ell}, \ell) > \chi(q_{\ell}, \ell + p) + p$ which is not true, and we get a contradiction.

    In (2), again, it follows that $\chi(r_{\ell}, \ell) 
    \ge \chi(q_{\ell}, \ell + p) + p$ for some positive $p$, 
    but as $q_{\ell}$ is unprimed, it follows that 
    $\chi(q_{\ell}, \ell) \ge \chi(q_{\ell}, \ell) + p$, which is a contradiction. 
    
    Thus, $Q$ is unique. 

    Now we need to show that the weight of every flagged super 
    tableau is taken into account in the r.h.s. From 
    the definition of flagged super tableaux, the weight 
    of $Q$ will be of the form:
    $$y_{p_1}...y_{p_{\alpha}}(-z_{q_1})...(-z_{q_{\beta}})$$
    where  $\alpha+\beta=\lambda'_i-\mu'_j-i+j$ 
    with $p_i,q_i$ s.t. 
    \[b_i\ge p_1>...> p_{\alpha}\ge a_j\] 
    \[b'_i\ge q_1\ge...\ge q_{\beta}\ge a'_j\] 
    and every term is in 
    $e_{\lambda'_i-\mu'_j-i+j}(\mathbf{y}_{a_j,b_i} / 
    \mathbf{z}_{a'_j,b'_i})$ by definition,
    which completes the proof.
\end{proof}

\begin{proof}[Proof of Theorem \ref{thm:signed-expansion}]
    We decompose the diagram into vertical ribbons. From Lemma~\ref{lemma:bad-paths-intersect} it follows that the system of paths $\mathbf{P} = (P_i)$ on the lattice $SL_{\lambda/\mu}(\pi, \mathbf{a}, \mathbf{b})$ is non-intersecting only if $P_i \in P(C_i, D_i)$ for all $i$. 
    Since there is a weight preserving bijection between flagged super tableaux 
    from $ST_{\lambda/\mu}(\mathbf{a}, \mathbf{b})$ and systems of non-intersecting paths $\mathbf{P}$ on  
    $SL_{\lambda/\mu}(\pi, \mathbf{a}, \mathbf{b})$ (Lemma \ref{lem:bijection-bt-bl}) it 
    follows that 
    \[
        \sum_{T \in ST_{\lambda/\mu}(\mathbf{a}, \mathbf{b})} wt(T) = \sum_{\mathbf{P}}wt(\mathbf{P}).
    \]
    As the lattice $SL_{\lambda/\mu}(\pi, \mathbf{a}, \mathbf{b})$ is acyclic, we can apply the LGV lemma to enumerate the systems of non-intersecting lattice paths on it. 
    Using the LGV lemma and the enumerator of paths $P_S(C_i,D_j)$
    from Lemma $\ref{lem:vert-rib-enum}$, we obtain   
    \[
        \sum_{\mathbf{P}}wt(\mathbf{P}) = \det \left[e_{\lambda'_i-\mu'_j-i+j}(\mathbf{y}_{a_j,b_i} / 
    \mathbf{z}_{a'_j,b'_i})\right]_{i,j}^{\lambda_1},
    \]
    which is $\mathsf{S}^{\mathbf{a}, \mathbf{b}}_{\lambda/\mu}(\mathbf{y} / \mathbf{z})$ 
    by definition.
\end{proof}


\begin{proof}[Proof of Theorem~\ref{thm:HG-1}]
    Consider the systems of non-intersecting paths on the super lattice $SL_{\lambda/\mu}(\pi, \mathbf{a}, \mathbf{b})$. 
    From Lemma~\ref{lemma:bad-paths-intersect} it follows that the system of paths $\mathbf{P} = (P_i)$ on the lattice $SL_{\lambda/\mu}(\pi, \mathbf{a}, \mathbf{b})$ is non-intersecting only if $P_i \in P(C_i, D_i)$ for all $i$. 
    On one hand, by Theorem~\ref{thm:signed-expansion} and Lemma~\ref{lem:bijection-bt-bl} it follows that $\mathsf{S}_{\lambda/\mu}^{\mathbf{a}, \mathbf{b}}
    (\mathbf{y} / \mathbf{z})$ is enumerator for the systems of non-intersecting paths. On the other hand, by Lemma~\ref{lem:bt-path-bfr}, theorem~\ref{thm:signed-expansion} it follows that $\mathsf{S}^{\mathbf{a}^{ij}, \mathbf{b}^{ij}}_{\theta_i\# \theta_j}
    (\mathbf{y} / \mathbf{z})$ is enumerator of paths 
    from $C_i$ to $D_j$, for all $i,j$. Thus, using the LGV lemma the result follows.
\end{proof}
\begin{proof}[Proof sketch of Theorem~\ref{thm:hg-strict-flags}]
    Let $\mathbf{\tilde{a}} = (\tilde{a}_i), \mathbf{\tilde{b}} = (\tilde{b}_i)$  be strict flags. The induced ribbon flags $\mathbf{\tilde{a}}^{ij}=(\tilde{a}^{ij}_r)_r, \mathbf{\tilde{b}}^{ij}=(\tilde{b}^{ij}_r)_r$ are given by  
    \begin{align*}
    \tilde{a}^{ij}_r = \tilde{a}_{\mathrm{col}(\min(c(\gamma^{ij}_r)))}
    \qquad 
    \tilde{b}^{ij}_r = \tilde{b}_{\mathrm{col}(\max(c(\delta^{ij}_r)))}-
    \lambda'_{\mathrm{col}(\max(c(\delta^{ij}_r)))} + \mathrm{row}(\max(c(\delta^{ij}_r))), 
    \end{align*}
    where $\delta^{ij}_r, \gamma^{ij}_r$ are the head and tail of $\theta_i \# \theta_j$. 

    In the definition of the lattice everything stays the same, except $\overline{a}_c$ given by:
    \[
        \overline{a}_c = a_{\mathrm{col}(m_c)}.
    \]
    Next, in Lemma~\ref{thm:bijection-path-ribbon}, given the path $P = \phi^{-1}_{i, j}(Q)$ from $C_i$ to $D_j$ where $Q \in ZT_{\theta_i \# \theta_j}(\mathbf{\tilde{a}}^{ij}, \mathbf{\tilde{b}}^{ij})$, we have to check that the edge $p_r = (r, w + f_r) \rightarrow (r + 1, w + f_r)$ contained in $P$ exists on the lattice, where $p_r$ is constructed by a cell of content $r$ contained in $s$-th column of $\theta_i \# \theta_j$. Let $\delta, \gamma$ be the bottom and top cells of $s$-th column of $\theta_i \# \theta_j$. We only need to verify that $\tilde{a}_{\mathrm{col}(\mathrm{min}(r))} + f_r \le w + f_r$. As $Q$ is semistandard we have $\tilde{a}^{ij}_s \le w$. Note that $\mathrm{col}({\mathrm{min}(r)}) \le \mathrm{col}({\mathrm{min}(c(\gamma))})$, hence $\tilde{a}_{\mathrm{col}({\mathrm{min}(r)})} \le \tilde{a}_{\mathrm{col}({\mathrm{min}(c(\gamma))})} = \tilde{a}^{ij}_s$ and the result follows. All other components of the above proof do not change. 

    Then one can repeat the proof of Theorem~\ref{thm:HG-1} changing the ordinary flags to strict flags, and the result follows. 
\end{proof}

\section{Some special decompositions}\label{sec:special}
In this section we show some new formulas for special cases of outer ribbon 
decompositions: Jacobi-Trudi and skew Giambelli type formulas. 

\subsection{Jacobi-Trudi-type formulas}
Let $n \in \mathbb{Z}_{>0}$, $c \in \mathbb{Z}$ and $\mathbf{p, q}$ be column flags for the r-shape $(n, c)$ (single row with shifted content $c$). 
Let us denote 
\begin{align*}
    h^{\mathbf{p, q}}_{n, c}(\mathbf{y} / \mathbf{z}) := \mathsf{S}^{\mathbf{p, q}}_{(n)}(\mathbf{y} / \mathbf{z}),
\end{align*}
where the polynomial on the r.h.s. is indexed by the r-shape $(n, c)$.
From Theorem~\ref{thm:signed-expansion} we obtain the following formula: 
\begin{align*}
    h^{\mathbf{p, q}}_{n, c}(\mathbf{y} / \mathbf{z}) = 
    \sum_{{\substack{i_1 \le ... \le i_n \\ p_k \le i_k \le q_k,\ k\in [n]}}}\prod_{r=1}^n(y_{i_r}-z_{i_r + c + r - 1}).
\end{align*}
Set $h^{\mathbf{p, q}}_{0, c}(\mathbf{y} / \mathbf{z}) = 1$ and $h^{\mathbf{p, q}}_{n, c}(\mathbf{y} / \mathbf{z}) = 0$ for $n < 0$. 
\begin{corollary}[Jacobi--Trudi-type formulas]
Let 
$\mathbf{a, b}$ be some flags for (the {usual shape}) $\lambda/\mu$. The following formulas hold: 
    \begin{align*}        
        \mathsf{S}^{\mathbf{a},\mathbf{b}}_{\lambda/\mu}(\mathbf{y} / \mathbf{z}) = 
        \det \left[h^{\mathbf{a}^{ij},\mathbf{b}^{ij}}
        _{\lambda_i-\mu_j-i+j,\, \mu_j + 1 - j}(\mathbf{y} / \mathbf{z}) \right]_{1 \le i,j \le n}, 
    \end{align*}
    where $n \ge \ell(\lambda)$  and $\mathbf{a}^{ij}, \mathbf{b}^{ij}$ are induced ribbon flags for the r-shape $(\lambda_i-\mu_j-i+j, \mu_j + 1 - j)$. 
\end{corollary}
\begin{proof}
    Let us decompose the r-shape $\lambda/\mu$ into horizontal ribbons. Taking $\theta_i = (\lambda_i - \mu_i, \mu_i + 1 - i)$, we have $\theta_j \# \theta_i = (\lambda_i - \mu_j - i + j, \mu_j + 1 - j)$. Then the result follows from Theorem~\ref{thm:HG-1} after transposing the determinant.
\end{proof}

Let $\mathbf{p}' = (p'_i), \mathbf{q}' = (q'_i)$ be some row flags for the r-shape $(1^n, c)$. Let 
\begin{align*}
    e^{\mathbf{p}', \mathbf{q}'}_{n, c}(\mathbf{y} / \mathbf{z}) := \overline{\mathsf{S}}^{\mathbf{p}', \mathbf{q}'}_{1^n}(\mathbf{y} / \mathbf{z}) 
\end{align*}
where the polynomial on the r.h.s. is indexed by the r-shape $(1^n, c)$.
From Corollary~\ref{cor:signed-expansion-dual} we obtain the following formula: 
\begin{align*}
    e^{\mathbf{p}', \mathbf{q}'}_{n, c}(\mathbf{y} / \mathbf{z}) = 
    \sum_{{\substack{i_1 < ... < i_n \\ p'_k \le i_k \le q'_k,\ k\in [n]}}}\prod_{r=1}^n(y_{i_r}-z_{i_r + c + n - r}).
\end{align*}
Set $e^{\mathbf{p}', \mathbf{q}'}_{0, c}(\mathbf{y} / \mathbf{z}) = 1$ and $e^{\mathbf{p}', \mathbf{q}'}_{n, c}(\mathbf{y} / \mathbf{z}) = 0$ for $n < 0$.
\begin{corollary}[Dual Jacobi--Trudi-type formulas]
Let $\mathbf{a}', \mathbf{b}'$ be some row flags for ({the usual shape}) $\lambda/\mu$. The following formulas hold:  
\begin{align*}
    \overline{\mathsf{S}}^{\mathbf{a}', \mathbf{b}'}_{\lambda/\mu}(\mathbf{y} / \mathbf{z}) = \det \left[e^{\mathbf{a}'^{ij},\mathbf{b}'^{ij}}_{\lambda'_i-\mu'_j-i+j,\, i - \lambda'_i}(\mathbf{y} / \mathbf{z}) \right]_{1 \le i,j \le n},
\end{align*}
where $n \ge \lambda_1$ and  $\mathbf{a}'^{ij}, \mathbf{b}'^{ij}$ are induced row ribbon flags for the r-shape $(\lambda'_i - \mu'_j - i + j,\ i - \lambda'_i)$.
\end{corollary}
\begin{proof}
    Let us decompose the r-shape $\lambda/\mu$ into vertical ribbons, then 
    the result follows from Corollary~\ref{thm:HG-1-dual}.
\end{proof}
\subsection{Skew Giambelli-type formula}
We need \textit{Frobenius notation} to write Giambelli formulas. 
For a partition $\lambda$, let $(k, k)$ be the rightmost diagonal 
cell in $\lambda$, then $k$ is said to be the size of the 
\textit{Durfee square}. Let $u_i$ and $v_i$ be the number of cells 
to the right of the cell $(i,i)$ and below $(i,i)$ for $i \in [k]$. 
Then we can write $\lambda = (u_1,...,u_k|v_1,...,v_k)$ known as Frobenius notation. 
Let $\mu = (c_1,...,c_{\ell}|d_1,...,d_{\ell})$ be also written in Frobenius notation. 
We denote hooks by $(u_j|v_i)$. Let 
$\mathbf{a, b}$ be flags for the r-shape $\lambda/\mu$, and 
$\Theta = (\theta_1, \ldots, \theta_{k+\ell})$ be an outer decomposition of  
$\lambda/\mu$ into hooks and strips as in Fig.~\ref{pic:hook-decomposition}. In this decomposition 
we have $k + \ell$ ribbons of three types: vertical, horizontal 
ribbons, and hooks:
$$
\begin{cases}
    \Theta(-v_i, u_i)\quad \text{for}\quad i \in [\ell+1, k] \qquad &\text{hooks}, \\
    \Theta(-v_i, -d_i - 1)\quad \text{for}\quad i \in [\ell] \qquad &\text{vertical ribbons}, \\
    \Theta(c_i + 1, u_i)\quad \text{for}\quad i \in [\ell] \qquad &\text{horizontal ribbons}.
\end{cases}
$$

\begin{figure}
    \centering
\begin{tikzpicture}[scale=0.4]

\foreach \x in {4,5,6,7} {
  \draw (\x,7) rectangle (\x+1,8);
}

\foreach \x in {4,5,6,7} {
  \draw (\x,6) rectangle (\x+1,7);
}

\foreach \x in {3,4,5,6} {
  \draw (\x,5) rectangle (\x+1,6);
}

\foreach \x in {1,2,3,4,5,6} {
  \draw (\x,4) rectangle (\x+1,5);
}

\foreach \x in {0,1,2,3,4,5} {
  \draw (\x,3) rectangle (\x+1,4);
}

\foreach \x in {0,1,2,3,4,5} {
  \draw (\x,2) rectangle (\x+1,3);
}

\foreach \x in {0,1,2,3,4} {
  \draw (\x,1) rectangle (\x+1,2);
}

\foreach \x in {0,1,2} {
  \draw (\x,0) rectangle (\x+1,1);
}

\draw [red, very thick] (0.5, 0) -- (0.5, 4);
\draw [red, very thick] (1.5, 0) -- (1.5, 5);
\draw [red, very thick] (2.5, 0) -- (2.5, 5);
\draw [red, very thick] (3, 5.5) -- (7, 5.5);
\draw [red, very thick] (4, 6.5) -- (8, 6.5);
\draw [red, very thick] (4, 7.5) -- (8, 7.5);
\draw [red, very thick] (3.5, 1) -- (3.5, 4.5) -- (7, 4.5);
\draw [red, very thick] (4.5, 1) -- (4.5, 3.5) -- (6, 3.5);
\draw [red, very thick] (5.5, 2) -- (5.5, 2.5) -- (6, 2.5);

\end{tikzpicture}
\caption{Ribbon decomposition into hooks and strips.}
\label{pic:hook-decomposition}
\end{figure}
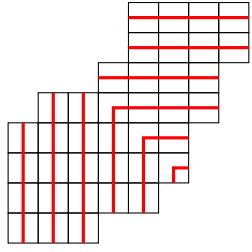

We specify the ribbons we get using $\#$ operation 
for ribbons in $\Theta$:
\[
\begin{cases}
    \Theta(-v_i,u_j) = (u_j|v_i) 
    \quad & i, j\in [k], \\[0.5em]
    \Theta(-v_i, -d_j - 1)  = 1^{v_i - d_j}
    \quad & i \in [k]\ , j\in [\ell], \\[0.5em]
    \Theta(c_i + 1, u_j) = u_j - c_i
    \quad & i \in [\ell]\ , j\in [k], \\[0.5em]
    \Theta(c_i + 1, -d_j - 1) = undefined 
    \quad & i \in [\ell]\ , j\in [\ell].
\end{cases}
\]
Let us define the following ribbon flags:
\begin{itemize}
    \item $\mathbf{a}^{ij} = (a^{ij}_r)_r, \mathbf{b}^{ij} = (b^{ij}_r)_r$ 
    for hooks $(u_j|v_i)$, $i,j\in [k]$
    \item $\mathbf{p}^{ij} = (p^{ij}_r)_r,\mathbf{q}^{ij} = (q^{ij}_r)_r$
    for vertical ribbons $1^{v_i - d_j}$, $i \in [k]\ , j\in [\ell]$
    \item $\mathbf{\tilde{p}}^{ij} = (\tilde{p}^{ij}_r)_r,\mathbf{\tilde{q}}^{ij} = (\tilde{q}^{ij}_r)_r$ 
    for horizontal ribbons $u_j - c_i$, $i \in [\ell]\ , j\in [k]$
\end{itemize}

\begin{corollary}[Skew Giambelli formulas for flagged supersymmetric Schur functions]
The following formulas hold: 
    \label{cor:Giambelli-main}
    \[
    \mathsf{S}^{\mathbf{a}, \mathbf{b}}_{\lambda/\mu}(\mathbf{y} / \mathbf{z}) 
    = (-1)^{\ell} \det \begin{bmatrix}
    \bigl[\mathsf{S}^{\mathbf{a}^{ij}, \mathbf{b}^{ij}}_{u_j \mid v_i}\bigr]_{\begin{subarray}{l}
            i,j \in [k] 
        \end{subarray}} 
    & \bigl[\mathsf{S}^{\mathbf{p}^{ij},\mathbf{q}^{ij}}_{1^{v_i - d_j}}\bigr]_{\begin{subarray}{l}
            i \in [k] \\ j \in [\ell] 
        \end{subarray}} \\\\
    \bigl[\mathsf{S}^{\mathbf{\tilde{p}}^{ij},\mathbf{\tilde{q}}^{ij}}_{u_j - c_i}\bigr]_{\begin{subarray}{l}
            i \in [\ell] \\ j \in [k] 
        \end{subarray}} 
    & \mathbf{0}
\end{bmatrix},
    \]
where the common arguments $(\mathbf{y} / \mathbf{z})$ are suppressed in all three nonzero matrix blocks.
\end{corollary}

\begin{proof}
    First we define the order of ribbons in $\Theta$:
    $$
    \theta_i = 
    \begin{cases}
        \Theta(-v_i, -d_i - 1)\quad &\text{for}\quad i \in [\ell], \\
        \Theta(c_{i - \ell} + 1, u_{i - \ell})\quad &\text{for}\quad i \in [\ell + 1, 2\ell], \\
        \Theta(-v_{i - \ell}, u_{i - \ell})\quad &\text{for}\quad i \in [2\ell + 1, k + \ell]. 
    \end{cases}
    $$
    Then we specify the ribbons $\theta_i \# \theta_j$:
    \[
    \theta_i \# \theta_j = 
    \begin{cases}
        \Theta(-v_i, -d_j - 1)  = 1^{v_i - d_j}
        ,\ & i,j \in [\ell], \\[0.5em]
        \Theta(-v_i,u_{j - \ell}) = u_{j - \ell}|v_i 
        ,\ & i \in [\ell], 
        j\in [\ell + 1, k + \ell], \\[0.5em]
        \Theta(c_{i - \ell} + 1, u_{j - \ell}) = u_{j - \ell} - c_{i - \ell}
        ,\ & i \in [\ell + 1, 2\ell], 
        j\in [\ell + 1, k + \ell], \\[0.5em]
        \Theta(c_{i - \ell} + 1, -d_j - 1) = undefined 
        ,\ & i \in [\ell + 1, 2\ell]\ , j\in [\ell], \\[0.5em]
        \Theta(-v_{i - \ell}, -d_j - 1) = 1^{v_{i - \ell} - d_j} 
        ,\ & i \in [2\ell + 1, k + \ell]\ , 
        j\in [\ell], \\[0.5em]
        \Theta(-v_{i - \ell},u_{j - \ell}) = u_{j - \ell}|v_{i - \ell},
        \ & i \in [2\ell + 1, k + \ell], 
        j\in [\ell + 1, k + \ell].
    \end{cases}
    \]
    Now, using our Hamel--Goulden-type formula we get determinant of some 
    matrix which can be split into block matrices. 
    To simplify notation, we shall use only shapes (i.e. indices) of corresponding polynomials:
    $$
    \begin{bmatrix}
        \bigl[ 1^{v_i - d_j} \bigr]_{i,j \in [\ell]} 
        & \bigl[ u_{j - \ell} | v_i \bigr]_{\begin{subarray}{l}
            i \in [\ell] \\ j \in [\ell + 1, k + \ell] 
        \end{subarray}}\\\\
        \bigl[ undefined \bigr]_{\begin{subarray}{l}
            i \in [\ell + 1, 2\ell] \\ j \in [\ell] 
        \end{subarray}} 
        &  
        \bigl[ u_{j-\ell} - c_{i - \ell} \bigr]_{\begin{subarray}{l}
            i \in [\ell + 1, 2\ell] \\ j \in [\ell + 1, k + \ell] 
        \end{subarray}} \\\\
        \bigl[ 1^{v_{i - \ell} - d_j} \bigr]_{\begin{subarray}{l}
            i \in [2\ell + 1, k + \ell] \\ j \in [\ell] 
        \end{subarray}} 
        &
        \bigl[ u_{j - \ell} | v_{i - \ell} \bigr]_{\begin{subarray}{l}
            i \in [2\ell + 1, k + \ell] \\ j \in [\ell + 1, k + \ell] 
        \end{subarray}} 
    \end{bmatrix}.
    $$
    which can be rewritten as:
    $$
    \begin{bmatrix}
        \bigl[ 1^{v_i - d_j} \bigr]_{i,j \in [\ell]} 
        & \bigl[ u_{j} | v_i \bigr]_{\begin{subarray}{l}
            i \in [\ell] \\ j \in [k] 
        \end{subarray}}\\\\
        \bigl[ undefined \bigr]_{\begin{subarray}{l}
            i, j \in [\ell] 
        \end{subarray}} 
        &  
        \bigl[ u_{j} - c_{i} \bigr]_{\begin{subarray}{l}
            i \in [\ell] \\ j \in [k] 
        \end{subarray}} \\\\
        \bigl[ 1^{v_{i} - d_j} \bigr]_{\begin{subarray}{l}
            i \in [\ell + 1, k] \\ j \in [\ell] 
        \end{subarray}} 
        &
        \bigl[ u_{j} | v_{i} \bigr]_{\begin{subarray}{l}
            i \in [\ell + 1, k] \\ j \in [k] 
        \end{subarray}} 
    \end{bmatrix}.
    $$
    Permuting the second and third row blocks, and  
    the first and second column blocks gives 
    the desired matrix. The number of overall 
    permutations of rows and columns is  
    $2k\ell - \ell^2$ , and so we need to multiply the 
    determinant by the number $(-1)^{2k\ell - \ell^2} = (-1)^{\ell}$.
\end{proof}

\section{Formulas for dual refined canonical stable Grothendieck polynomials}\label{sec:g}
In this section discuss specializations of the above formulas to dual refined canonical stable Grothendieck polynomials. We show that the functions $g_{\lambda/\mu}(\mathbf{x}; \boldsymbol{\alpha}, \boldsymbol{\beta})$ have a combinatorial formula via $g$-tableaux deduced from super tableaux, examine the Hamel--Goulden formulas in more detail and show new Jacobi-Trudi-type and skew Giambelli-type formulas. 


\subsection{$g$-tableaux}
\begin{definition}
    Let $m\in \mathbb{Z}_{\ge 0}$. Let $\mathbf{a}$, $\mathbf{b}$ be column flags for the r-shape $\lambda/\mu$.
    Let $\tilde{T} \in ST_{\lambda/\mu}(\mathbf{a, b})$ be a super tableau such that $\tilde{T}$ does not contain a primed element from $[m']$. We produce a \textit{$\mathsf{g}$-tableau} $T = (T_{i,j})$ of the same r-shape as follows:
    \begin{align*}
        T_{i,j} = 
        \begin{cases}
            \tilde{T}_{i,j}, & \text{ if } \quad \tilde{T}_{i,j} \in [m], \\
            (\tilde{T}_{i,j} - m)^{\circ}, & \text{ if } \quad \tilde{T}_{i,j} > m, \\
            (\tilde{T}_{i,j} - 1)^{\bullet}, & \text{ if } \quad \tilde{T}_{i,j} \le 0, \\
            (\tilde{T}_{i,j} - m)^{\bullet}, & \text{ if } \quad \tilde{T}_{i,j} > m', \\
            (\tilde{T}_{i,j} - 1)^{\circ}, & \text{ if } \quad \tilde{T}_{i,j} \le 0'.
        \end{cases}
    \end{align*}   
    Let ${GT}_{\lambda/\mu}(\mathbf{a}, \mathbf{b})$ be the set of all flagged $\mathsf{g}$-tableaux 
    produced from the set $ST_{\lambda/\mu}(\mathbf{a, b})$ 
    (more precisely, from tableaux in $ST_{\lambda/\mu}(\mathbf{a, b})$ that do not contain a primed element from $[m']$). 
    For $r \in \mathbb{Z}_{> 0}$, define the weight $wt(T_{ij})$ as follows: 
    \begin{equation*}
        wt(T_{ij}) := 
        \begin{cases}
            \label{eq:wt-func-bt-1}
            x_r, &\text{if} \quad T_{ij}=r, \\ 
            \alpha_r , &\text{if} \quad T_{ij}=r^{\bullet}, \\ 
            -\alpha_{r} , &\text{if} \quad T_{ij}=(-r)^{\bullet}, \\
            \beta_r, &\text{if} \quad T_{ij}=r^{\circ}, \\ 
            -\beta_{r}, &\text{if} \quad T_{ij}=(-r)^{\circ}.
        \end{cases}
    \end{equation*}
    and define the weight $wt(T)$ of the  $\mathsf{g}$-tableau $T$ as
    \begin{equation*}
        \label{eq:wt-func-gt}
        wt(T) = \prod_{i,j}wt(T_{ij}).
\end{equation*}

Let us fix the column flags $\mathbf{a} = (-i + 2)_i$, $\mathbf{b} = (\lambda'_i + m - 1)_i$  for the usual shape $\lambda/\mu$. We call a $\mathsf{g}$-tableau with the flags $\mathbf{a}, \mathbf{b}$ a {\it $g$-tableau}. Let $GT_{\lambda/\mu}$ be the set of all $g$-tableaux of the shape $\lambda/\mu$. 
\end{definition}

\begin{proposition}[Tableaux formula for refined dual canonical Grothendieck polynomials] 
    The following tableaux formula holds:
    \begin{align*}
    \mathsf{g}^{\mathbf{a}, \mathbf{b}}_{\lambda/\mu}(\mathbf{x}_m; \boldsymbol{\alpha}, \boldsymbol{\beta}) = \sum_{T\in {GT}_{\lambda/\mu}(\mathbf{a}, \mathbf{b})} wt(T).
    \end{align*}
    In particular, if we fix column flags $\mathbf{a} = (-i + 2)_i$, $\mathbf{b} = (\lambda'_i + m - 1)_i$  for the usual shape $\lambda/\mu$, the following tableaux formula holds:
    \begin{align*}
    g_{\lambda/\mu}(\mathbf{x}_m; \boldsymbol{\alpha}, \boldsymbol{\beta}) = \sum_{T\in GT_{\lambda/\mu}} wt(T).
    \end{align*}
\end{proposition}
\begin{proof}
    Let $\tilde{T}\in ST_{\lambda/\mu}(\mathbf{a,b})$ and $T\in {GT}_{\lambda/\mu}(\mathbf{a}, \mathbf{b})$. It is not hard to see that $wt(\tilde{T}_{i,j}) \rightarrow wt(T_{i,j})$ corresponds to $g$-specialization (the effect of $z_i \rightarrow 0$ is achieved by not taking into account super tableaux with primed elements from $[m']$), then the result follows from  Theorem~\ref{thm:signed-expansion}.
\end{proof}

\subsection{Determinantal formulas}
Let $m \in \mathbb{Z}_{\ge 0}$, $\mathbf{a} = (-i + 2)_i$, $\mathbf{b} = (\lambda'_i + m - 1)_i$ be column flags for the usual shape $\lambda/\mu$, and $\mathbf{a}^{ij}, \mathbf{b}^{ij}$ be induced ribbon flags. Recall that we have the following Hamel--Goulden-type formula, written via flagged dual Grothendieck enumerators:
\begin{align*}
    g_{\lambda/\mu}(\mathbf{x}_m; \boldsymbol{\alpha}, \boldsymbol{\beta}) = 
    \det \left[\mathsf{g}_{\theta_i \# \theta_j}^{\mathbf{a}^{ij}, 
    \mathbf{b}^{ij}}(\mathbf{x}_m; \boldsymbol{\alpha}, \boldsymbol{\beta})\right]_{1\le i,j \le k}. 
\end{align*} 
Let  
$\theta_i \# \theta_j = \rho^{ij}_1 \to \ldots \to \rho^{ij}_{k_{ij}}$ be canonical decompositions, and $\delta^{ij}_r, \gamma^{ij}_r$ 
be the bottom and top cells of $\rho^{ij}_r$. 
Let $m^{ij}_r = \min(c(\gamma^{ij}_r))$ and $M^{ij}_r = 
\max(c(\delta^{ij}_r))$.

\begin{proposition}
We can specify the induced ribbon flags $\mathbf{a}^{ij}, \mathbf{b}^{ij}$ as follows: 
\begin{alignat*}{2}
    &\mathbf{a}^{ij} =  (-c(m^{ij}_r) - \mu'_{\mathrm{col}(m^{ij}_r)} + 1)_{r\in [{k_{ij}}]} &\qquad& 
    \mathbf{b}^{ij} = (\mathrm{row}(M^{ij}_r) + m - 1)_{r\in [k_{ij}] } \\ 
    &\mathbf{a'}^{ij} =  (- \mu'_{\mathrm{col}(m^{ij}_r)} + 1)_{r\in [k_{ij}]} &\qquad& 
    \mathbf{b'}^{ij} = (\mathrm{col}(M^{ij}_r) + m - 1)_{r\in [k_{ij}]}
\end{alignat*}
where $\mathbf{a'}^{ij}, \mathbf{b'}^{ij}$ are conjugate flags to $\mathbf{a}^{ij}, \mathbf{b}^{ij}$. We also have
    \begin{align*}
        \mathsf{g}_{\theta_i \# \theta_j}^{\mathbf{a}^{ij}, \mathbf{b}^{ij}}(\mathbf{x}_m; \boldsymbol{\alpha}, \boldsymbol{\beta}) = \det \left[ e_{c(\gamma^{ij}_q) - c(\delta^{ij}_p) + 1}(\mathbf{x}_{u^{ij}_q, m};\overline{\boldsymbol{\alpha}}_{\tau(a^{ij}_q)}, \boldsymbol{\beta}_{\eta(a^{ij}_q), \eta(b^{ij}_p)} \,/\, \overline{\boldsymbol{\alpha}}_{\eta(b'^{ij}_p)}, \boldsymbol{\beta}_{\tau(a'^{ij}_q)}) \right]_{1\le p,q \le n},
    \end{align*}
    where $n \ge k_{ij}$ and $u^{ij}_q = max(1, a^{ij}_q)$ and  we denote $\tau(k) = -k + 1$ and $\eta(k) = k - m$. 
\end{proposition}
\begin{proof}
    Follows from Proposition~\ref{prop:g-enum-ribbon} using the fact that $\tau(b^{ij}_p) \le 0$, $\eta(a'^{ij}_q) \le 1, \tau(b'^{ij}_p) \le 0$, for all $p,q \in [n]$. 
\end{proof}

\begin{proposition}
Note that the flags $\mathbf{a,b}$ are also strict. We can specify the induced strict ribbon flags $\mathbf{\tilde{a}}^{ij}, \mathbf{\tilde{b}}^{ij}$ as follows:
\begin{alignat*}{2}
    &\mathbf{\tilde{a}}^{ij} = 
    (-\mathrm{col}(m^{ij}_r) + 2)_{r\in [k_{ij}]} &\qquad& 
    \mathbf{\tilde{b}}^{ij} = 
    (\mathrm{row}(M^{ij}_r) + m - 1)_{r\in [k_{ij}]} \\ 
    &\mathbf{\tilde{a}}'^{ij} = 
    (-\mathrm{row}(m^{ij}_r) + 2)_{r \in [k_{ij}]} &\qquad& 
    \mathbf{\tilde{b}}'^{ij} = 
    (\mathrm{col}(M^{ij}_r) + m - 1)_{r \in [k_{ij}] }
\end{alignat*}
We also have 
\begin{align*}
    \mathsf{g}_{\theta_i \# \theta_j}^{\mathbf{\tilde{a}}^{ij}, \mathbf{\tilde{b}}^{ij}}(\mathbf{x}_m; \boldsymbol{\alpha}, \boldsymbol{\beta}) = \det \left[ e_{c(\gamma^{ij}_q) - c(\delta^{ij}_p) + 1}(\mathbf{x}_{m};\overline{\boldsymbol{\alpha}}_{\tau(\tilde{a}^{ij}_q)}, \boldsymbol{\beta}_{\eta(\tilde{b}^{ij}_p)} \,/\, \overline{\boldsymbol{\alpha}}_{\eta(\tilde{b}'^{ij}_p)}, \boldsymbol{\beta}_{\tau(\tilde{a}'^{ij}_q)}) \right]_{1\le p,q \le n},
\end{align*}
where $n \ge k_{ij}$.
\end{proposition}
\begin{proof}
    Follows from Proposition~\ref{prop:g-enum-ribbon} using the fact that $\tau(\tilde{b}^{ij}_p) \le 0, \eta(\tilde{a}^{ij}_q) \le 1, \eta(\tilde{a}'^{ij}_q) \le 1, \tau(\tilde{b}'^{ij}_p) \le 0$, for all $p,q \in [n]$. 
\end{proof}

\begin{remark}
The latter equation for $\mathsf{g}_{\theta_i \# \theta_j}^{\mathbf{\tilde{a}}^{ij}, \mathbf{\tilde{b}}^{ij}}(\mathbf{x}_m; \boldsymbol{\alpha}, \boldsymbol{\beta})$ looks somewhat simpler than the preceding one. 
\end{remark}

We now show Jacobi-Trudi-type formulas for $g_{\lambda/\mu}(\mathbf{x}_m; \boldsymbol{\alpha}, \boldsymbol{\beta})$. Let $\mathbf{a, b}$ be some column flags for the r-shape $(n, c)$ and let:
$$
    h^{\mathbf{a, b}}_{n,c}(\mathbf{x}_m; \boldsymbol{\alpha}, \boldsymbol{\beta}) := h^{\mathbf{a, b}}_{n,c}(\mathbf{y} / \mathbf{z}) |_g = \mathsf{g}^{\mathbf{a, b}}_{(n)}(\mathbf{x}_m; \boldsymbol{\alpha}, \boldsymbol{\beta}).
$$

\begin{corollary}[Jacobi-Trudi-type formula for dual refined canonical stable Grothendieck polynomials]
    Let $\mathbf{a} = (-i + 2)_i,\ \mathbf{b} = (\lambda'_i + m - 1)_i$ be column flags, and $\mathbf{a}^{ij}, \mathbf{b}^{ij}$ be corresponding induced ribbon flags for the r-shape $(\lambda_i-\mu_j-i+j, \mu_j + 1 - j)$. We have 
    \begin{align*}
        g_{\lambda/\mu}(\mathbf{x}_m; \boldsymbol{\alpha}, \boldsymbol{\beta}) = \det \left[h^{\mathbf{a}^{ij}, \mathbf{b}^{ij}}_{\lambda_i - \mu_j - i + j,\, \mu_j + 1 - j}(\mathbf{x}_m; \boldsymbol{\alpha}, \boldsymbol{\beta})\right]_{1 \le i,j \le n}, 
    \end{align*}
    where $n \ge \lambda'_1$. 
\end{corollary}

\begin{remark}
We note that this formula is different from the known Jacobi-Trudi-type formula obtained in \cite{hwang1, hwang2}. On the other hand, we can deduce the known formula from the row flagged supersymmetric Schur function as follows. Consider the row flags  $\mathbf{a}' = (-\mu_i + 1)_i, \mathbf{b}' = (i + m - 1)_i$ for the (usual) shape $\lambda/\mu$, and the conjugate (column) flags $\mathbf{a} = (-i + 2)_i$, $\mathbf{b} = (\lambda_i + m - 1)_i$. We then have 
\begin{align*}
    \overline{\mathsf{S}}^{\mathbf{a}', \mathbf{b}'}_{\lambda/\mu}(\mathbf{y} / \mathbf{z}) |_g &= \det \left[ h_{\lambda_i-\mu_j-i+j}(\mathbf{y}_{-\mu_j + 1,\ i + m - 1}\ /\ \mathbf{z}_{-j + 2,\ \lambda_i + m - 1}) |_g  \right]_{1 \le i,j \le \ell(\lambda)} \\ 
    &= \det \left[ h_{\lambda_i-\mu_j-i+j}(\mathbf{x}_m, \overline{\boldsymbol{\alpha}}_{\mu_j} , \boldsymbol{\beta}_{i - 1}\ /\ \overline{\boldsymbol{\alpha}}_{\lambda_i - 1}, \boldsymbol{\beta}_{j - 1})  \right]_{1 \le i,j \le \ell(\lambda)} \\
    &= g_{\lambda/\mu}(\mathbf{x}_m; \boldsymbol{\alpha}, \boldsymbol{\beta}). 
\end{align*}
\end{remark}

We also show another new dual Jacobi-Trudi-type formula. Let $\mathbf{a}', \mathbf{b}'$ be some row flags for r-shape $(1^n, c)$, let 
\begin{align*}
    e^{\mathbf{a}', \mathbf{b}'}_{n,c}(\mathbf{x}_m; \boldsymbol{\alpha}, \boldsymbol{\beta}) := e^{\mathbf{a}', \mathbf{b}'}_{n, c}(\mathbf{y} / \mathbf{z}) |_g = \overline{\mathsf{S}}^{\mathbf{a}', \mathbf{b}'}_{1^n}(\mathbf{y} / \mathbf{z}) |_g. 
\end{align*}

\begin{corollary}[Dual Jacobi-Trudi-type formula for dual refined canonical stable Grothendieck polynomials]
    Let $\mathbf{a}' = (-\mu_i + 1)_i,\ \mathbf{b}' = (i + m - 1)_i$ be row flags and $\mathbf{a}'^{ij},\ \mathbf{b}'^{ij}$ the induced row ribbon flags for the r-shape $(1^{\lambda'_i-\mu'_j-i+j}, i - \lambda'_i)$. We have
    \begin{align*}
        g_{\lambda/\mu}(\mathbf{x}_m; \boldsymbol{\alpha}, \boldsymbol{\beta}) = \det \left[e^{\mathbf{a}'^{ij}, \mathbf{b}'^{ij}}_{\lambda'_i - \mu'_j - i + j,\, i - \lambda'_i}(\mathbf{x}_m; \boldsymbol{\alpha}, \boldsymbol{\beta})\right]_{i,j}^{n}, 
    \end{align*}
    where $n \ge \lambda_1$.
\end{corollary}

We now show a skew Giambelli-type formula for $g_{\lambda/\mu}(\mathbf{x}_m; \boldsymbol{\alpha}, \boldsymbol{\beta})$. Let us write the shapes in Frobenius notation: 
\begin{align*}
    \lambda = (u_1,...,u_k | v_1, ..., v_k) \qquad \mu = (c_1,...,c_{\ell}|d_1,...,d_{\ell}). 
\end{align*}
Let us define the following ribbon flags:
\begin{itemize}
    \item $\mathbf{a}^{ij} = (a^{ij}_r)_r, \mathbf{b}^{ij} = (b^{ij}_r)_r$ 
    for hooks $(u_j|v_i)$, $i,j\in [k]$
    \item $\mathbf{p}^{ij} = (p^{ij}_r)_r,\mathbf{q}^{ij} = (q^{ij}_r)_r$
    for vertical ribbons $1^{v_i - d_j}$, $i \in [k]\ , j\in [\ell]$
    \item $\mathbf{\tilde{p}}^{ij} = (\tilde{p}^{ij}_r)_r,\mathbf{\tilde{q}}^{ij} = (\tilde{q}^{ij}_r)_r$ 
    for horizontal ribbons $u_j - c_i$, $i \in [\ell]\ , j\in [k]$
\end{itemize}
\begin{corollary}[Giambelli-type formula for dual refined canonical stable Grothendieck polynomials]
    \label{cor:gen-giam-g}
    \begin{align*}
        g_{\lambda/\mu}(\mathbf{x}_m; \boldsymbol{\alpha}, \boldsymbol{\beta}) = (-1)^{\ell} \det \begin{bmatrix}
    \bigl[\mathsf{g}^{\mathbf{a}^{ij}, \mathbf{b}^{ij}}_{u_j \mid v_i}\bigr]_{\begin{subarray}{l}
            i,j \in [k] 
        \end{subarray}} 
    & \bigl[\mathsf{g}^{\mathbf{p}^{ij},\mathbf{q}^{ij}}_{1^{v_i - d_j}}\bigr]_{\begin{subarray}{l}
            i \in [k] \\ j \in [\ell] 
        \end{subarray}} \\\\
    \bigl[\mathsf{g}^{\mathbf{\tilde{p}}^{ij},\mathbf{\tilde{q}}^{ij}}_{u_j - c_i}\bigr]_{\begin{subarray}{l}
            i \in [\ell] \\ j \in [k] 
        \end{subarray}} 
    & \mathbf{0}
\end{bmatrix}, 
\end{align*}
where the common arguments $(\mathbf{x}_m; \boldsymbol{\alpha}, \boldsymbol{\beta})$ are suppressed in all three nonzero matrix blocks.
\end{corollary}

\begin{remark}
In the special case of dual stable Grothendieck polynomials $g_{\lambda}(\mathbf{x}_m)$ (for $\boldsymbol{\alpha} = 0$ and $\boldsymbol{\beta} = (1, 1, \ldots)$) of straight shape $\mu =\varnothing$, this formula gives the Giambelli-type formula obtained by Lascoux--Naruse \cite{ln}. 
Namely, let $\lambda = (u_1, \ldots, u_k \mid v_1, \ldots, v_k)$, then the following determinantal formula holds:
    \begin{align*}
        g_{\lambda}(\mathbf{x}) = \det \left[ g^{(i,j)}_{u_i \mid v_j}(\mathbf{x}) \right]_{1\le i,j \le k}, 
    \end{align*}
    where 
    \begin{align*}
        g^{(i,j)}_{u \mid v}(\mathbf{x}) = \sum_{p = 0}^{u} \sum_{q = 0}^{v} \binom{p + i - 2}{p} \binom{q + j - 2}{q} g_{u-p \mid v-q}(\mathbf{x}) 
        + \sum_{t = 2}^{\min(i,j)} \binom{u + i - t}{u} \binom{v + j - t}{v}.
    \end{align*}
    It can be shown that     
    \begin{align*}
        \mathsf{g}^{\mathbf{a}^{ij}, \mathbf{b}^{ij}}_{u_j \mid v_i}(\mathbf{x}_m; \mathbf{0}, \mathbf{1}) = g^{(j, i)}_{u_j \mid v_i}(\mathbf{x}_m),
    \end{align*}
  where $\mathbf{a}^{ij}, \mathbf{b}^{ij}$ are the induced ribbon flags for $(u_j \mid v_i)$. 
\end{remark}

\section{Concluding remarks and some open questions}\label{sec:concl}
\subsection{Equivalence between $g$-tableaux and marked reverse plane partitions}
    It is known that dual refined canonical stable Grothendieck polynomials $g_{\lambda/\mu}(\mathbf{x}; \boldsymbol{\alpha}, \boldsymbol{\beta})$ have a tableau formula via {\it marked reverse plane partitions}, see \cite{hwang2}. Here we showed another combinatorial formula for these functions using $g$-tableaux (as defined in previous section). It would be interesting to establish a direct bijective argument to show equivalence of two tableaux formulas for dual refined canonical stable Grothendieck polynomials. 

\subsection{Equivalent skew shapes}
Hamel--Goulden formulas are useful for studying equivalent skew shapes for Schur functions, i.e. when $s_{\kappa}= s_{\eta}$ for skew shapes $\kappa, \eta$, see  \cite{reiner}. Some equivalent skew shapes for stable Grothendieck polynomials were studied in \cite{alwaise, aan, hmps}. It would be interesting  to apply our Hamel--Goulden-type formulas for such problems on dual stable Grothendieck polynomials. 


\subsection{Hamel-Goulden formulas for stable Grothendieck polynomials}
It is natural to ask to develop analogous ribbon decomposition formulas for refined canonical stable Grothendieck polynomials $G_{\lambda/\mu}(\mathbf{x}; \boldsymbol{\alpha}, \boldsymbol{\beta})$ which form a family dual to $g_{\lambda/\mu}(\mathbf{x}; \boldsymbol{\alpha}, \boldsymbol{\beta})$ (via the Hall inner product, for straight shapes) studied in \cite{hwang1, hwang2}. 
We resolve this in \cite{ay26}.

\section*{Acknowledgements}
We are grateful to Askar S. Dzhumadil'daev and Alejandro Morales for helpful conversations. 
We are also grateful to the referee for useful comments. 


\end{document}